\documentclass[a4paper, 12pt]{article}

\usepackage[utf8]{inputenc}
\usepackage{amsmath}
\usepackage{amsthm}
\usepackage{amsfonts}
\usepackage{amssymb}
\usepackage{amstext}
\usepackage{mathrsfs}
\usepackage{mathpazo}
\usepackage{yhmath}
\usepackage{dsfont}
\usepackage{graphicx}
\usepackage{color}
\usepackage[colorlinks]{hyperref}
\usepackage{epigraph}
\usepackage{fullpage}

\usepackage{tikz}

\usetikzlibrary{arrows,shapes,matrix}

\graphicspath{{Figures/}{.}}%
\title{ The Circular Unitary Ensemble and the Riemann zeta function: the microscopic landscape and a new approach to ratios}
\author{Reda \textsc{Chhaibi}        \footnote{\texttt{reda.chhaibi@math.uzh.ch}}\;\;\;\;\;\;\;\;\;
        Joseph \textsc{Najnudel}     \footnote{\texttt{joseph.najnudel@math.univ-toulouse.fr}}\\
        Ashkan \textsc{Nikeghbali}   \footnote{\texttt{ashkan.nikeghbali@math.uzh.ch}}
        } %\footnotemark

\allowdisplaybreaks[4]

\DeclareMathOperator{\sgn}{sgn}
\DeclareMathOperator{\Var}{Var}

\begin{document}

\def\half{\frac{1}{2}}

\def\l{\lambda}
\def\t{\theta}
\def\T{\Theta}
\def\m{\mu}
\def\a{\alpha}
\def\b{\beta}
\def\g{\gamma}
\def\o{\omega}
\def\p{\varphi}
\def\D{\Delta}
\def\O{\Omega}

%Ordinals
\def\N{{\mathbb N}}
\def\Z{{\mathbb Z}}
\def\Q{{\mathbb Q}}
\def\R{{\mathbb R}}
\def\C{{\mathbb C}}

%Probability
\def\P{{\mathbb P}}
\def\E{{\mathbb E}}

%Caligraphed letters
\def\Ac{{\mathcal A}}
\def\Bc{{\mathcal B}}
\def\Cc{{\mathcal C}}
\def\Dc{{\mathcal D}}
\def\Ec{{\mathcal E}}
\def\Fc{{\mathcal F}}
\def\Gc{{\mathcal G}}
\def\Hc{{\mathcal H}}
\def\Ic{{\mathcal I}}
\def\Kc{{\mathbb K}}
\def\Lc{{\mathcal L}}
\def\Oc{{\mathcal O}}
\def\Pc{{\mathcal P}}
\def\Qc{{\mathcal Q}}
\def\Rc{{\mathcal R}}
\def\Sc{{\mathcal S}}
\def\Tc{{\mathcal T}}
\def\Uc{{\mathcal U}}
\def\Zc{{\mathcal Z}}

%Fraktur
\def\afrak{{\mathfrak a}}
\def\bfrak{{\mathfrak b}}
\def\gfrak{{\mathfrak g}}
\def\hfrak{{\mathfrak h}}
\def\kfrak{{\mathfrak k}}
\def\nfrak{{\mathfrak n}}
\def\pfrak{{\mathfrak p}}
\def\slfrak{{\mathfrak sl}}

\maketitle

\setlength{\footskip}{2cm}

\numberwithin{equation}{section}
\newtheorem{thm}{Theorem}[section]
\newtheorem{proposition}[thm]{Proposition}
\newtheorem{corollary}[thm]{Corollary}
\newtheorem{question}[thm]{Question}
\newtheorem{conjecture}[thm]{Conjecture}
\newtheorem{definition}[thm]{Definition}
\newtheorem{example}[thm]{Example}
\newtheorem{lemma}[thm]{Lemma}
\newtheorem{properties}[thm]{Properties}
\newtheorem{property}[thm]{Property}
\newtheorem{rmk}[thm]{Remark}
\newtheorem{hypothesis}[thm]{Hypothesis}

\newcommand{\Card}[1]{\textup{Card($#1$)} \xspace}

\begin{abstract}
We show in this paper that after proper scalings, the characteristic polynomial of a random unitary matrix converges almost surely to a random analytic function whose  zeros, which are on the real line, form a determinantal point process with sine kernel. Our scaling is performed at the so-called "microscopic" level, that is we consider the characteristic polynomial at points which are of order $1/n$ distant. We prove this in the framework of virtual isometries to circumvent the fact that the rescaled characteristic polynomial does  not even have a moment of order one, hence making the classical techniques of random matrix theory difficult to apply.  The strong convergence results in this setup provide us with a new approach to ratios: we are able to solve  open problems about the limiting distribution of ratios of characteristic polynomials evaluated at points of the form $\exp(2 i \pi \alpha/n)$ and related objects (such as the logarithmic derivative). We also explicitly describe the dependence relation 
for the logarithm of the characteristic polynomial evaluated at several points  on the microscopic scale. On the number theory side, inspired by the Keating-Snaith philosophy, we conjecture some new limit theorems for the value distribution of the Riemann zeta function on the critical line at the stochastic process level.
\end{abstract}
\pagebreak
% --------------------------------------------------------------------
% Table of contents
\tableofcontents

% --------------------------------------------------------------------
\newpage
\section{Introduction}

A major breakthrough in the so-called random matrix approach in number theory is  the seminal paper of Keating and Snaith \cite{bib:KS}, where they conjecture that the characteristic polynomial of a random unitary matrix, restricted to the unit circle, is a good and accurate model to predict the value distribution of the Riemann zeta function on the critical line.  In particular, using this philosophy, they were able to conjecture the exact asymptotics of the moments of the Riemann zeta function, a result which was considered to be out of reach with classical tools from analytic number theory. One simple and naive explanation for the success of the characteristic polynomial as a random model to the Riemann zeta function comes from Montgomery's conjecture that asserts that the zeros of the Riemann zeta function on the critical line (after rescaling) statistically behave like the eigenangles (after rescaling) of large random unitary matrices. 
Moreover the limiting point process obtained from the eigenvalues is the determinantal sine kernel point process. A natural question which then naturally arose in the community was the existence of a random analytic function with zeros which are this sine kernel process and which would be obtained as a limiting object from characteristic polynomials. As we shall see below, the sequence of characteristic polynomials of random unitary matrices of growing dimensions does not converge. We shall nonetheless prove that after a proper rescaling in "time" (the characteristic polynomial can be viewed as a stochastic process with parameter $z\in\C$, and we shall consider the characteristic polynomial at the scale $z/n$) and space (that is we normalize with the value of the characteristic polynomial at $1$), this sequence converges locally uniformly on compact subsets of the complex plane to a random analytic function with the desired property.

 To be more specific, the convergence discussed above  will be proved to occur almost surely, thanks to the use of virtual isometries introduced in \cite{bib:BNN}. The basic idea behind virtual isometries is that of coupling the different dimensions of the unitary groups $U(n)$ together in such a way that marginal distribution on each $U(n)$, for fixed $n$, is the Haar measure. Along with some of the fine estimates on the eigenvalues from \cite{bib:MNN} and some new deep estimates  related to the logarithmic derivative and ratios of our limiting random analytic function, we establish almost sure convergence. This strong convergence will in turn imply the weak convergence of the same objects. But since our rescaled characteristic polynomials do not even have a moment of order one, proving the weak convergence as stated in Theorem \ref{thm:maininlaw} with classical methods does not seem to be an easy task. 
 Moreover, this approach based on almost sure convergence will somehow miraculously provide us with straightforward proofs to some known difficult problems on ratios, logarithms and logarithmic derivatives  of characteristic polynomials of random unitary matrices.  The solutions to these problems always involve the limiting random analytic function discussed above, so that one might think of it as a canonical object in random matrix theory. More generally this approach to ratios of characteristic polynomials brings new insight  not only in random matrix theory but also in number theory where we propose a new point of view using random analytic functions to make predictions for the value distribution of the Riemann zeta function. More precisely  the contributions of this paper can be summarized as follows:
\begin{enumerate}
\item \textbf{Ratios and correlations on the microscopic scale:} ratios of characteristic polynomials  are relevant objects which have been  extensively studied in recent years, for instance  in relation with quantum chaotic systems or analytic number theory (see \cite{bib:FyStr}, \cite{bib:StrFy}, \cite{bib:BoStr}, \cite{bib:CoSn}, \cite{bib:CFZ08}, \cite{bib:BG}), using a wide range of techniques (e.g. classical analysis, representation theory or supersymmetry methods).  It has been an open question to characterize the limiting object   obtained from ratios of characteristic polynomials evaluated at points of the form $\frac{\alpha}{n}$ for $\alpha\in\C$, when $n\to\infty$.\footnote{This  question was asked to A.N. by Alexei Borodin in a private communication.}  It was more or less observed that the expectation of such ratios converges but the  limiting object was not known. In this paper, we shall give an explicit formula for the limit of the ratios and prove that the expectation of the ratios of the 
characteristic polynomial converges (
locally uniformly on compact subsets of $\mathbb C-\mathbb R$) to the expectation of the corresponding limiting object. We shall also derive the limiting object 
for the rescaled logarithmic derivative of the characteristic polynomial at the microscopic scale and give two alternative formulas to compute its joint moments evaluated at several points. 

We shall also give a description of the dependence between the log of the characteristic polynomial evaluated at various points distant of $\frac{\alpha}{n}$. To the best of our knowledge this problem was not solved before. 

\item \textbf{Linear statistics:} we shall see that the logarithmic derivative of our random analytic function is related to linear statistics for the sine kernel point process  for test functions in  $H^{1/2}$ and we shall prove a convergence result, on the mesoscopic scale, to a holomorphic  Gaussian field. It should be noted that such objects, for more general point 
processes, have been recently studied by Aizenmann and Warzel in \cite{bib:AiWa} and our results can be viewed as a complement to the results obtained in there in the special case of the sine kernel determinantal point process.  

Besides we also prove, using a general result from \cite{bib:MNN}, a weak convergence for linear statistics on the microscopic scale for test functions which are only assumed to be integrable and prove that linear statistics on this scale have a natural representation in terms of our limiting random analytic function.  

\item \textbf{Value distribution of the Riemann zeta function on the critical line:} on the number theory side, we shall state some conjectures relating our limiting  random analytic function to the Riemann zeta function. We conjecture that our scaling amounts to eliminating the contribution of prime numbers to keep only those of the Riemann zeros and thus obtain a limiting object whose zeros form a sine kernel determinantal point process, in agreement with the GUE conjecture and the fact that short range statistics conform the GUE predictions (unlike long-range statistics which are better described with primes). We shall also relate the logarithmic derivative of our limiting function  to recent conjectures of Goldston, Gonek and Montgomery \cite{bib:GGM01} on the second moment of the logarithmic derivative of the Riemann zeta function. We shall be able to provide a very general conjecture on the logarithmic derivative of the Riemann zeta function in agreement with the predictions obtained in \cite{bib:GGM01}
 and in \cite{bib:FGLL}. Moreover the general formula for the expectation of the ratios provides simpler corresponding conjectures for the Riemann zeta function. 
The point of view we take is new in the literature on the random matrix approach in number theory: indeed we  suggest to model ratios (and not only their expectations) by some natural  random analytic functions.

\end{enumerate}

In the sequel, we introduce the main objects and notation and state our main theorem.

\subsection{The characteristic polynomial of random unitary matrices and the number theory connections}

It is a well known fact in the theory of random unitary matrices that, when properly rescaled, the eigenvalues converge to a determinantal point process with sine kernel:

\begin{proposition}\label{Laplacefunctionals}
Let $E_n$ denote the set of eigenvalues taken in $(-\pi,\pi]$ and multiplied by $n/2\pi$ of a random unitary matrix of size $n$ following the Haar measure. Let us also define, for $y\neq y^{'}$,
$$K(y,y^{'})=\dfrac{\sin[\pi(y^{'}-y)]}{\pi (y^{'}-y)}$$and $$K(y,y)=1.$$ Let $E_{\infty}$ be a determinantal sine-kernel process, i.e. a point process such that for all $r \in \{1, \dots, n\}$, and for all 
Borel measurable and bounded functions $F$  with compact support from $\R^r$ to $\R$,

$$\mathbb{E} \left( \sum_{x_1 \neq \dots
\neq x_r \in E_{\infty}} F(x_1, \dots, x_r) \right) = \int_{\R^r} 
F(y_1, \dots, y_r) \rho_r(y_1, \dots,
y_r) dy_1 \dots dy_r,$$
where
$$\rho_r(y_1, \dots, y_r) 
= \operatorname{det} ( (K (y_j, y_k))_{1 \leq j, k \leq r} ).$$
Then, the point process $E_n$ converges to $E_{\infty}$ in the following sense: 
for all Borel measurable bounded functions $f$ with compact support from $\R$ to
$\R$, 
$$\sum_{x \in E_n} f(x) 
\underset{n \rightarrow \infty}{\longrightarrow}
\sum_{x \in E_{\infty}} f(x),$$
where the convergence above holds in law. 

\end{proposition}
We now recall basic facts about the Riemann zeta function (the reader can find more details in classical textbooks such as \cite{bib:Titchmarsh}).  The Riemann zeta function is defined, for $\Re(s)>1$,  by
$$\zeta(s)=\sum_{n=1}^\infty \dfrac{1}{n^s}.$$It has a meromorphic continuation to the whole complex plane with a single pole at $1$. It also satisfies a functional equation which we can be stated as follows: 
$$\pi^{-s/2}\Gamma(s/2)\zeta(s)=\pi^{(s-1)/2}\Gamma((1-s)/2)\zeta(1-s),$$and $$\zeta(s)=\chi(s)\zeta(1-s),$$where
$$\chi(1-s)=\chi(s)^{-1}=2(2\pi)^{-s}\Gamma(s)\cos(\pi s/2).$$The non-trivial zeros of the zeta function are denoted by $\rho=\sigma+i t$, where $0<\sigma<1$. The Riemann hypothesis is the assertion that all non trivial zeros satisfy $\sigma=1/2$ and hence all non trivial zeros are of the form $\rho=1/2+it$, with $t\in\R$. If we assume the Riemann hypothesis, then the zeros come in conjugate pairs and  we note the zeros in the upper half-plane as $1/2+i \gamma_j$, where $0<\gamma_1\leq\gamma_2\leq\cdots$. One can count the number of such zeros up to some height $T$:
$$N(T):=\#\{j;\;0\leq\gamma_j\leq T\}=\dfrac{T}{2\pi}\log \dfrac{T}{2\pi e}+O(\log T).$$

The connection to random matrix theory was conjectured by Montgomery in \cite{bib:Montgomery}: it is conjectured that the rescaled zeros of the zeta function $\tilde\gamma:=\gamma/(2\pi)\log\gamma$ (this rescaling is done in order to obtain an average spacing of order 1) satisfy the same limit theorem as the one given in Proposition \ref{Laplacefunctionals} for the rescaled eigenvalues of random unitary matrices (in fact the conjecture was initially stated for the pair correlation and then extended to all correlations by Rudnik and Sarnak in \cite{bib:RudSar}; see the recent paper of Conrey and Snaith \cite{bib:CoSn14} for a detailed account and new methods).

Another major insight came with the work of Keating and Snaith (\cite{bib:KS}) where they use the characteristic polynomial of random unitary matrices to model the value distribution of the Riemann zeta function on the critical line (i.e. the family $\{\zeta(1/2+it),\;t\geq0\}$) to make spectacular predictions on the moments of the Riemann zeta function.  In particular, in \cite{bib:KS} they computed the moments of the characteristic polynomial of a random unitary 
matrix following the Haar measure. They deduced that the characteristic polynomial asymptotically behaves like a 
log-normal distributed random variable when the dimension $n$ goes to infinity: more precisely, 
its logarithm, divided by $\sqrt{\log n}$, tends to a complex Gaussian random variable 
$Z$ such that $\E[Z] = \E [Z^2] = 0$ and $\E [|Z|^2]= 1$. 
This result has been generalized in Hughes, Keating and O'Connell \cite{bib:HKO}, where the authors proved the asymptotic
independence of the characteristic polynomial taken at different fixed points. A question which then naturally arises  concerns the behavior of the characteristic polynomial at points which vary with the 
dimension and which are sufficiently close to each other in order to avoid asymptotic independence. 
The scale we consider in the present paper is the average spacing of the eigenangles of a unitary
matrix in dimension $n$, i.e. $2 \pi/n$. 
More precisely, let $(U_n)_{n \geq 1}$ be a sequence of matrices, $U_n$ being Haar-distributed in $U(n)$, and let 
$Z_n$ be the characteristic polynomial of $U_n$: 
\begin{eqnarray}
\label{eq:def_Z_n}
Z_n(X) = \operatorname{det} \left(  \textrm{Id}- U_n^{-1}X \right)=\operatorname{det} \left(  \textrm{Id}- U_n^{*}X \right).
\end{eqnarray}
For a given $z \in \C$, we consider the value of $Z_n$ at the two points $1$ and $e^{2 i z \pi/n}$, 
whose distance is equivalent to $2 \pi |z|/n$ when $n$ goes to infinity. We know that the law of $Z_n(1)$ can be approximated by the exponential of a gaussian variable of variance $\log n$, so it does not converge when $n$ goes to infinity. The same is true for $Z_n(e^{2 i z \pi/n})$. In order to obtain a convergence in law, it is then natural to consider the ratio $Z_n(e^{2i z \pi/n})/ Z_n(1)$, which has order of magnitude $1$ and which is well-defined as soon as $1$ is not an eigenvalue of $U_n$, an event occurring almost surely. 

If we consider all the values of $z$ together, we obtain a random entire function $\xi_n$, defined by 
\begin{eqnarray}
\label{eq:def_xi_n}
\xi_n (z) = \frac{Z_n(e^{2i z \pi/n})}{Z_n(1)}.
\end{eqnarray}
Because $(U_n)_{n \geq 1}$ is a sequence of unitary matrices, the following functional equation holds:
\begin{eqnarray}
\label{eq:functional_eq_xi_n}
\overline{ \xi_n (z) } = e^{-i 2 \pi \bar{z}} \xi_n\left( \bar{z} \right)
\end{eqnarray}

We will prove that this function has a limiting distribution when $n$ goes to infinity. More precisely, one of the main results of this article is the following: 
\begin{thm} \label{thm:maininlaw}
 In the space of continuous functions from $\C$ to $\C$, endowed with the topology of uniform 
 convergence on compact sets, the random entire function $\xi_n$ converges in law to a limiting 
 entire function $\xi_{\infty}$. The zeros of $\xi_{\infty}$ are all real and form a determinantal sine-kernel point 
 process,
 i.e. for all $r \geq 1$, the $r$-point correlation function $\rho_r$ corresponding to this point process is 
 given, for all $x_1, \dots, x_r \in \R$, by 
 $$\rho_r (x_1, \dots, x_r) = \operatorname{det} \left( \frac{\sin [\pi(x_j - x_k)]}{ \pi(x_j - x_k)} \right)_{1 \leq 
 j,k  \leq r}. $$
 \end{thm}
 Notice that this theorem cannot be straightforwardly deduced from the convergence of the zeroes of $\xi_n$ to a sine-kernel process. Afterall, the convergence of point processes is local in nature while the random analytic function $\xi_\infty$ is certainly not local, being a infinite product over all zeroes.
 
 Taking a finite number of points $z_1, \dots, z_p \in \C$, we see in particular that the joint law of 
 the mutual ratios of $Z_n (e^{2 i \pi z_1/n}), \dots, Z_n (e^{2 i \pi z_p/n})$ converges when $n$ goes to infinity. Now one can hope to gain new insights on the behaviour of ratios of characteristic polynomials on this microscopic scale. More precisely, let us define:
 
\begin{equation}\label{eq::ratios}
R(\alpha_1,\cdots,\alpha_r;\beta_1,\cdots,\beta_r):=\dfrac{Z_n(e^{2i \alpha_1 \pi/n})\cdots Z_n(e^{2i \alpha_r \pi/n})}{Z_n(e^{2i \beta_1 \pi/n})\cdots Z_n(e^{2i \beta_r \pi/n})},
\end{equation}where $r\in\N$ and $\alpha_j\in\C$, $\beta_j\in\C$, for all $1\leq j\leq r$.
Ratios such as \eqref{eq::ratios}, on the macroscopic scale (i.e. without the $1/n$ in the arguments) have been
extensively studied in random matrix theory for different random matrix ensembles, e.g. the GUE
by Fyodorov and Strahov in \cite{bib:FyStr} and \cite{bib:StrFy}, the COE and the CSE by Borodin and Strahov in \cite{bib:BoStr} or in the CUE case by Conrey, Farmer and Zirnbauer (\cite{bib:CFZ08}), by Conrey and Snaith (\cite{bib:CoSn}) or Bump and Gamburd (\cite{bib:BG}). In all cases, one considers the expectation of the ratios and the $n$-limit of this expression. But finding the $n$-limit of $R(\alpha_1,\cdots,\alpha_r;\beta_1,\cdots,\beta_r)$ had remained an open problem. In fact, we shall prove a strong version (i.e. with almost sure convergence) of Theorem \ref{thm:maininlaw} which will immediately yield the $n$-limit of $R(\alpha_1,\cdots,\alpha_r;\beta_1,\cdots,\beta_r)$ as well as some central limit theorem for the vector $(\log Z_n (e^{2 i \pi z/n}), \log Z_n (1))$. The almost sure convergence is established through the machinery of virtual isometries that we
  recall in the next paragraph.

\subsection{Virtual isometries and almost sure convergence}
 
In order to prove Theorem \ref{thm:maininlaw}, we will define the sequence $(U_n)_{n \geq 1}$ of unitary matrices in a common probability space, with a coupling chosen in such a way that an almost sure convergence occurs. An interest of this method is that it is more convenient to deal with pointwise convergence than with convergence in law when we work on a functional space. Moreover, the coupling gives a powerful way to keep track of the sequence $(\xi_n)_{n \geq 1}$ of holomorphic functions, and a deterministic link between this sequence and the limiting function $\xi_{\infty}$.

Besides it is important to stress that the moments method, which is a classical technique in random matrix theory, seems tedious to implement at best. Indeed the random function at hand $\xi_n$ does not have any integer moment when evaluated on the circle, which makes the use of the formulas on moments of ratios in \cite{bib:BG} and \cite{bib:CFZ08} difficult to use. For example, in Theorem 3 of the article \cite{bib:BG}, one clearly sees the divergence of moments of ratios, as the evaluation points get close to $1$.

The coupling we consider here corresponds to the notion of {\it virtual isometries}, as defined by Bourgade, Najnudel and Nikeghbali in \cite{bib:BNN}. The sequence $(U_n)_{n \geq 1}$ can be constructed in the following way: 

\begin{enumerate}
\item One considers a sequence $(x_n)_{x \geq 1}$ of independent random vectors, $x_n$ being uniform on the unit sphere of 
$\C^n$. 
\item Almost surely, for all $n \geq 1$, $x_n$ is different from the last basis vector $e_n$ of $\C^n$, which implies that there exists a unique $R_n \in U(n)$ such that $R_n(e_n) = x_n$ and $R_n - I_n$ has rank one. 
\item We define $(U_n)_{n \geq 1}$ by induction as follows: $U_1 = x_1$ and for all $n \geq 2$, 
$$U_n = R_n \left( \begin{array}{cc}
U_{n-1} & 0\\
0 & 1 \end{array} \right).$$
\end{enumerate}

It has already been proven in \cite{bib:BHNY} that with this construction, $U_n$ follows, for all $n \geq 1$, the Haar measure on $U(n)$. From now on, we always assume that the sequence $(U_n)_{n \geq 1}$ is defined with this coupling.

For each value of $n$, let $\lambda_1^{(n)}, \dots, \lambda_n^{(n)}$ be the eigenvalues of $U_n$, ordered counterclockwise, starting from $1$: they are almost surely pairwise distinct and different from $1$. If $1 \leq k \leq n$, we denote by $\theta_k^{(n)}$ the argument of $\lambda_k^{(n)}$, taken in the interval $(0, 2\pi)$: $\theta_k^{(n)}$ is the $k$-th strictly positive eigenangle of $U_n$. If we consider all the eigenangles of $U_n$, taken not only in $(0, 2 \pi)$ but in the whole real line, we get a $(2 \pi)$-periodic set with $n$ points in each period.  If the eigenangles are indexed increasingly by $\mathbb{Z}$, we obtain a sequence
 $$ \dots < \theta_{-1}^{(n)} < \theta_{0}^{(n)} < 0 < \theta_{1}^{(n)} < \theta_{2}^{(n)} < \dots, $$ 
for which $\theta_{k+n}^{(n)} = \theta_k^{(n)} + 2 \pi$ for all $k \in \mathbb{Z}$. 

It is also convenient to extend the sequence of eigenvalues as a $n$-periodic sequence indexed by $\mathbb{Z}$, in such a way that for all $k \in \Z$, 
$$ \lambda_k^{(n)} = \exp\left( i\theta_k^{(n)} \right). $$

With the notation above, the following holds:
\begin{thm}[Theorem 7.3 in \cite{bib:MNN}]
\label{thm:as_cv_sine}
Almost surely, the point process 
$$\left( y_k^{(n)} := \frac{n}{2 \pi} \theta_{k}^{(n)} , \ k \in \Z \right)$$
converges pointwise to a determinantal sine-kernel point process $\left( y_k , \ k \in \Z \right)$. And 
moreover, almost surely, the following estimate holds for all $\varepsilon  > 0$:
$$ \forall  k \in [-n^{\frac{1}{4}},n^{\frac{1}{4}}],  \ y_k^{(n)} = y_k + O_\varepsilon\left( (1+k^2) n^{-\frac{1}{3}+\varepsilon} \right)  $$
\end{thm}
\begin{rmk}
 The implied constant in $O_\varepsilon$ is random: more precisely, it may depend on the sequence $(U_m)_{m \geq 1}$ and 
 on $\varepsilon$. However, it does not depend on $k$ and $n$. 
\end{rmk}
We are now able to state  the main convergence result of the paper.
\begin{thm} \label{thm:main}
Almost surely and uniformly on compact subsets of $\C$, we have the convergence:
$$ \xi_n\left( z \right) \stackrel{ n \rightarrow \infty}{\longrightarrow} \xi_\infty(z) := e^{i \pi z} 
\prod_{k \in \Z}\left( 1 - \frac{z}{y_k}\right)$$
Here, the infinite product is not absolutely convergent. It has to be understood as the limit of the following 
product, obtained by regrouping the factors  two by two: 
$$\left( 1 - \frac{z}{y_0} \right) \prod_{k \geq 1} \left[ \left( 1 - \frac{z}{y_k} \right)
\left( 1 - \frac{z}{y_{-k}} \right) \right],$$
which is absolutely convergent. 
\end{thm}

This theorem immediately implies Theorem \ref{thm:maininlaw}, provided that $\xi_{\infty}$ is entire and that the zeros of $\xi_{\infty}$ are exactly given by the sequence $(y_k)_{k \in \Z}$. This first point is a direct consequence of the fact that $\xi_{\infty}$ is the uniform limit on compact sets of the sequence of entire functions $(\xi_n)_{n \geq 1}$, and the second point is a consequence of the fact that the $k$-th factor of the absolutely convergent product above vanishes at $y_k$ and $y_{-k}$ and only at these points.
 
Now, thanks to the almost sure convergence, we can state the following corollaries.
 
 \begin{corollary}\label{ratiosconv}
Let $r\in\N$ and $\alpha_j\in\C$, $\beta_j\in\C$ but $\beta_j\notin (y_k)_{k\in\Z}$, for all $1\leq j\leq r$. Then the following convergence holds a.s. as $n\to\infty$:
 $$R(\alpha_1,\cdots,\alpha_r;\beta_1,\cdots,\beta_r):=\dfrac{Z_n(e^{2i \alpha_1 \pi/n})\cdots Z_n(e^{2i \alpha_r \pi/n})}{Z_n(e^{2i \beta_1 \pi/n})\cdots Z_n(e^{2i \beta_r \pi/n})}\to\dfrac{\xi_\infty(\alpha_1)\cdots\xi_\infty(\alpha_r)}{\xi_\infty(\beta_1)\cdots\xi_\infty(\beta_r)}$$
 \end{corollary}
 
 In Section \ref{lenotre} we shall establish that the above convergence also holds in expectation locally uniformly. Since the convergence in Theorem \ref{thm:main} holds almost surely in the space of holomorphic functions, we immediately obtain:
\begin{corollary} \label{laderivee}
We have a.s. uniformly on compact sets, that as $n\to\infty$:
$$\dfrac{2 i \pi}{n}\dfrac{Z_n^{'} (e^{2 i \pi z/n})}{Z_n(1)}\to \xi_\infty^{'}.$$
\end{corollary}
The next corollary involves the logarithm of 
$Z_n$.  We provide a simple proof thanks to our functional convergence result. The determination of this logarithm is the only one such that $\log Z_n$ vanishes at $0$ (recall that $Z_n(0) = 1$), and which is continuous on the following maximal simply connected domain %(see figure \ref{fig:domain_dc}):
$$ \Dc := \C \backslash \left\{ r e^{i \theta_k^{(n)}} \ | \ k \in \Z, \ r \geq 1 \right\}.$$

Note that for all $z \in \Dc$, we have:
$$ \log Z_n(z) = \sum_{k=1}^n \log\left( 1 - \frac{z}{\lambda_k^{(n)}} \right),$$
where the principal branch of the logarithm is considered. 
 
\begin{corollary}
Let $r \in \N$. and fix $\left( z_1, z_2, \dots, z_r \right) \in \C^r$. The following convergence holds in law as $n\to\infty$
$$\left(\dfrac{\log Z_n (e^{2 i \pi z_1/n})}{\sqrt{\half \log n}}, \dots, \dfrac{\log Z_n (e^{2 i \pi z_r/n})}{\sqrt{\half \log n}} \right)
  \to
  \mathcal{N} e $$
 where $\mathcal N$ stands for a standard complex Gaussian random variable, and $e$ denotes the vector $\left( 1, 1, \dots, 1 \right) \in \C^r$.
\end{corollary}
A more general version of the corollary, and a similar result relative to the behavior of the 
Riemann zeta function near the critical line, have been obtained 
by Bourgade in \cite{bib:Bou10} (see Theorems 1.1. and 1.4. there). 
\begin{proof}
Let $z$ be a complex number among $\left( z_1, z_2, \dots, z_r \right)$. One checks that 
 $$\log Z_n (e^{2 i \pi z/n}) - \log Z_n (1) = \log \xi_n(z),$$
 where $\log \xi_n$ is the unique determination of the logarithm, vanishing at 
 $0$, and continuous in the domain 
 $$ \Dc'_n := \C \backslash \left\{ y_k^{(n)} - iu | k \in \Z, u \geq 0 \right\}.$$
 Let $\log \xi_{\infty}$ be the similar determination of the logarithm of $\xi_{\infty}$. 
Let us fix $z \in \C$, $t > 0$ such that $z + it$ has strictly positive imaginary part, and let $L$ be the line consisting of the two segments from $0$ to $z + it$ and from $z + it$ to $z$. We also recall that 
the random functions $(\xi_n)_{n \geq 1}$ and 
$\xi_{\infty}$ are coupled in such a way that 
almost surely, $\xi_{n}$ tends to $\xi_{\infty}$ uniformly on compact sets of $\C$. 
Almost surely, for $n$ large enough, $0$ and $\Re z$ are not zeros of $\xi_n$ and one deduces that 
$L$ is included in $\Dc'_n$. Hence, 
$$\log \xi_n(z) = \int_L \frac{\xi'_n(s)}{\xi_n(s)} ds$$
and 
$$\log \xi_{\infty}(z) = \int_L \frac{\xi'_{\infty}(s)}{\xi_{\infty}(s)} ds.$$
Now, $(\xi_n, \xi'_n)$ tends to $(\xi_{\infty}, \xi'_{\infty})$ uniformly on $L$. Moreover, $\xi_{\infty}$ is continuous 
and nonvanishing on the compact set $L$, which implies that $|\xi_{\infty}|$, and then 
$|\xi_n|$ for $n$ large enough, are bounded away from zero on $L$. Hence, 
$\xi'_n/\xi_n$ tends to $\xi'_{\infty}/\xi_{\infty}$ uniformly on $L$, and then $\log \xi_{n}(z) $ tends to 
$\log \xi_{\infty}(z)$. 
We deduce that 
$$\dfrac{\log Z_n (e^{2 i \pi z/n})}{\sqrt{(1/2)\log n}} - \dfrac{\log Z_n (1)}{\sqrt{(1/2) \log n}} \underset{n \rightarrow 
\infty}{\longrightarrow} 0$$ 
almost surely with the coupling above, and then in probability. Since we already know that the second term of the difference tends in law to $\mathcal{N}$, we are done. 
\end{proof}

We can use the a.s. convergence of $\log \xi_n$ to  $\log \xi_\infty$ that we established above to give a simple proof for the convergence of the number of points in an arc at the microscopic scale:

\begin{corollary}
For all $z \in \mathbb{R}$, the number of eigenvalues of $U_n$ in the arc between $1$ and $e^{2i \pi z/n}$ tends
in law to the number of points of a determinantal sine-kernel process in an interval of length $|z|$. 
\end{corollary}

\begin{proof}
	For $z \in \mathbb{R}$, the number of eigenvalues of $U_n$ in the arc between $1$ and $e^{2i \pi z/n}$, multiplied by the 
	sign of $z$, is equal to (see Corollary \ref{corollary:kykn})
	$$ N_n (z) := z - \frac{1}{\pi} \Im \log (\xi_n (z)),$$
	if $1$ and $e^{2i\pi z/n}$ are not eigenvalues of $U_n$ (which holds almost surely). Now 
	we know  that $ \log (\xi_n)$ tends a.s. to $\log (\xi_{\infty})$
	when $n$ goes to infinity. The proof of the corollary is completed once one notes that  
	$$N_{\infty} (z) := z - \frac{1}{\pi} \Im \log (\xi_{\infty} (z)),$$
	has the same absolute value as the number of zeros of $\xi_{\infty}$ between $0$ and $z$.
\end{proof}

\begin{rmk}
	We shall prove in Section \ref{sec::statlin} more general results on convergence of linear statistics at the microscopic scale.
\end{rmk}

\begin{rmk}
From Corollary \ref{laderivee} one can also deduce a joint central limit theorem for the log of the derivative of the characteristic polynomial at $e^{2 i \pi z/n}$ and the log of the characteristic polynomial at $1$.
\end{rmk}

%Since the convergence in Theorem \ref{thm:main} holds almost surely in the space of holomorphic functions, we immediately obtain:
%\begin{corollary} Attention il faudrait deriver l'exponentielle k fois ce qui donne des termes supplementaires (derivation d'une fonction composee...)
%Let $k\in\N$. Then the following convergence %holds a.s. as $n\to\infty$:
%$$\dfrac{1}{n^k}\dfrac{Z_n^{(k)} (e^{2 i \pi %z/n})}{Z_n(1)}\to \xi_\infty^{(k)}.$$
%\end{corollary}
%\begin{corollary} Preuve du corrolaire? Je ne vois pas pourquoi c'est immediat. 
%Let $z\in\C$. The following convergence holds in law as $n\to\infty$
%$$\left(\dfrac{\log Z_n^{(k)} (e^{2 i \pi z/n})-k\log n}{\sqrt{1/2 \log n}}, \dfrac{\log Z_n (1)}{\sqrt{1/2 \log n}}\right)\to(\mathcal N, \mathcal N)$$where $\mathcal N$ stands for a standard complex Gaussian random variable.
%\end{corollary}

We can eventually easily derive the limiting random analytic function for the logarithmic derivative:

\begin{corollary}
We have almost surely, for all $z \notin\{y_k,\; k\in\mathbb Z\}$ :
$$\dfrac{2 i \pi}{n} \dfrac{Z_n^{'} (e^{2 i \pi z/n})}{Z_n (e^{2 i \pi z/n})} \underset{n \rightarrow \infty}{\longrightarrow}
\frac{ \xi_{\infty}'(z) }{ \xi_{\infty}(z)},$$
where 
$$ \frac{ \xi_{\infty}'(z) }{ \xi_{\infty} (z)} = i \pi + \sum_{k \in \Z} \frac{1}{z-y_k}
 := i \pi + \frac{1}{z - y_0} + \sum_{k = 1}^{\infty} \left(\frac{1}{z - y_k} + \frac{1}{z - y_{-k}} \right).$$
Hence, for all $\alpha_1, \dots, \alpha_r \notin \{y_k,\; k\in\mathbb Z\}$, 
$$ \left( \dfrac{2 i \pi}{n} \right)^r
\dfrac{Z'_n(e^{2i \alpha_1 \pi/n})}{Z_n(e^{2i \alpha_1 \pi/n})}
\dfrac{Z'_n(e^{2i \alpha_2 \pi/n})}{Z_n(e^{2i \alpha_2 \pi/n})}
\cdots
\dfrac{Z'_n(e^{2i \alpha_r\pi/n})}{Z_n(e^{2i \alpha_r \pi/n})}
\underset{n \rightarrow \infty}{\longrightarrow} 
\frac{ \xi_{\infty}'(\alpha_1) }{ \xi_{\infty} (\alpha_1)} \cdots
\frac{ \xi_{\infty}'(\alpha_r) }{ \xi_{\infty} (\alpha_r)}.$$
 \end{corollary}

\subsection{Outline of the paper}
 The proof of Theorem \ref{thm:main}  will be made in several steps in Section \ref{preuvedumain} , using estimates
on the argument of $Z_n$, stated in Section \ref{section:estimatesZn}, and estimates on the renormalized
eigenangles $y_k^{(n)}$, stated in Section \ref{section:estimatesykn}.

 In Section \ref{section:properties}, we establish 
some properties of the limiting random function $\xi_{\infty}$, and prove some general results about convergence of  linear statistics at the microscopic scale. Unlike  other scales, convergence in law is proved for very general test functions (essentially integrable) and, as expected, no renormalization is needed in the non smooth cases (e.g. indicator functions).

In Section \ref{subsection:log_der}, we prove some fine and technical estimates related to the logarithmic derivative that we shall need later for the convergence of moments of ratios and 
we state some related conjectures 
on the behavior of the Riemann zeta function in the neighborhood of the critical line. 

In Section \ref{lenotre}, using  estimates from previous sections  and the work of Borodin, Olshanski and Strahov (\cite{bib:BOS}), we prove the convergence of the expectation of  ratios of characteristic polynomials to the corresponding expectations of ratios of $\xi_\infty$. This in turn provides simpler formulas for the corresponding conjecture for ratios of the Riemann zeta function.   The results in Section \ref{lenotre} complete the convergence obtained in Corollary \ref{ratiosconv} and we can summarize them in the following proposition:

\begin{thm}
The following results on ratios hold:
\begin{enumerate}
	\item For any $p > 0$ and any compact set $K \subset \C \backslash \R$, we have:
	$$ \sup_{n \in \N \sqcup \{ \infty \}} \E\left( \sup_{ \left(z,z'\right) \in K^2}
	\left| \frac{\xi_n(z')}{\xi_n(z)} \right|^p \right) < \infty.$$
	\item For $z_1, \dots, z_k, z_1', \dots, z_k' \in \C \backslash \R$, and for all $n \in \mathbb{N} \sqcup \{\infty\}$,
	$$\E\left( \prod_{j=1}^k \left |\frac{\xi_n(z_j')}{\xi_n(z_j)}\right |\right) < \infty$$
	Moreover, for every compact set $K$ in $\C \backslash \R$,
	we have the following convergence, uniformly in $z_1, z_2, \dots, z_k, z'_1, \dots, z'_k \in K$: 
	$$  \E\left( \prod_{j=1}^k \frac{\xi_n     (z_j')}{\xi_n     (z_j)}\right)
	\underset{n \rightarrow \infty}{\longrightarrow} 
	\E\left( \prod_{j=1}^k \frac{\xi_\infty(z_j')}{\xi_\infty(z_j)}\right).$$        
	\item For all $z_1, \dots, z_k, z'_1, \dots, z'_k \in \C \backslash \R$ such that 
	$z_i \neq z'_j$ for $1 \leq i, j \leq n$, we have
	$$
	\det\left( \frac{1}{z_i-z_j'} \right)_{i,j = 1}^{k}\E\left( \prod_{j=1}^k \frac{\xi_\infty(z_j')}{\xi_\infty(z_j)} \right) =
	\det\left( \frac{1}{z_i-z_j'} 
	\E\left( \frac{\xi_\infty(z_j')}{\xi_\infty(z_i)} \right) \right)_{i,j=1}^k
	$$
	and moreover:
	$$
	\E\left( \frac{\xi_\infty(z')}{\xi_\infty(z)} \right)
	= \left\{\begin{array}{cc}
	1              & \textrm{if } \Im(z)>0 \\
	e^{i2\pi(z'-z)} & \textrm{if } \Im(z)<0 \\
	\end{array}\right.
	$$     
\end{enumerate}
And we conjecture that if $\omega$ is a uniform random variable on $[0, 1]$ and $T>0$ a real parameter going to infinity, 
then, for all $z_1, \dots, z_k, z'_1, \dots, z'_k \in \C \backslash \R$, 
such that $z_i \neq z'_j$ for all $i, j$, 
\begin{align*}& \mathbb{E}\left( \prod_{j=1}^k
\frac{ \zeta\left( \half + i T \omega - \frac{i 2\pi z'_j}{\log T} \right) }
{ \zeta\left( \half + i T \omega  - \frac{i 2\pi z_j}{\log T} \right)  } \right)
\\ & \stackrel{T \rightarrow \infty}{\longrightarrow} 
\det\left( \frac{1}{z_i-z_j'} \right)^{-1} \det\left( \frac{\mathds{1}_{\Im (z_i) > 0} + e^{2i \pi (z'_j - z_i)}
	\mathds{1}_{\Im (z_i) < 0} } {z_i-z_j'} \right)_{i,j=1}^k,
\end{align*}  
where the last expression is well-defined where the $z_i$ and the $z'_j$ are all distinct, and is 
extended by continuity to the case where some of the $z_i$ or some of the $z'_j$ are equal. 
\end{thm}
This last conjecture looks simpler than the usual one (see e.g. \cite{bib:CFZ08}) which involves  complicated sums and difficult combinatorics. Note that this  simpler form of the conjecture first appeared in Rodgers' work \cite{Rod15} where he also used the Borodin-Olshanski-Strahov formula. It should be added that  it was already observed by the authors in \cite{bib:BOS} that taking the limit in the expectation of ratios of characteristic polynomials made sense. However, the natural question whether the ratios themselves converge remained open, as well as establishing the convergence stated in the proposition above. 

The expectation of products of the logarithmic derivative evaluated at distinct points was also computed in \cite{bib:CFZ08}; we also provide an alternative formula using the  determinantal form above.  

Eventually, in Section \ref{section:mesoscopic}, we prove that in a sense which can be made precise, the fluctuations of the 
determinantal sine-kernel process, viewed at a scale 
tending to infinity, converge in law to a blue noise, i.e. a noise whose spectral density is proportional to 
the frequency. In relation with this convergence, we show that the fluctuations of $\xi'_{\infty}/\xi_{\infty}$, viewed at a large
scale, tend to a holomorphic Gaussian process on $\C \backslash \R$, whose covariance structure is explicitly computed. 
This covariance is consistent with the computation of the two first moments of $\xi_{\infty}'/\xi_{\infty}$.

\section*{Acknowledgement}

We would like to thank Brad Rodgers for very stimulating discussions and for his suggestions which were very helpful in the proof of Theorem \ref{thm:ratio_formula_xi_n}. A.N. would also like to thank Alexei Borodin for mentioning  the problems on ratios of characteristic polynomials at the microscopic scale.

%--------------------------------------------------------------------
\section{Proof of Theorem \ref{thm:main}}\label{preuvedumain}
\subsection{On the argument of the characteristic polynomial } 
\label{section:estimatesZn}
In this section, we study the argument of $Z_n$, in order to deduce estimates on the deviation of $y_k^{(n)}$ from $k$. 

Here, we define the argument as the imaginary part of $\log Z_n$, with the determination of the logarithm given in the previous section.

The next proposition gives a link between the number of eigenvalues of $U_n$ in a given arc of circle, and 
the variation of the argument of $Z_n$ along this arc. 
 The derivation is relatively standard and we shall not reproduce a proof here (see \cite{bib:HughesPhD}, p. 35-36. or \cite{bib:BHNN}, proof of Proposition 2.2).
\begin{proposition}
Consider $A$ and $B$ two points on the unit circle. Note $\wideparen{AB}$ for the arc joining $A$ and $B$ counterclockwise. Denote by $\ell\left( \wideparen{AB} \right)$ the length of the arc and $N\left( \wideparen{AB} \right)$ the number of zeros of $Z_n$ in the arc. We assume that $A$ and $B$ are not zeros of $Z_n$. Then:
$$ N\left( \wideparen{AB} \right) = \frac{ n \ell\left( \wideparen{AB} \right) }{2 \pi}
                 - \frac{1}{\pi}\left[  \Im \log\left(Z_n(B)\right) - \Im \log\left(Z_n(A)\right) \right].$$
\end{proposition}
\begin{rmk}
This shows that the imaginary part of the determination of the logarithm $\Im \log Z_n(z)$ increases with speed $n/2$ and jumps by $-\pi$ when encountering a zero.
\end{rmk}

\begin{corollary} \label{corollary:kykn}
Let $k \in \Z$, and let $\varepsilon>0$ be small enough so that there are no eigenangles of $U_n$ in
$[0, \varepsilon]$ and $(\theta_k^{(n)}, \theta_k^{(n)} + \varepsilon]$. Then:
$$ k = y_k^{(n)} - \frac{1}{\pi}\Im\left( \log\left( Z_n( e^{i (\theta_k^{(n)} + \varepsilon) } )\right) 
                                        - \log\left( Z_n( e^{i \varepsilon } )\right) \right)$$
\end{corollary}
\begin{proof}
Notice first that if $k$ is increased by $n$, $\theta_k^{(n)}$ increases by $2 \pi$, $y_k^{(n)}$ increases by 
$n$, $\lambda_k^{(n)} = e^{i \theta_k^{(n)}}$ does not change, and the assumption made on $\varepsilon$ remains the same. 
Hence, in the equality we want to prove, the right-hand side and the left-hand side both increase by $n$, which implies 
that it is sufficient to show the corollary for $1 \leq k \leq n$. 
If these inequalities are satisfied, let us choose, in the previous proposition, 
$A = e^{i \varepsilon}$ and $B = e^{i (\theta_k^{(n)} + \varepsilon) }$. Then we note that
$$ N\left( \wideparen{AB} \right) = k,$$
and 
$$ \frac{ n \ell\left( \wideparen{AB} \right) }{2 \pi} = \frac{ n \theta_k^{(n)} }{2 \pi} = y_k^{(n)},$$
which proves the corollary. 
\end{proof}

This corollary shows that it is equivalent to control the argument of $Z_n$, and the distance between $k$ and 
$y_k^{(n)}$. In the remaining of this section, we give some explicit bounds on the distribution of 
$\Im \log (Z_n)$ on the unit circle. 
\begin{proposition}
For all $x > 0$, one has 
$$\P\left( |\Im\left( \log Z_n(1) \right)| \geq x \right) \leq 2\exp\left( - \frac{x^2}{C + \log n}\right),$$
 where $C > 0$ is a universal constant. 
\end{proposition}
\begin{rmk}
 In the proof below, we prove that one can take $C = \frac{\pi^2}{6} + 1$. 
\end{rmk}

\begin{proof}
Let us note
$$ X_n = \Im\left( \log Z_n(1)  \right)$$
Thanks to the formula (1.1) in \cite{bib:BHNY}:
$$ \forall \lambda \in \R, \E\left( e^{ \lambda X_n} \right) = \prod_{k=1}^n \frac{ \Gamma\left( k \right)^2}{ \Gamma\left( k+\frac{i\lambda}{2} \right) \Gamma\left( k-\frac{i\lambda}{2} \right) }$$
Let us start with the standard Chernoff bound:
$$ \forall \lambda > 0, \P\left( X_n \geq x \right) \leq e^{- \lambda x } \E\left( e^{\lambda X_n}\right).$$
Now, using the infinite product formula for the Gamma function:
$$ \forall z \in \C, \frac{1}{\Gamma(z)} = e^{\gamma z} z \prod_{j=1}^\infty \left( 1 + \frac{z}{j} \right) e^{-z/j},$$
we have:
\begin{align*}
\E\left( e^{\lambda X_n}\right) & = \prod_{k=1}^n \frac{ \Gamma\left( k \right)^2}{ \Gamma\left( k+\frac{i\lambda}{2} \right) \Gamma\left( k-\frac{i\lambda}{2} \right) }\\
& = \prod_{k=1}^n \left( \frac{k^2 + \frac{\lambda^2}{4} }{k^2} \prod_{j=1}^\infty \frac{ \left( 1 + \frac{k+\frac{i\lambda}{2}}{j} \right)\left( 1 + \frac{k-\frac{i\lambda}{2}}{j} \right) }{ \left( 1 + \frac{k}{j} \right)^2 } \right)\\
& = \prod_{k=1}^n \left( \frac{k^2 + \frac{\lambda^2}{4} }{k^2} \prod_{j=1}^\infty \frac{ \left( j + k+\frac{i\lambda}{2} \right)\left( j + k-\frac{i\lambda}{2} \right) }{ \left( j + k \right)^2 } \right)\\
& = \prod_{k=1}^n \prod_{j=0}^\infty \frac{ \left(j + k\right)^2 +\frac{\lambda^2}{4} }{ \left( j + k \right)^2 }\\
& = \prod_{k=1}^n \prod_{j=0}^\infty \left( 1 + \frac{\lambda^2}{4\left(j + k\right)^2} \right)\\
& \leq \exp\left( \sum_{k=1}^n \sum_{j=0}^\infty \frac{\lambda^2}{4\left(j + k\right)^2} \right)\\
& = \exp\left( \frac{\lambda^2}{4} \sum_{k=1}^n \sum_{j=k}^\infty \frac{1}{j^2} \right)\\
& \leq \exp\left( \frac{\lambda^2}{4} \sum_{k=1}^n \left( \frac{1}{k^2} + \int_k^\infty \frac{dt}{t^2} \right) \right)\\
& = \exp\left( \frac{\lambda^2}{4} \sum_{k=1}^n \left( \frac{1}{k^2} + \frac{1}{k} \right) \right)\\
& \leq \exp\left( \frac{\lambda^2}{4} \left( \frac{\pi^2}{6} + 1 + \log n \right) \right)
\end{align*}
Eventually  for $C = \frac{\pi^2}{6} + 1$, we obtain
$$ \P\left( X_n \geq x \right) \leq \min_{\lambda > 0} e^{- \lambda x + \frac{\lambda^2}{4} \left( C + \log n \right) }. $$
The minimum is reached for $\lambda = \frac{2x}{C + \log n}$, giving us the bound:
$$ \P\left( \Im\left( \log Z_n(1) \right) \geq x \right) \leq \exp\left( - \frac{x^2}{C + \log n}\right).$$
The desired bound is obtained from the symmetry of $\Im\left( \log Z_n(1) \right)$, as eigenvalues are invariant in law under conjugation:
\begin{align*}
  & \P\left( |\Im\left( \log Z_n(1) \right)| \geq x \right)\\
= & \P\left( \Im\left( \log Z_n(1) \right) \geq x \right) + \P\left( -\Im\left( \log Z_n(1) \right) \geq x \right)\\
= & 2 \P\left( \Im\left( \log Z_n(1) \right) \geq x \right)
\end{align*}
\end{proof}
We deduce the following estimate on the maximum of the argument of $Z_n$ on the unit circle:  

\begin{proposition} \label{proposition:maxZn}
Almost surely:
$$ \sup_{ |z| = 1 , z \in \Dc } \left| \Im \log Z_n(z) \right| = O\left( \log n \right)$$
More precisely, for any $D > \sqrt{2}$:
$$ \exists n_0 \in \N, \forall n \geq n_0, \sup_{ |z| = 1 , z \in \Dc } \left| \Im \log Z_n(z) \right| \leq D \log n$$
which means that almost surely:
$$ \limsup_{n} \frac{1}{\log n} \sup_{ |z| = 1 , z \in \Dc } \left| \Im \log Z_n(z) \right| \leq \sqrt{2} $$
\end{proposition}
% \begin{rmk} Cette remarque me parait fausse car $Im(Z_n)$ se comporte comme une gaussienne de variance 1/2 log n et non 
%log n.
 %This $\sqrt{2}$ has to be understood as the same $\sqrt{2}$ in the following convergence in probability. 
% If $X_i$ are i.i.d. real standard Gaussian random variables, then we have:
% $$ \frac{ \max_{i=1, \dots, n} X_i }{ \sqrt{\log n} }\stackrel{n \rightarrow \infty}{\longrightarrow} \sqrt{2}$$
 %The heuristic is the following. Consider $n^{1-\varepsilon}$ points spaced evenly on the circle. The value of the argument of $Z_n$ at these points behaves as i.i.d. Gaussians with variance $\log n$ as we need a distance larger than $O\left( 1/n \right)$ to observe decorrelation. Hence the supremum is larger than:
 %$$ \sup \sqrt{\log n} \max_{i} Z_n(x_i) \approx \sqrt{2} \log n $$
% \end{rmk}
\begin{proof}
Consider $n$ regularly spaced points on the circle, say:
$$ x_{k,n} := e^{ i \frac{2\pi k}{n} }, \quad k = 0, 1, 2, \dots , n-1,$$
and the events:
$$ A_{k,n} := \left\{ |\Im \log Z_n\left( x_{k,n}\right) | \geq D \log n \right\}$$
Because the law of the spectrum of $U_n$ is invariant under rotation, all the events $A_{k,n}$ have the same probability for different $k$'s. Moreover, thanks to the previous Chernoff bound:
\begin{align*}
n \P\left( A_{0,n} \right) & \leq 2 n \exp\left( - \frac{D^2 (\log n)^2 }{C + \log n}\right) \\
                           & \leq 2 n \exp\left( - D^2 \left( \log n - C \right)\right) \\
                           & \leq 2 e^{D^2 C} n^{1-D^2}
\end{align*}
Hence:
$$ \sum_{n=1}^\infty \sum_{k=1}^n \P\left( A_{k,n} \right) = \sum_{n=1}^\infty n \P\left( A_{0,n} \right) < \infty $$
The Borel-Cantelli lemma ensures that, almost surely:
$$ \exists n_0 \in \N, \forall n \geq n_0, \forall k, \quad |\Im \log Z_n\left( x_{k,n}\right)| \leq  D \log n$$

Now consider a point $z = e^{i\theta} \in \Dc$. For fixed $n$, it lies on the arc between $x_{k,n}$ and $x_{k+1,n}$ for a certain $k$. Because
$$ \theta \mapsto \Im \log Z_n(e^{i\theta})$$
is piece-wise linear, increasing with speed $n/2$ and only jumping by $-\pi$, we have:
$$ \Im \log Z_n(e^{i\theta}) \leq \Im \log Z_n(x_{k,n}) + \frac{n}{2}\left( \theta - \frac{2\pi k}{n}\right) \leq \Im \log Z_n(x_{k,n}) + \pi$$
In the other direction, we have
$$ \Im \log Z_n(e^{i\theta}) \geq \Im \log Z_n(x_{k+1,n}) - \frac{n}{2}\left( \frac{2\pi (k+1)}{n} - \theta \right) \geq \Im \log Z_n(x_{k+1,n}) - \pi$$
So that, almost surely:
$$ \exists n_0 \in \N, \forall n \geq n_0, \forall z \in \Dc, \quad |\Im \log Z_n\left( z \right)| \leq \pi + D \log n$$
The more precise estimate  $|\Im \log Z_n\left( z \right)| \leq D \log n$ follows after replacing $D$ by $D' \in (\sqrt{2}, D)$ in the previous computation
and considering $n_0$ large enough so that $\pi < (D-D') \log n$.
\end{proof}

% --------------------------------------------------------------------
\subsection{Precise estimates for the eigenvalues of virtual isometries} 
\label{section:estimatesykn}

The following estimate will reveal crucial  for the proof of Theorem \ref{thm:main}. 

\begin{proposition}
\label{proposition:key_estimate}
Almost surely and uniformly in $n$ and $k$:
$$ y_k^{(n)} = k + O\left( \log(2+|k|) \right)$$ 
\end{proposition}

In fact, if $y_k^{(n)}$ is replaced by $y_k$ ($n \rightarrow \infty$), this estimate is already easily 
deduced from existing literature (for example \cite{bib:meckes}, \cite{bib:soshnikov02}).
The main tool used here is the following lemma:
\begin{lemma}
\label{lemma:preparatory}
Let $E$ be a point process equal to $\{y_k, k \in \mathbb{Z} \}$ or to $\{y^{(n)}_k, k \in \mathbb{Z}\}$ for 
some $n \geq 1$. Then, for all finite intervals $I$, we have 
\begin{eqnarray}
\label{eq:mean_point_count} 
 \E\left( X_I \right) = |I|,\\
\label{eq:var_point_count} 
 \Var ( X_I ) \leq 2 + \frac{2}{\pi^2} \log\left( 1+|I| \right),
\end{eqnarray}
 where $|I|$ denotes the length of $I$, $X_I$ the number of points of $E$ in $I$. 
Moreover the following tail estimates hold for the random  variables $\widetilde{X_I} := X_I - |I|$: 
$$ \forall t \geq 0, \P\left( \left| \widetilde{X_I} \right| \geq t \right)
\leq \exp\left( -\min\left( \frac{t^2}{4 \Var(X_I) }, \frac{t}{2}\right) \right),$$
and, all the exponential moments of $X_I$ are finite, with the following bound for $0 \leq q < \half$,
$$ \E\left( e^{q \left| \widetilde{X_I} \right|} \right) \leq \frac{1}{1 - 2q} + q \sqrt{4 \pi \Var(X_I)} e^{4 q^2 \Var(X_I)}.$$
\end{lemma}
\begin{proof}
Equation \eqref{eq:mean_point_count} is a consequence of the fact that the 1-point correlation 
function of the point processes $E$ is identically $1$. 

Let us now prove the bound \eqref{eq:var_point_count}. Let $f = \mathds{1}_I$;  we have $|\hat{f}(k)|^2 = \frac{2 \sin^2\left( \frac{|I|k}{2} \right)}{\pi k^2}$,
where the Fourier transform of $f$ is normalized as follows: 
$$\hat{f}(k) = \frac{1}{\sqrt{2\pi}} \int_{\R} f(x) e^{-ikx} dx.$$
We also obviously have
$$\sum_{y \in E} f(y)=X_I. $$
 
Then, using the 2-point correlation of $E$, we obtain that
$$  \Var \left( \sum_{y \in E} f(y) \right) = \sum_{k \in \Z^*} \frac{2 \pi }{n} \left( 1 \wedge \frac{|k|}{n} \right) \left|
\hat{f}(\frac{2 \pi k}{n}) \right|^2$$
if $E$ is $\{y^{(n)}_k, k \in \mathbb{Z} \}$, and 
$$ \Var \left( \sum_{y \in E} f(y) \right) = \int_\R  2 \pi  (1 \wedge |k| ) \left| \hat{f}(2 \pi k) \right|^2 dk,$$
if $E$ is $\{y_k, k \in \mathbb{Z} \}$.
If $\nu$ denotes the measure $(1/n) \sum_{k \in \mathbb{Z}^*} \delta_{k/n}$ in the first case and the Lebesgue measure 
in the second case, we get in both cases: 
$$ \Var \left( \sum_{y \in E} f(y) \right)  = \int_\R  2 \pi  (1 \wedge |k| ) \left| \hat{f}(2 \pi k) \right|^2 d
\nu(k).$$

Hence, 
\begin{align*} \Var (X_I) & = \int_\R  2 \pi  (1 \wedge |k| ) \frac{2 \sin^2\left( \pi |I| k \right)}{\pi(2 \pi k)^2}
d \nu(k) 
 \\ & = \frac{2}{\pi^2} \left( \int_{[0,1]} \frac{ \sin^2\left( \pi |I| k \right)}{k} d \nu(k)
+ \int_{(1, \infty)}  \frac{\sin^2\left( \pi |I| k \right)}{k^2} d \nu(k) \right). 
\end{align*}
Now, 
$$  \int_{(1, \infty)}  \frac{\sin^2\left( \pi |I| k \right)}{k^2} d \nu(k) 
\leq \int_{(1, \infty)} \frac{d \nu(k)}{k^2} \leq \int_{1}^{\infty} \frac{dk}{k^2} = 1,$$
and, using the inequality $|\sin t | \leq 1 \wedge |t|$, 
$$ \int_{[0,1]} \frac{ \sin^2\left( \pi |I| k \right)}{k} d \nu(k)
\leq \int_{[0, 1 \wedge (1/ \pi |I|)]} \pi^2 |I|^2 k d \nu(k) 
+ \int_{(1 \wedge (1/ \pi |I|), 1 ]} \frac{ d \nu(k)}{k}.$$
Now, for $0 < a \leq 1$, and for $E = \{ y_k^{(n)}, k \in \mathbb{Z}\}$, 
$$\int_{[0,a]} k d \nu(k) = \frac{1}{n} \sum_{1 \leq m \leq na} \frac{m}{n} \leq \frac{na(na+1)}{2 n^2} 
\mathds{1}_{na \geq 1} \leq \frac{na(2na)}{2 n^2} 
\mathds{1}_{na \geq 1} \leq a^2.$$
and
\begin{align*} \int_{(a,1]} \frac{ d \nu(k)}{k} & = \frac{1}{n}\sum_{an < m \leq n} \frac{1}{(m/n)}   = \left( \sum_{1 \leq m \leq n}\frac{1}{m} \right) - \left( \sum_{1 \leq m \leq an} \frac{1}{m} \right) 
\\ & \leq 1 +  \left( \sum_{1 \leq m \leq n}\frac{1}{m} \right) - \left( \sum_{1 \leq m \leq an+1} \frac{1}{m} \right)
\\ & \leq 1 + (1 + \log n) - \log(an) \leq 2 + \log (1/a).
\end{align*}
These bounds are obvious for $E = \{ y_k, k \in \mathbb{Z}\}$ since $\nu$ is the Lebesgue measure in this case, so we
get
\begin{align*} \int_{[0,1]} \frac{ \sin^2\left( \pi |I| k \right)}{k} d \nu(k)
& \leq \pi^2 |I|^2 \left( 1 \wedge (1/\pi |I|)^2 \right) 
+ \log (\pi |I| \vee 1) + 2
\\ & \leq 1 + \log (\pi (1 + |I|)) + 2 \leq 5 + \log (1 + |I|),
\end{align*}
and then 
$$ \Var (X_I) \leq 2 + \frac{2}{\pi^2} \log( 1 + |I|).$$ 
The estimate of the tail of $\widetilde{X_I}$ can be obtained as follows. If $E = \{y^{(n)}_k, k \in \mathbb{Z}\}$, 
we can assume $|I| < n$, since any interval of size $n$ has a.s. $n$ points in $E$. In this case, and also 
for $E = \{y_k, k \in \mathbb{Z}\}$, the restriction of $E$ to $I$ is determinantal,
its kernel is self-adjoint, nonnegative, and locally trace-class with eigenvalues in $[0, 1]$.
Thanks to Proposition 2 in \cite{bib:meckes} (which is by the
way also a standard result in the theory of point processes),  $X_I$ is a sum of independent 
Bernoulli random variables. We deduce that if $(p_j)_{j \geq 1}$ are the parameters of these 
variables, and if $q \geq 0$, 
$$ \mathbb{E} [e^{q X_I} ] = \prod_{j \geq  1} \left( 1 + p_j(e^{q} -1) \right) 
\leq e^{(e^q -1)\sum_{j  \geq 1} p_j} = e^{(e^q -1) \mathbb{E}[X_I]} = e^{|I|(e^q -1)} < \infty.$$
Moreover, as in Corollary 4 in \cite{bib:meckes}, we can
deduce, for $q < 1/2$, the claimed estimate of the tail by using the Bernstein inequality. 

We get the bound on the exponential moment as follows. One has
$$ \E\left( e^{q \left| \widetilde{X_I} \right|} \right)
=  q \int_{-\infty}^\infty e^{qt} \P\left( |\widetilde{X_I}| \geq t \right) dt $$
Then, we split the integral as an integral on $\R_-$, which is bounded by $1$, and an integral on $\R_+$. For the
integral on $\R_+$, we use the following bound on the tails:
\begin{align*}
       q \int_{\R_+} e^{qt} \P\left( |\widetilde{X_I}| \geq t \right) dt
\leq & q \int_{\R_+} e^{qt} \exp\left( -\min\left( \frac{t^2}{4 \Var(X_I) }, \frac{t}{2}\right) \right) dt \\
\leq & q \int_0^\infty e^{qt-\frac{t^2}{4 \Var(X_I) }} dt + q \int_0^{\infty} e^{qt-\frac{t}{2}} dt dt\\
\leq &  \frac{q}{\half-q} + q \int_\R e^{qt -\frac{t^2}{4 \Var(X_I) } }dt\\
\leq &  \frac{q}{\half-q} + q \sqrt{4 \pi \Var(X_I)} e^{q^2 \Var(X_I)}.
\end{align*}
Adding $1$ to this quantity gives the desired bound. 
\end{proof}
\begin{rmk}
In the case where $E = \{y_k, k \in \mathbb{Z}\}$, an asymptotic estimate for the variance of $X_I$ is 
proven by Costin and Lebowitz \cite{bib:CL}  (see also
Soshnikov \cite{bib:soshnikov02}):
$$ \Var(X_I) = \frac{1}{\pi^2} \log (1+|I|) + O(1).$$
The bound we have proven here is twice this estimate plus $O(1)$. 
\end{rmk}

\begin{lemma}
\label{lemma:lemma_mm_sosh}
Almost surely:
$$ \forall k \in \Z, y_k = k + O\left( \log(2+|k|) \right)$$
\end{lemma}
\begin{proof}
Consider a sine-kernel process $y_k$. For $A >1$ and $a < b$, let $X_{[a,b]}$ be the number of particles
$y_k$ in $[a,b]$, and let $X_A := X_{[0,A]}$.  
From the estimate given in Lemma \ref{lemma:preparatory},
$$ \Var(X_A) \leq  \frac{2}{\pi^2} \log A + O(1)$$

Therefore, for all $D > 0$, 
$$ \P\left( |X_A - A| \geq D \log A \right) \leq 2 \exp\left( -(\log A) \min\left(
\frac{D^2 \pi^2}{8 + O(1/\log A)}, \frac{D}{2} \right)\right).$$

For $D > 2$, and $A$ large enough, $D^2 \pi^2/[8 + O(1/\log A)] > D/2$, which implies: 
$$\P\left( |X_A - A| \geq D \log A \right) \leq 2\exp\left( -(\log A)(D/2)\right) = 2 A^{-D/2}.$$
This quantity is summable for positive integer values of $A$. By Borel-Cantelli's lemma, we deduce that almost surely, 
for $A \in \mathbb{N}$:
$$ X_A = A + O\left( \log (2 +  |A|)  \right).$$
From the inequality
$$X_{[0, \lfloor A \rfloor ]} \leq X_{[0,A]} \leq X_{[0, \lceil A \rceil ]},$$
we deduce that the estimate remains true for all $A \geq 0$. 
Taking $A = y_k$ for $k > 0$ proves the proposition for positive indices. With the same argument one 
handles the negative ones.
\end{proof}

In order to prove Proposition \ref{proposition:key_estimate}, we will also need the following two lemmas:
\begin{lemma}
\label{lemma:lemma_1}
Almost surely:
$$ \forall k \in \Z, y_k^{(n)} = k + O\left( \log n \right)$$
\end{lemma}
\begin{proof}
This is an immediate consequence of Corollary \ref{corollary:kykn} and Proposition \ref{proposition:maxZn}. 
\end{proof}

\begin{lemma}
\label{lemma:lemma_2}
For every $0 < \eta < \frac{1}{6}$, there exists $\varepsilon > 0$ such that, almost surely:
$$ \forall k \in [- n^\eta,  n^\eta], y_k^{(n)} = y_k + O\left( n^{-\varepsilon}\right)$$
\end{lemma}
\begin{proof}
Since $k \in [- n^{1/4},  n^{1/4}] $, we can apply Theorem \ref{thm:as_cv_sine}, which gives, for all $\delta > 0$, 
$$ y_k^{(n)} = y_k + O_{\delta} \left((1+k^2) n^{-\frac{1}{3} + \delta} \right).$$
Since $k = O(n^{\eta})$, 
$$ y_k^{(n)} = y_k +  O_{\delta} \left(n^{2 \eta - \frac{1}{3} + \delta} \right),$$
which, by taking 
$$\delta = \frac{1}{6} - \eta > 0,$$
gives the desired result, for 
$$\varepsilon = - 2 \eta + \frac{1}{3} - \delta = 2 \delta - \delta = \delta > 0.$$
\end{proof}

\begin{proof}[Proof of Proposition \ref{proposition:key_estimate}]
In the range $|k| \geq n^{1/7}$, it is a consequence of Lemma \ref{lemma:lemma_1}. In the range $|k| < n^{1/7}$, it is a consequence of Lemmas \ref{lemma:lemma_mm_sosh} and \ref{lemma:lemma_2}
(for $\eta = 1/7$). 
\end{proof}

% --------------------------------------------------------------------
\subsection{Infinite product representation of the ratio and its convergence}
\label{section:convergence}

First, let us express $\xi_n$ in function of the renormalized eigenangles of $U_n$.
\begin{proposition} \label{proposition:expressionxin}
One has 
$$ \xi_n\left( z \right) =  e^{i \pi z} \prod_{k \in \Z}\left( 1 - \frac{z}{y_k^{(n)}}\right),$$
where the infinite product has to be understood as the limit of the product from $k = -A$ to $k = A$ when the 
integer $A$ goes to infinity. 
\end{proposition}
\begin{proof}
\begin{align*}
\xi_n\left( z \right) & = \frac{Z_n\left( \exp( \frac{i 2 \pi z }{n}) \right)}{Z_n(1)}\\
                      & = \prod_{k=1}^n \frac{1- \frac{\exp( \frac{i 2 \pi z }{n})}{\lambda_k^{(n)}} }{1 - \frac{1}{\lambda_k^{(n)}} }\\
                      & = \prod_{k=1}^n \frac{1-\exp( \frac{i 2 \pi z }{n} - i\theta_k^{(n)}) }{1 - \exp\left( -i\theta_k^{(n)} \right) }\\
                      & = \prod_{k=1}^n \frac{ \exp( \frac{i 2 \pi z }{2n} - \half i\theta_k^{(n)}) } { \exp( -\half i\theta_k^{(n)})}
                                        \frac{ \exp( -\frac{i 2 \pi z }{2n} + \half i\theta_k^{(n)}) - \exp\left( -\half i\theta_k^{(n)} + \frac{i 2 \pi z }{2n} \right) }
                                             { \exp\left( \half i\theta_k^{(n)} \right) - \exp\left( -\half i\theta_k^{(n)} \right) }\\
                      & = \prod_{k=1}^n \exp( \frac{i \pi z }{n} ) 
                                        \frac{ \sin\left( \frac{ \pi z }{n} - \half \theta_k^{(n)} \right) }
                                             { \sin\left( -\half \theta_k^{(n)} \right) }\\
                      & = \exp( i \pi z ) \prod_{k=1}^n \frac{ \sin\left( \half \theta_k^{(n)} - \frac{ \pi z }{n} \right) }
                                                             { \sin\left( \half \theta_k^{(n)} \right) }\\
\end{align*}
Now, the standard product formula for the sine function can be written as follows:
$$ \forall \alpha \in \C, \sin\left( \alpha \right) = \alpha \underset{A \rightarrow \infty}{\lim} 
\prod_{0  < |j| \leq A} \left( 1 - \frac{\alpha}{\pi j} \right).$$
We then have: 
\begin{align*}
\xi_n\left( z \right) & = \exp( i \pi z ) \prod_{k=1}^n \left( 
\							\frac{\half \theta_k^{(n)} - \frac{ \pi z }{n}}{\half \theta_k^{(n)}}
						\underset{A \rightarrow \infty}{\lim} \prod_{0 < |j| \leq A}
						\frac{ 1 - \frac{\half \theta_k^{(n)} - \frac{ \pi z }{n}}{\pi j} }
                                                                               { 1 - \frac{\half \theta_k^{(n)} }{\pi j} }
							\right)\\
                      & = \exp( i \pi z ) \prod_{k=1}^n \left( 
							\left( 1 - \frac{z }{y_k^{(n)}} \right)
						\underset{A \rightarrow \infty}{\lim} 	\prod_{0 < |j| \leq A} \left( 1 - \frac{z}{nj + y_k^{(n)}} \right)
							\right)\\
                      & = \exp( i \pi z ) \prod_{k=1}^n \underset{A \rightarrow \infty}{\lim} 	\prod_{0 \leq |j| \leq A} 
                     \left( 1 - \frac{z}{nj + y_k^{(n)}} \right)
\end{align*}
Using the periodicity of the eigenangles, we have:
$$ y_{k + jn}^{(n)} = jn + y_{k}^{(n)},$$
and then 
$$\xi_n\left( z \right) =  \exp( i \pi z ) \underset{A \rightarrow \infty}{\lim} \prod_{1-nA \leq k \leq n + nA}
\left( 1 - \frac{z}{y_k^{(n)}}\right).$$
Now, for $B \geq 2n$, $A \geq 2$ integers such that $An \leq B \leq An + n -1$,
the product of $1 - \frac{z}{y_k^{(n)}}$ from $1-nA$ to $n+nA$ and the product from $-B$ to $B$ differ by at most 
$2n$ factors, which are all $1 + O(|z|/y^{(n)}_{nA}) + O(|z|/|y^{(n)}_{1 -nA}|) = 1 + O(|z|/nA)$.  
The quotient between these two products is then well-defined and $\exp [O(|z|/A)] = \exp[O(n|z|/B)]$ for
$B$ large enough, which implies 
that it tends to one when $B$ goes to infinity. Hence, 
$$\xi_n\left( z \right) =  \exp( i \pi z ) \underset{B \rightarrow \infty}{\lim} \prod_{-B \leq k \leq B}
\left( 1 - \frac{z}{y_k^{(n)}}\right).$$
\end{proof}

We are now ready to prove Theorem \ref{thm:main}. 
\begin{proof}[Proof of theorem \ref{thm:main}]
Thanks to the estimate from Proposition \ref{proposition:key_estimate}:
$$ y_k^{(n)} = k + O\left( \log(2+|k|) \right)$$
We have that, for $k\geq 1$ and $z$ in a compact $K$:
\begin{align*}
\left( 1 - \frac{z}{y_k^{(n)}}\right)\left( 1 - \frac{z}{y_{-k}^{(n)}}\right) & = 1 - z \frac{O(\log(2+|k|))}{k^2} + O\left( \frac{|z|^2}{k^2} \right)\\
& = 1 + \frac{ O_{K}\left( \log(2+|k|) \right) }{k^2}
\end{align*}
Hence:
$$ \xi_n\left( z \right) =  e^{i \pi z} \prod_{k \in \Z}\left( 1 - \frac{z}{y_k^{(n)}}\right)$$
is a sequence of entire functions uniformly bounded on compact sets. Therefore, by Montel's theorem, uniform 
convergence on compact sets is implied by pointwise convergence. Let us then focus on proving pointwise convergence.

Fix $A \geq 2$. Let us prove that:
\begin{equation}
 \prod_{|k| \leq A}\left( 1 - \frac{z}{y_k^{(n)}}\right) - \prod_{k \in \Z}\left( 1 - \frac{z}{y_k^{(n)}}\right) = 
 O_{K}\left( \frac{\log A}{A} \right), \label{prodkAn}
 \end{equation}
 \begin{equation}
 \prod_{|k| \leq A}\left( 1 - \frac{z}{y_k      }\right) - \prod_{k \in \Z}\left( 1 - \frac{z}{y_k  
 }\right) = O_{K}\left( \frac{\log A}{A} \right). \label{prodkA}
 \end{equation}
Here, the infinite products are, as before, the limits of the products from $-B$ to $B$ for $B$ going to infinity. 
Note that the existence of the infinite product involving $y_k$ is an immediate consequence of the absolute 
convergence of the product
$$\left( 1 - \frac{z}{y_0} \right) \prod_{k \geq 1} \left[ \left( 1 - \frac{z}{y_k} \right)
\left( 1 - \frac{z}{y_{-k}} \right) \right],$$
 stated in Theorem \ref{thm:main}, and following from the estimate:  
 $$\left( 1 - \frac{z}{y_k} \right)
\left( 1 - \frac{z}{y_{-k}} \right) =  1 - z \frac{O(\log(2+|k|))}{k^2} + O\left( \frac{|z|^2}{k^2} \right)
= 1 + \frac{ O_{K}\left( \log(2+|k|) \right) }{k^2}.$$

We now prove \eqref{prodkAn}: a proof of \eqref{prodkA} is simply obtained by removing the indices $n$. We have:
$$ \prod_{|k| \geq A}\left( 1 - \frac{z}{y_k^{(n)}}\right) = 1 +  O_{K}\left( \sum_{k \geq A} \frac{ \log(2+|k|) }{k^2} \right) = 1 +  O_{K}\left( \frac{\log A}{A} \right)$$
and
$$ \prod_{|k| \leq A}\left( 1 - \frac{z}{y_k^{(n)}}\right) = O_{K}\left( 1 \right)$$

Therefore:
\begin{align*}
  & \prod_{|k| \leq A}\left( 1 - \frac{z}{y_k^{(n)}}\right) - \prod_{k \in \Z}\left( 1 - \frac{z}{y_k^{(n)}}\right)\\
= & \prod_{|k| \leq A}\left( 1 - \frac{z}{y_k^{(n)}}\right)\left( 1 - \prod_{|k| > A}\left( 1 - \frac{z}{y_k^{(n)}}\right) \right)\\
= & \prod_{|k| \leq A}\left( 1 - \frac{z}{y_k^{(n)}}\right)\left( 1 - \left(1 +  O_{K}\left( \frac{\log A}{A} \right) \right) \right)\\
= & O_{K}\left( \frac{\log A}{A} \right)
\end{align*}
Because errors are uniform in $n$, this is saying:
$$ \sup_{n} \left| \prod_{|k| \leq A}\left( 1 - \frac{z}{y_k^{(n)}}\right) - \prod_{k \in \Z}\left( 1 - \frac{z}{y_k^{(n)}}\right)\right| \underset{A \rightarrow \infty}{\longrightarrow} 0$$

Now:
\begin{align*}
     & \left| \prod_{k \in \Z}\left( 1 - \frac{z}{y_k^{(n)}}\right) - \prod_{k \in \Z}\left( 1 - \frac{z}{y_k}\right)\right|\\
\leq &    \left| \prod_{|k| \leq A}\left( 1 - \frac{z}{y_k^{(n)}}\right) - \prod_{|k| \leq A}\left( 1 - \frac{z}{y_k}\right)\right|\\
     &  + \left| \prod_{k \in \Z}\left( 1 - \frac{z}{y_k^{(n)}}\right)   - \prod_{|k| \leq A}\left( 1 - \frac{z}{y_k^{(n)}}\right)\right| \\
     &  + \left| \prod_{k \in \Z}\left( 1 - \frac{z}{y_k}\right)         - \prod_{|k| \leq A}\left( 1 - \frac{z}{y_k}\right)\right| \\
\leq &  \left| \prod_{|k| \leq A}\left( 1 - \frac{z}{y_k^{(n)}}\right) - \prod_{|k| \leq A}\left( 1 - \frac{z}{y_k}\right)\right| + O_{K}\left( \frac{\log A}{A} \right)
\end{align*}

Hence, as $y_k^{(n)} \rightarrow y_k$ pointwise:
$$ \limsup_{n \rightarrow \infty}\left| \prod_{k \in \Z}\left( 1 - \frac{z}{y_k^{(n)}}\right) - \prod_{k \in \Z}\left( 1 - \frac{z}{y_k}\right)\right| = O_{K}\left( \frac{\log A}{A} \right)$$

Taking $A \rightarrow \infty$ completes the proof.
\end{proof}

% --------------------------------------------------------------------
\section{First properties of $\xi_\infty$ and linear statistics}
\label{section:properties}

\subsection{The order of \texorpdfstring{$\xi_\infty$}{our limiting function} as an entire function}

We first start with a simple statement on the order of $\xi_\infty$ as an entire function:

\begin{proposition}
Almost surely, $\xi_\infty$ is of order $1$. More precisely, there exists a.s. a random 
$C > 0$, such that for all $z \in \C$. 
$$ |\xi_{\infty}(z)| \leq e^{C|z| \log(2 + |z|)}.$$
On the other hand, there exists a.s. a random $c > 0$ such that for all $x \in \R$, 
$$|\xi_{\infty}(ix)| \geq c e^{c|x|}.$$
\end{proposition}
\begin{proof}
%The function $\xi_\infty$ is an absolutely convergent product with a Weierstrass
%product for functions of order $1 \leq \varrho < 2$. Therefore, it is
%of finite order and the order is at least $1$. It suffices to prove the first inequality.
We have:
$$ \left( 1 - \frac{z}{y_k}\right)\left( 1 - \frac{z}{y_{-k}}\right) = 1 - z \frac{O(\log(2+|k|))}{k^2} + O\left( \frac{|z|^2}{k^2} \right) $$
with errors being uniform in $z$ and $k \geq 1$. We distinguish between three regimes for $k \in \Z$ different
from zero:
$|k| \geq e^{|z|} $, $|z| \leq  |k| < e^{|z|}$, $1 \leq |k| < |z|$. 
In the first regime, 
$$\left( 1 - \frac{z}{y_k}\right)\left( 1 - \frac{z}{y_{-k}}\right) =  1 + 
O \left( \frac{|z|(\log(2+|k|))}{k^2} \right),$$
which implies 
\begin{align*} \left| \prod_{k \geq e^{|z|}} \left( 1 - \frac{z}{y_k}\right) \, \left( 1 - \frac{z}{y_{-k}}\right)
\right|
& \leq \exp \left( O \left( |z| \sum_{k \geq e^{|z|}} \frac{\log(2 + k)}{k^2} \right) \right)
\\ & = \exp \left( O \left( |z| \sum_{k \geq e^{|z|}} k^{-3/2} \right) \right)
\\ & =  \exp \left( O \left( |z|e^{-|z|/2}  \right) \right) = O(1). 
\end{align*}
In the second regime, $$\log (2 + |k|) \leq \log (e^{|z|} + 2) \leq \log(3 e^{|z|}) \leq |z| + 2,$$
and then 
$$\left( 1 - \frac{z}{y_k}\right)\left( 1 - \frac{z}{y_{-k}}\right) = 1 + O \left( \frac{|z|(|z|+2)}{k^2}
\right),$$
which implies 
$$ \left| \prod_{|z| \leq k <  e^{|z|}} \left( 1 - \frac{z}{y_k}\right)\left( 1 - \frac{z}{y_{-k}}\right) 
\right|
\leq \exp \left( O \left( |z|(|z|+2)  \sum_{k \geq |z| \vee 1} \frac{1}{k^2} \right) \right)
 = \exp  O(|z|).$$
 Finally, in the third regime, we have, since $|y_k/k|$ is a.s. bounded from below, 
 $$1 - \frac{z}{y_k} = 1 + O (|z/k|),$$
 which in turn implies 
 $$  \left| \prod_{1 \leq k < |z|} \left( 1 - \frac{z}{y_k}\right)\left( 1 - \frac{z}{y_{-k}}\right) 
\right| \leq \exp \left( O \left(|z|  \sum_{1 \leq k < |z|} (1/k) \right) \right) 
= \exp   O \left(|z| \log( 2 + |z|) \right) .$$
 Since 
 $$\left|1 - \frac{z}{y_0} \right| \leq \exp (|z|/ y_0) = \exp O(|z|),$$
 we deduce by combining the three regimes, the  following upper bound: 
 $$|\xi_{\infty} (z)| \leq \exp  O \left(|z| \log( 2 + |z|) \right).$$
  In order to prove the lower bound, we first use the equality: 
  $$ |\xi_{\infty}(ix)|^2 = \prod_{k \in \mathbb{Z}} \left( 1 + \frac{x^2}{y^2_k} \right).$$
 Since $|y_k| = O(|k|)$ for $k \neq 0$, we deduce that there exists a random $c > 0$ such that 
 $$|\xi_{\infty}(ix)|^2  \geq  \prod_{k \neq 0} \left( 1 + \frac{x^2}{c k^2} \right),$$
 and then 
 $$|\xi_{\infty}(ix)| \geq \prod_{k \geq 1} \left( 1 + \frac{x^2}{c k^2} \right) = \frac{\sinh(\pi x/\sqrt{c})}
 {\pi x/ \sqrt{c}},$$
 which shows the lower bound given in the proposition. 

\end{proof}

\subsection{Convergence of linear statistics}\label{sec::statlin}

We proved in \cite{bib:MNN} that if $E_n$ is the set of zeros of $\xi_{n}$ (i.e. the set of eigenvalues of $U_n$, multiplied by $n/2\pi$), and if
$E_{\infty}$ is the set of zeros of $\xi_{\infty}$, then, for all measurable and bounded functions $f$ from $\mathbb{R}$ to $\mathbb{C}$, with compact support,  the following convergence in law holds:
$$\sum_{x \in E_n} f(x)  \underset{n \rightarrow \infty}{\longrightarrow} \sum_{x \in E_{\infty}} f(x).$$
We now improve this result by showing that it holds for more general test functions and show how linear statistics can be expressed in terms of $\xi_\infty$.

\begin{proposition}
	Let $E_n$ be the set of zeros of $\xi_{n}$ (i.e. the set of eigenvalues of $U_n$, multiplied by $n/2\pi$), and 
	$E_{\infty}$ the set of zeros of $\xi_{\infty}$. Then, for all integrable functions $f$ from $\mathbb{R}$ to $\mathbb{C}$, 
	$$\sum_{x \in E_n} f(x)  \underset{n \rightarrow \infty}{\longrightarrow} \sum_{x \in E_{\infty}} f(x)$$
	in distribution. 
\end{proposition}
\begin{proof}
	For $A > 0$, let $f_A$ be the function given by $f_A(x) := f(x) \mathds{1}_{|x| \leq A, |f(x)| \leq A}$ and let 
	$g_A := f - f_A$. Proposition 4.1 of \cite{bib:MNN}  implies that the proposition is true when $f$ is replaced by $f_A$, 
	i.e.  for all $\lambda \in \mathbb{R}$, 
	\begin{equation} 
	\mathbb{E} \left[ \exp \left( i \lambda \sum_{x \in E_n} f_A(x) \right)  \right] 
	\underset{n \rightarrow \infty}{\longrightarrow} 
	\mathbb{E} \left[ \exp \left( i \lambda \sum_{x \in E_{\infty}} f_A(x) \right)  \right]. \label{convergencefourierNNM}
	\end{equation}
	On the other hand, since the one-point correlation function of $E_n$ and $E_{\infty}$ is equal to $1$, we have
	$$ \mathbb{E} \left[ \sum_{x \in E_n} |g_A(x)| \right] = 
	\mathbb{E} \left[ \sum_{x \in E_n} |g_A(x)| \right] = \int_{\mathbb{R}} |g_A|.$$
	Hence, 
	
	\begin{align*} 
	\mathbb{E} \left[ \exp \left( i \lambda \sum_{x \in E_n} f(x) \right)  \right] 
	& =  \mathbb{E} \left[ \exp \left( i \lambda \sum_{x \in E_n} f_A(x) \right)  \right] 
	+ O \left( |\lambda| \mathbb{E} \left[ \sum_{x \in E_n} |g_A(x)| \right] \right) 
	\\ &  =  \mathbb{E} \left[ \exp \left( i \lambda \sum_{x \in E_n} f_A(x) \right)  \right] 
	+ O \left(  |\lambda| \int_{\mathbb{R}} |g_A| \right)
	\end{align*}
	and the similar estimate with $E_n$ replaced by $E_{\infty}$. 
	Taking the limsup of the difference when $n$ goes to infinity gives, using \eqref{convergencefourierNNM}:
	$$ \underset{n \rightarrow \infty}{\lim \sup} 
	\left| \mathbb{E} \left[ \exp \left( i \lambda \sum_{x \in E_n} f(x) \right)  \right]
	-  \mathbb{E} \left[ \exp \left( i \lambda \sum_{x \in E_{\infty}} f(x) \right)  \right] \right|
	= O \left(  |\lambda| \int_{\mathbb{R}} |g_A| \right)
	$$
	for all $A > 0$. Now, by dominated convergence, the integral of $|g_A|$ goes to zero when $A$ goes to 
	infinity, which gives the desired result. 
\end{proof}
It is natural to conjecture that something similar happens for the zeros of the Riemann zeta function: 
\begin{conjecture}
	Assume the Riemann hypothesis. For all functions $f$ from $\mathbb{R}$ to $\mathbb{R}$ 
	such that 
	$$\int_{\mathbb{R}} |f(x)| \log ( 2 + |x|) dx<\infty,$$
	$$\sum_{x \in E^{\zeta}_T} f(x) 
	\underset{T \rightarrow \infty}{\longrightarrow} \sum_{x \in E_{\infty}} f(x)$$
	in distribution, where $E^{\zeta}_T$ denotes the non-trivial zeros of $z \mapsto \zeta\left( \frac{1}{2} + i T \omega - \frac{2 i \pi
		z}{\log T} \right) $, and $\omega$ is a uniform variable in $[0,1]$. 
\end{conjecture}
The extra factor $\log (2 + |x|)$ in the integrability condition is due to the fact that we sum $f$ over all zeros of 
$\zeta$, who have a logarithmically increasing average density.

One can also express linear statistics of $E_n$ in terms of the logarithm of $\xi_{\infty}$. We
have the following: 
\begin{proposition}
	Let $f$ be a $\mathcal{C}^1$ function from $\mathbb{R}$ to $\mathbb{R}$, integrable, such that 
	$$|x f(x) | \underset{|x| \rightarrow \infty}{\longrightarrow} 0, \; \int_{\mathbb{R}} |x f'(x)| < \infty.$$
	Then we have a.s. 
	$$\sum_{x \in E_n} f(x) - \int_{\mathbb{R}} f(x) dx =  
	\frac{1}{\pi} \int_{\infty}^{\infty}  \Im \log (\xi_n (y)) f'(y) dy.$$
\end{proposition}
Of course, a similar result holds with $E_{\infty}$ instead of $E_n$. 
\begin{proof}
	For all $A > 0$, 
	\begin{align*} \sum_{x \in E_n \cap [-A,A]} f(x) 
	& = f(-A)  \operatorname{Card} (E_n \cap [-A,A]) 
	+   \sum_{x \in E_n \cap [-A,A]} \int_{-A}^x f'(y) dy 
	\\ & = f(-A)  \operatorname{Card} (E_n \cap [-A,A]) 
	+ \int_{-A}^{A} f'(y) \operatorname{Card} (E_n \cap [y,A]) dy 
	\\ & = f(-A) (N_n(A) - N_n(-A)) +  \int_{-A}^{A} (N_n(A) - N_n(y)) f'(y) dy 
	\\ & = f(A) N_n(A) - f(-A) N_n(-A) -  \int_{-A}^{A} N_n(y)f'(y) dy.
	\end{align*}
	Since $N_n(y) = O(1 + |y|)$ a.s., we deduce, from the assumptions made on $f$, that almost surely:
	$$ \sum_{x \in E_n } f(x)  = - \int_{-\infty}^{\infty}  N_n(y)f'(y) dy 
	=  \frac{1}{\pi} \int_{\infty}^{\infty}  \Im \log (\xi_n (y)) f'(y) dy - 
	\int_{-\infty}^{\infty}  y f'(y) dy.$$
	Doing an integration by parts gives the desired result. 
\end{proof}

\section{Fine estimates for the logarithimic derivative and related conjectures for the Riemann zeta function}
\label{subsection:log_der}

We now state a conjecture which relates the random function $\xi_{\infty}$ to the behavior of the Riemann zeta function $\zeta$ close to the critical line:
\begin{conjecture}
\label{conjecture}
Let $\omega$ be a uniform random variable on $[0, 1]$ and $T>0$ a real parameter going to infinity. Our random limiting function should be related to the renormalized zeta function with randomized argument. We conjecture the following convergence in law, uniformly in the parameter $z$ on every compact set:
$$ \left( \frac{ \zeta\left( \half + i T \omega - \frac{i 2\pi z}{\log T} \right) }{ \zeta\left( \half + i T \omega \right) } ; z \in \C \right)
   \stackrel{T \rightarrow \infty}{\longrightarrow} \left( \xi_\infty(z) ; z \in \C \right)$$

By taking logarithmic derivatives, it is natural also to conjecture the following convergence 
$$ \left( \frac{-i 2\pi}{\log T} \frac{\zeta'}{\zeta}\left( \half + i T \omega - \frac{i 2\pi z}{\log T} \right) ; z \in \C \right)
   \stackrel{T \rightarrow \infty}{\longrightarrow} \left( \frac{ \xi_{\infty}' }{ \xi_{\infty} }(z) ; z \in \C \right)$$
on compact sets bounded away from the real line.
\end{conjecture}

This conjecture is supported by the following lemma: 
\begin{lemma} \label{lemma33}
We have, for $z \notin \R$, 
$$ \frac{ \xi_{\infty}' }{ \xi_{\infty} }(z) = i \pi + \sum_{k \in \Z} \frac{1}{z-y_k}
 =: i \pi + \frac{1}{z - y_0} + \sum_{k = 1}^{\infty} \left(\frac{1}{z - y_k} + \frac{1}{z - y_{-k}} \right),$$
and when the random variable $\omega$ is fixed:
$$ \frac{-i 2\pi}{\log T} \frac{\zeta'}{\zeta}\left( \half + i T \omega - \frac{i 2\pi z}{\log T} \right) 
 = i \pi + \sum_{\tilde{\gamma} \in E^{\zeta}_T} \frac{1}{z-\tilde{\gamma}} + o(1)$$ 
where $E^{\zeta}_T$ are the non-trivial zeros of the Riemann zeta function centered around $\half + i\omega T$ \emph{and} renormalized so that their average spacing around the origin is $\Oc\left( 1 \right)$. More precisely, if $\tilde{\gamma} \in E^{\zeta}_T$, then:
$$ \tilde{\gamma} := \frac{ -\log T }{2 \pi i} \left( \rho - \half - i \omega T \right)$$
with $\rho$ a zero of $\zeta$.
The infinite sum on $\tilde{\gamma}$ has to be understood as follows:
$$\sum_{\tilde{\gamma}} \frac{1}{z-\tilde{\gamma}} = \frac{1}{z - \tilde{\gamma}_0}
+ \sum_{k=1}^{\infty} \left(\frac{1}{z-\tilde{\gamma}_k} 
+ \frac{1}{z-\tilde{\gamma}_{-k}} \right),$$
where $(\tilde{\gamma}_k)_{k \in \mathbb{Z}}$ are ordered by increasing real part, increasing
imaginary part if they have the same real part, and counted with multiplicity. 
\end{lemma}
\begin{rmk}
 The absolute convergence of the last sum can be easily deduced from the classical 
 estimate, for $A > 2  $, on the number of nontrivial zeros 
 $N(A)$ with imaginary part in $[0,A]$, or in $[-A,0]$: 
 $$N(A) =\varphi(A) + O(\log A),$$
 for   $$\varphi(A) = \frac{A}{2 \pi} \log \left( \frac{A}{2 \pi e} \right).$$
Indeed, all the ways to number the renormalized zeros $\tilde{\gamma}$ consistently with the statement of the 
lemma are deduced 
from each other by translation of the indices, and for any such numbering one checks that 
$$ \tilde{\gamma}_k =  \sgn (k) \frac{\log T}{2 \pi} \varphi^{(-1)} (|k|)  + O(\log (2+|k|)),$$
where $\varphi^{(-1)}$ is the inverse of the bijection from $[2 \pi e, \infty)$ to $\R_+$, induced by 
$\varphi$. The implicit constant depends on $T, \omega$ and the precise numbering of the zeros, but not on $k$. 
 This  estimate is sufficient to ensure the convergence of the last series in the lemma, when one takes into
 account that $\varphi^{(-1)} (k) \geq k/\log k$ for all $k \geq 2$.
 The sum of the series does not depend on the numbering of the $\tilde{\gamma}$'s, since any translation of the 
 indices change the partial sums by a bounded number of terms, which tend to zero. 
 Note that the $\tilde{\gamma}$'s are all real if and only if the Riemann hypothesis is satisfied. 
\end{rmk}

\begin{proof}
The convergence of the first series in the lemma is easily deduced from the estimate
 in Proposition \ref{proposition:key_estimate}. The partial sums are the logarithmic derivatives 
 of the corresponding partial products associated to $\xi_{\infty}$. Since
 uniform convergence on compact sets of non-vanishing holomorphic functions
 implies the corresponding convergence of the logarithmic derivative, we get the part of the lemma 
 related to $\xi_{\infty}'/\xi_{\infty}$. 
For the formula involving $\zeta$, we start by the Hadamard product formula for the zeta function:
$$ \forall s \in \C\backslash \{ 1 \}, \zeta\left( s \right) = \pi^{s/2}
\frac{ \prod_{\rho}\left( 1 - \frac{s}{\rho} \right) }{ 2(s-1) \Gamma\left( 1 + \frac{s}{2} \right) }.$$
The product has to be computed by grouping pairs of conjugate non-trivial zeros of zeta. 
Hence, for $s$ not a zero nor a pole:
$$ \frac{\zeta'}{\zeta}(s) = \half \log \pi + \sum_{\rho} \frac{1}{s-\rho} - \frac{1}{s-1} - \half \frac{\Gamma'}{\Gamma}\left( 1 + \frac{s}{2} \right)$$
Take $s = \half + i T \omega - \frac{i 2\pi z}{\log T}$ with $T \rightarrow \infty$ and use the asymptotics 
$\frac{\Gamma'}{\Gamma}\left( 1 + \frac{s}{2} \right) = \log T + O(1)$. The error is uniform in $z$ on compact sets
away from the real line. Then:
\begin{align*}
\frac{-i 2\pi}{\log T} \frac{\zeta'}{\zeta}\left( \half + i T \omega - \frac{i 2\pi z}{\log T} \right) 
& = \frac{-i 2\pi}{\log T} \sum_{\rho} \frac{1}{s-\rho} + \frac{i 2\pi}{\log T} \half \left( \log T + O(1) \right) + o(1) \\ 
& = i \pi + \frac{-i 2\pi}{\log T} \sum_{\rho} \frac{1}{- \frac{i 2\pi z}{\log T} - \left(\rho-\half - i\omega T \right) } + o(1)
\end{align*}
Here, all the sums on $\rho$ are obtained by grouping pairs of conjugate values of $\rho$. 
Writing the last sum as a  function of the sequence $(\tilde{\gamma}_k)_{k \in \mathbb{Z}}$ gives 
$$\frac{-i 2\pi}{\log T} \frac{\zeta'}{\zeta}\left( \half + i T \omega - \frac{i 2\pi z}{\log T} \right) 
= i \pi + \sum_{k=1}^{\infty} \left(  \frac{1}{z-\tilde{\gamma}_{a+k}} +  \frac{1}{z-\tilde{\gamma}_{a+1 -k}}
\right)+ o(1),$$
where $a$ depends only on the way to number the $\tilde{\gamma}_k$'s. Changing the partial sums by at most $2|a| + 1$
terms, all tending to zero, gives the partial sums of the series in the lemma. 
\end{proof}

Our formulation can be easily related to the GUE conjectures \cite{bib:RudSar}, which is the natural extension of  Montgomery's conjecture \cite{bib:Montgomery} on pair correlations. Indeed, the previous lemma gives a good heuristic of Conjecture \ref{conjecture}: since the randomized and renormalized zeros $\tilde{\gamma}$ are expected to behave like a sine kernel point process, the two expressions should match in law when $T \rightarrow \infty$. 
It is interesting to notice that the term $i\pi$ in the expression of $\zeta'/\zeta$ is due to the  Archimedian gamma factor in the Hadamard product of $\zeta$.  With the same renormalization corresponding to the average spacing of the zeros, we get the same term for the logarithmic derivative of the characteristic polynomial of the CUE. 

We will now compute the first two  moments of $\frac{ \xi_{\infty}'}{\xi_{\infty}}$, which will naturally give
a conjecture on the corresponding moments of $\frac{\zeta'}{\zeta}$. A particular case of our 
conjecture is in fact equivalent to the pair correlation conjecture under Riemann hypothesis, thanks 
to results  by Goldston, Gonek and Montgomery \cite{bib:GGM01}.
One should also note that recently Farmer, Gonek, Lee and Lester obtain in \cite{bib:FGLL} an
equivalent formulation, with different methods, for the moments of the logarithmic derivative of the Riemann
zeta function in terms of the correlation functions of the sine kernel: the objects that are introduced 
there are different but our formulation is essentially the same as theirs. The main difference is that
we propose to consider directly a random meromorphic function which follows from a conjecture for 
the ratios of the zeta function itself (in particular there is no more $n$-limit to consider on the random matrix
side) and that the logarithmic derivative $\xi_\infty^{'}/\xi_\infty$ seems to carry some spectral interpretation 
(see the last section and the reference there to the recent work by Aizenman and Warzel \cite{bib:AiWa}). 
\bigskip

As shown in Lemma \ref{lemma33}, $\xi'_{\infty}/\xi_\infty$ can be written as an infinite sum 
indexed by $\mathbb{Z}$ which is not 
absolutely convergent, but which converges if we cut the sum at $-k$ and $k$ for $k \in \mathbb{N}$, 
and then let 
$k \rightarrow \infty$. Instead of considering the terms indexed by $m \in \{-k, -k+1, \dots, k\}$, it 
can be more conveninent to take all the terms of index $m$ such that $|y_m| \leq A$, and then
to let $A \rightarrow \infty$. 
The following result says that the two ways to consider the infinite sum give the same result. 
\begin{proposition} \label{proposition:twowaystosum}
 Almost surely, for all $z \notin \{y_k, k \in 
\mathbb{Z}\}$, 
$$\frac{\xi'_{\infty}(z)}{\xi_{\infty} (z)} 
= i \pi+\underset{A  \rightarrow \infty}{\lim} 
\sum_{[y_k| < A} \frac{1}{z - y_k}.
$$
\end{proposition}
\begin{proof}
 By Lemma \ref{lemma:lemma_mm_sosh}, there 
exists almost surely $C > 0$ such that 
$$|y_k - k| \leq C \log (2 + |k|)$$
for all $k \in \mathbb{Z}$. 
It is sufficient to show that 
almost surely, for all $z \notin \{y_k, k 
\in \mathbb{Z}\}$, 
$$\left(i \pi +\sum_{|y_k| < A} \frac{1}{z-y_k} \right)
- \left(\sum_{|k|<A  - C \log (2+A)} \frac{1}{z-y_k}+i \pi\right)
\underset{A \rightarrow \infty}{\longrightarrow} 0.$$
Indeed, the second term of the difference is 
already known to converge to 
$\xi'_{\infty}(z)/\xi_{\infty} (z)$.
Now, $|k| < A - C \log (2+A)$ implies that 
$$|y_k| \leq |k| + C \log (2+|k|)
 \leq |k| + C \log (2+A) <A,$$
 and then we have to show that
 $$ \sum_{|k|\geq A  - C \log (2+A), |y_k| < A}
 \frac{1}{z - y_k} \underset{A \rightarrow \infty}{\longrightarrow} 0.$$
 Since $|y_k| \geq |k| -C \log (2+|k|)$, it
 is sufficient to prove 
 $$\sum_{|k|\geq A  - C \log (2+A), 
 |k| - C \log (2+|k|) < A } \frac{1}{|z - y_k|} \underset{A \rightarrow \infty}{\longrightarrow} 0.$$
Now, this convergence holds since for $C$,
$z$ and $(y_k)_{k \in \mathbb{Z}}$ fixed,  the 
number of terms of the sum is $O(\log A)$ when $A$ 
goes to infinity, and all the terms are 
$O(1/A)$.
 
\end{proof}
We will now bound some exponential moments related to $\xi'_{\infty}/\xi_{\infty}$. 
In order to apply this bound later to convergence results, it will also be useful to consider $\xi'_n/\xi_n$ for finite $n$. 
The infinite product given in Proposition \ref{proposition:expressionxin} is clearly uniformly convergent
 in compact sets if we regroup the terms of indices $k$ and $-k$, and $\xi_n$ does not vanish outside 
 the real axis. Hence, we can take the logarithmic derivative: 
 $$\frac{\xi'_{n}(z)}{\xi_{n} (z)} 
= i \pi+ \frac{1}{z - y^{(n)}_0} + \sum_{k \geq 1} \left(\frac{1}{z - y^{(n)}_k} + \frac{1}{z - y^{(n)}_{-k}}\right).
$$
Since $y^{(n)}_k - k$ is $n$-periodic and then bounded, one deduces that we also have 
 $$\frac{\xi'_{n}(z)}{\xi_{n} (z)} 
= i \pi+\underset{A  \rightarrow \infty}{\lim} 
\sum_{[y^{(n)}_k| < A} \frac{1}{z - y^{(n)}_k}.
$$
From now, we will allow $n$ to be either $\infty$ or a strictly positive integer, and we will write 
by convention $y^{(\infty)}_k := y_k$. Moreover, we define:
$$ \sum_{ |y^{(n)}_k| > A } \frac{1}{z-y^{(n)}_k} := \frac{\xi'_{n}(z)}{\xi_{n} (z)}  - i \pi 
- \sum_{[y^{(n)}_k| \leq A} \frac{1}{z - y^{(n)}_k}.$$
Then, we have the following estimate: 
\begin{proposition}
\label{proposition:uniform_tail_control}
Let $K \subset \C \backslash \R$ be a compact set. Then, there exists 
$ C_K > 0$, depending only on $K$, such that for all $p \geq 0$ and for all $A \geq C_K(1+p^2\log(2+p))$, 
$$ \sup_{n \in \mathbb{N} \sqcup \{\infty\}} 
\E\left[ \sup_{z \in K} e^{ p \left| \sum_{ |y^{(n)}_k| > A } \frac{1}{z-y^{(n)}_k}\right| } \right]
\leq 1 +  \frac{C_K p \log A}{\sqrt{A}}$$
In particular, for all fixed $p>0$, we have:
$$ \limsup_{A \rightarrow \infty} \sup_{n \in \mathbb{N} \sqcup \{\infty\}} 
\E\left[ \sup_{z \in K} e^{ p \left| \sum_{ |y^{(n)}_k| > A } \frac{1}{z-y^{(n)}_k}\right| } \right]=  1
$$
\end{proposition}
\begin{proof}
Let $\alpha>1$ be an exponent to be decided later and denote for every $\ell \in \Z$ the intervals:
$$ I_\ell   := \left( |\ell|^\alpha, \left( |\ell| + 1 \right)^\alpha \right]$$
$$ I_\ell^A := \sgn(\ell) \left( I_\ell \cap \left[ A, \infty \right) \right)$$
First there is a deterministic constant $C_{K, \alpha}>0$ such that for $\ell \geq 0$:
$$ |y^{(n)}_k| \in I_\ell\Rightarrow \left| \frac{1}{z-y^{(n)}_k} - \frac{\sgn y^{(n)}_k}{\left( 1 + 
\ell \right)^\alpha } \right| \leq \frac{C_{K, \alpha}}{\left( 1 + \ell \right)^{\alpha+1}}$$
Then, thanks to the triangular inequality and Proposition \ref{proposition:twowaystosum}, and using the notation $X$ 
and $\widetilde{X}$ given in Lemma \ref{lemma:preparatory}, 
\begin{align*}
     & \left| \sum_{ |y^{(n)}_k| > A } \frac{1}{z-y^{(n)}_k}\right|
\leq  \sum_{ \ell  \geq 0 } \left| \sum_{ |y^{(n)}_k| \in I_\ell^A  } \frac{1}{z-y^{(n)}_k}  \right| \\
\leq & \sum_{ \ell \geq 0} \left( 
       \left| \sum_{ |y^{(n)}_k| \in I_\ell^A  } \frac{\sgn y^{(n)}_k}{\left( 1 + \ell \right)^{\alpha}} \right| 
     + \sum_{ |y^{(n)}_k| \in I_\ell^A  } \frac{C_{K, \alpha}}{\left( 1 + \ell \right)^{\alpha+1}} \right)\\
\leq & \sum_{ \ell \geq 0} \left( 
       \frac{ \left| X_{I_\ell^A} - X_{-I_\ell^A} \right|}{\left( 1 + \ell \right)^{\alpha}}
     + C_{K, \alpha} \frac{ X_{I_\ell^A} + X_{-I_\ell^A} }{\left( 1 + \ell \right)^{\alpha+1}} \right)\\
\leq & \sum_{ \ell \geq 0 } \left( 
       \frac{ \left| \widetilde{X_{I_\ell^A}} \right| + \left| \widetilde{X_{-I_\ell^A}} \right|}{\left( 1 + \ell \right)^{\alpha}}
     + C_{K, \alpha} \frac{ X_{I_\ell^A} + X_{-I_\ell^A} }{\left( 1 + \ell \right)^{\alpha+1}} \right)\\
\leq & 2 C_{K, \alpha} \sum_{\ell \geq 0} \frac{\left| I_\ell^A\right|}{\left( 1 + \ell \right)^{\alpha+1} }
       + \sum_{\ell \geq 0} \frac{1}{\left( 1 + \ell \right)^{\alpha}} \left( \left| \widetilde{X_{I_\ell^A}} \right|+ \left| \widetilde{X_{-I_\ell^A}} \right| \right)\left( 1 + \frac{C_{K, \alpha}}{1+\ell} \right)
\end{align*}
Notice that $I_\ell^A$ is empty when $|\ell| < A^{\frac{1}{\alpha}} - 1$. Thanks to that, we will now prove that:
\begin{align}
\label{eq:resolvant_tail}
\left| \sum_{ |y^{(n)}_k| > A } \frac{1}{z-y^{(n)}_k}\right| 
= \Oc_{K, \alpha}\left( \frac{1}{A^{\frac{1}{\alpha}}} \right) + \left( 1 + \Oc_{K, \alpha}\left( \frac{1}{A^{\frac{1}{\alpha}}} \right) \right) \sum_{\ell \geq 0} \frac{\left( \left| \widetilde{X_{I_\ell^A}} \right|+ \left| \widetilde{X_{-I_\ell^A}} \right| \right)}{\left( 1 + \ell \right)^\alpha}
\end{align}
Indeed, from the previous equation, the first sum can be written as:
\begin{align*}
        2 C_{K, \alpha} \sum_{\ell \geq A^{\frac{1}{\alpha}}-1} \frac{\left| I_\ell^A\right|}{\left(1+\ell\right)^{\alpha+1}}
\ll_{K,\alpha} & \sum_{\ell \geq A^{\frac{1}{\alpha}}-1} \frac{\left(1+\ell\right)^{\alpha} - \ell^\alpha}{\left(1+\ell\right)^{\alpha+1}}\\
\ll_{\alpha}   & \sum_{\ell \geq A^{\frac{1}{\alpha}}-1} \frac{\left(1+\ell\right)^{\alpha-1}}{\left(1+\ell\right)^{\alpha+1}}\\
=     & \sum_{\ell \geq A^{\frac{1}{\alpha}}-1} \frac{1}{\left(1+\ell\right)^{2}}\\
\ll_{\alpha}  & \frac{1}{A^{\frac{1}{\alpha}}}
\end{align*}
And, in the second sum, write $\left( 1 + \frac{C_{K, \alpha}}{1+\ell} \right) = 1 + \Oc_{K, \alpha}\left( \frac{1}{A^{\frac{1}{\alpha}}} \right)$ to deduce inequality \eqref{eq:resolvant_tail}.

Now, we are ready to exponentiate the inequality \eqref{eq:resolvant_tail} after multiplication by $p \geq 0$. 
Let $\left( \beta_\ell \right)_{\ell \in \Z}$ be the probability weights given by:
$$ \beta_\ell = \frac{1}{Z_{A,\alpha}} \mathds{1}_{\left\{ 1 + |\ell| \geq A^{\frac{1}{\alpha}}\right\}} \frac{1}{\left(1+|\ell|\right)^\alpha}$$
where $Z_{A, \alpha}$  is the normalisation constant, chosen 
in such a way that the sum of $\beta_{\ell}$ for $\ell \in \mathbb{Z}$ is equal to $1$. One easily checks that 
$$Z_{A, \alpha}= \Oc_{\alpha} \left( \frac{1}{A^{1 - (1/\alpha)}} \right).$$
We have:
\begin{align*}
     \quad 
     & \exp\left( p \left| \sum_{ |y^{(n)}_k| > A } \frac{1}{z-y^{(n)}_k}\right| \right)\\
\leq \quad 
     &
     \exp\left( p \Oc_{K, \alpha}\left( \frac{1}{A^{\frac{1}{\alpha}}} \right) \right)
     \exp\left( p \left( 1 + \Oc_{K, \alpha}\left( \frac{1}{A^{\frac{1}{\alpha}}} \right) \right) Z_{A, \alpha} \sum_{\ell \in \Z} \beta_\ell \left| \widetilde{X_{ I_\ell^A}} \right| \right)\\
\stackrel{ \textrm{(Jensen) } }{ \leq }
     &
     e^{ p \Oc_{K, \alpha}\left( \frac{1}{A^{\frac{1}{\alpha}}} \right)}
     \sum_{\ell \in \Z} \beta_\ell \exp\left( p \left( 1 + \Oc_{K, \alpha}\left( \frac{1}{A^{\frac{1}{\alpha}}} \right) \right) \Oc_{\alpha} \left( \frac{1}{A^{1 - (1/\alpha)}} \right) \left| \widetilde{X_{ I_\ell^A}} \right| \right)\\
=    \quad
     &
     e^{ p \Oc_{K, \alpha}\left( \frac{1}{A^{\frac{1}{\alpha}}} \right)}
     \sum_{\ell \in \Z} \beta_\ell \exp\left( q \left| \widetilde{X_{ I_\ell^A}} \right| \right) \\
    =    \quad
     &
     e^{ p \Oc_{K, \alpha}\left( \frac{1}{A^{\frac{1}{\alpha}}} \right)}
    \left( \sum_{\ell \in \Z} \beta_\ell \left[\exp\left( q \left| \widetilde{X_{ I_\ell^A}} \right| \right) - 1 \right]
     +  \sum_{\ell \in \Z} \beta_\ell\right) \\ 
       =    \quad
     &
     e^{ p \Oc_{K, \alpha}\left( \frac{1}{A^{\frac{1}{\alpha}}} \right)} + 
      e^{ p \Oc_{K, \alpha}\left( \frac{1}{A^{\frac{1}{\alpha}}} \right)}
       \sum_{\ell \in \Z} \beta_\ell \left[\exp\left( q \left| \widetilde{X_{ I_\ell^A}} \right| \right) - 1 \right].
\end{align*}
where 
$$ q = \Oc_{K, \alpha}\left( \frac{p}{A^{1 - (1/\alpha)}} \right).$$ 
If $0 < q \leq 1/4$,
we get $1/(1-2q) \leq 1 + 4q$ and then by Lemma \ref{lemma:preparatory}: 
\begin{align*}
 \mathbb{E} \left[e^{q |  \widetilde{X_{ I_\ell^A}} |} - 1 \right] 
 & \leq \frac{1}{1 - 2q} - 1 + q \sqrt{4 \pi  \Var(X_{ I_\ell^A})} e^{4 q^2 \Var(X_{ I_\ell^A} )}  
  \\ & \leq q \left(4 + \sqrt{4 \pi  \Var(X_{ I_\ell^A})} \right) \, e^{4 q^2 \Var(X_{ I_\ell^A} )}  
  \\ & \leq q \left(4 + \frac{ 1 + 4 \pi  \Var(X_{ I_\ell^A})}{2} \right)e^{4 q^2 \Var(X_{ I_\ell^A})}
  \\ & \leq 7 q  \left( 1 +  \Var(X_{ I_\ell^A})  \right)e^{4 q^2 \Var(X_{ I_\ell^A})}.
\end{align*}
Using the estimate of the variance given in  Lemma \ref{lemma:preparatory}, we deduce 
\begin{align*} \mathbb{E} \left[e^{q |  \widetilde{X_{ I_\ell^A}} |} - 1 \right] 
& \ll  q \log ( 2 + |I_{\ell}^A| ) e^{\Oc (q^2 \log ( 2 + |I_{\ell}^A| ) )} \\ 
& =   q \log \left( 2 +  \Oc_{\alpha} \left((1 + |\ell|)^{\alpha - 1}\right) \right) 
e^{\Oc \left(q^2 \log \left( 2 + \Oc_{\alpha} \left((1 + |\ell|)^{\alpha - 1} \right) \right) \right)   } 
\\ & \ll_{\alpha} q \log (2 + |\ell|) e^{\Oc_{\alpha} (q^2 \log (2  + |\ell|))} 
= q (2 + |\ell|)^{\Oc_{\alpha}(q^2)}  \log (2 + |\ell|) 
\end{align*}
Hence, in the region where $q \leq 1/4$, we get 
\begin{align*}
\mathbb{E} \left[\sup_{z \in K} e^{ p \left| \sum_{ |y^{(n)}_k| > A } \frac{1}{z-y^{(n)}_k}\right|} \right]
& \leq e^{\Oc_{K, \alpha}\left( \frac{p}{A^{\frac{1}{\alpha}}} \right) }
+ \Oc_{\alpha} \left( q e^{\Oc_{K, \alpha}\left( \frac{p}{A^{\frac{1}{\alpha}}} \right) }
  \sum_{\ell \in \Z} \beta_\ell (2+|\ell|)^{\Oc_{\alpha}(q^2 ) }   \log (2 + |\ell|)  \right)
  \end{align*}
  The sum in $\ell$ is smaller than or equal to 
  $$ \frac{1}{Z_{A, \alpha}}
  \sum_{|\ell| \geq A^{1/\alpha} - 1} ( 1+ |\ell|)^{-\alpha} 
  (2+|\ell|)^{\Oc_{K, \alpha}(q^2)} \log (2 + |\ell|) .$$

If the exponent $\Oc_{K, \alpha}(q^2)$ is strictly smaller than $(\alpha - 1)/2$, then the terms of the last series are
bouded by  $2^{\alpha} (2 + |\ell|)^{-\beta}  \log (2 + |\ell|)$, where $\beta > 1 + (\alpha-1)/2$ and 
$\beta = \alpha - \Oc_{K, \alpha}(q^2)$. Hence, in this case, 
\begin{align*} \sum_{|\ell| \geq A^{1/\alpha} - 1} ( 1+ |\ell|)^{-\alpha} 
  (2+|\ell|)^{\Oc_{K, \alpha}(q^2)}\log (2 + |\ell|) & \ll_{\alpha}
   \sum_{\ell \geq A^{1/\alpha} - 1}  (2+ |\ell|)^{-\beta}  \log (2 + |\ell|)
   \end{align*}
   Now, since $\beta > 1$, 
   the function $x \mapsto x^{-\beta} \log x$ is nonincreasing on $[e, \infty)$. Hence, for $A \geq e^{\alpha}$, 
   \begin{align*} &  \sum_{|\ell| \geq A^{1/\alpha} - 1} ( 1+ |\ell|)^{-\alpha} 
  (2+|\ell|)^{\Oc_{K, \alpha}(q^2)} \log (2 + |\ell|) \ll_{\alpha} 
  \int_{A^{1/\alpha}}^{\infty} x^{-\beta} \log x dx
  \\  & \ll_{\alpha} \left[ \frac{x^{1- \beta} \log x}{1 - \beta} \right]_{A^{1/\alpha}}^{\infty} 
   -\frac{1}{1 - \beta} \int_{A^{1/\alpha}}^{\infty}  x^{-\beta} dx 
  \\ &  \leq \frac{1}{\beta -1} A^{(1 - \beta)/\alpha} \log (A^{1/\alpha}) + \frac{1}{( \beta-1)^2} A^{(1 - \beta)/\alpha} 
   \\ & \leq \frac{2}{\alpha -1} A^{(1 - \beta)/\alpha} \left( \frac{2}{\alpha - 1} + \log (A^{1/\alpha}) \right)
    \ll_{\alpha}   A^{(1 - \beta)/\alpha} \log A 
    \\ & = \frac{\log A}{A^{1 - (1/\alpha) + \Oc_{K, \alpha}(q^2)}}.
  \end{align*}
   Moreover,  
   $$Z_{A, \alpha} = \sum_{|\ell| \geq A^{1/\alpha} - 1} \frac{1}{(1+ |\ell|)^{\alpha}} 
   \geq \int_{A^{1/\alpha} + 1}^{\infty} u^{-\alpha}  \gg_{\alpha} \frac{1}{(1 + A^{1/\alpha})^{\alpha - 1}} 
   \gg_{\alpha} \frac{1}{A^{1 - (1/\alpha)}},$$
   if $A \geq 1$. 
The condition $\Oc_{K, \alpha}(q^2) < (\alpha - 1)/2$ is satisfied as soon as $q \ll_{K, \alpha} 1$, and 
since $q$ is dominated by $p/A^{1 - (1/\alpha)}$, as soon as 
$A  \gg_{K, \alpha} p^{\alpha/(\alpha -1)}$.
Hence, if $A \gg_{K, \alpha} 1 + p^{\alpha/(\alpha -1)}$, 
\begin{align*}
 \mathbb{E} \left[ \sup_{z \in K} e^{ p \left| \sum_{ |y^{(n)}_k| > A } \frac{1}{z-y^{(n)}_k}\right|} \right]
 & =  e^{\Oc_{K, \alpha}\left( \frac{p}{A^{\frac{1}{\alpha}}} \right) }
  + \Oc_{\alpha} \left( q e^{ \Oc_{K, \alpha}\left( \frac{p}{A^{\frac{1}{\alpha}}} \right) }
  A^{\Oc_{K, \alpha}(q^2)} \log A \right)
  \\ & = e^{\Oc_{K, \alpha}\left( \frac{p}{A^{\frac{1}{\alpha}}} \right) }
  + \Oc_{K,\alpha} \left( \frac{p \log A}{A^{1 - (1/\alpha)}}
  e^{ \Oc_{K, \alpha}\left( \frac{p}{A^{\frac{1}{\alpha}}} + \frac{p^2 \log A}{A^{2 - (2/\alpha)}}  \right)} \right).
  \end{align*}
  Let us now choose $\alpha = 2$. For $A \gg_K 1 + p^2$, we get 
  \begin{align*} \mathbb{E} \left[\sup_{z \in K} e^{ p \left| \sum_{ |y^{(n)}_k| > A } \frac{1}{z-y^{(n)}_k}\right|} \right]
  & = e^{\Oc_{K}\left( \frac{p}{\sqrt{A}} \right) }
  + \Oc_{K} \left( \frac{p \log A}{\sqrt{A}}
  e^{ \Oc_{K}\left( \frac{p}{\sqrt{A}} + \frac{p^2 \log A}{A}  \right)} \right).
  \end{align*} 
  Now, by assumption, $p/\sqrt{A} \ll_K 1$, which implies that 
  $$ e^{\Oc_{K}\left( \frac{p}{\sqrt{A}} \right) } = 1 + \Oc_{K}\left( \frac{p}{\sqrt{A}} \right)
  = \Oc_K(1)$$
  and 
  $$ \mathbb{E} \left[\sup_{z \in K} e^{ p \left| \sum_{ |y^{(n)}_k| > A } \frac{1}{z-y^{(n)}_k}\right|} \right]
   = 1 +  \Oc_{K} \left( \frac{p \log A}{\sqrt{A}} e^{ \Oc_K \left(\frac{p^2 \log A}{A} \right)} \right).$$
  Now, $\log A / A$ is nonincreasing in $A \geq e$, so if $A \geq e + p^2 \log (2 + p)$, we get 
 $$\Oc_K \left( \frac{p^2 \log A}{A} \right) \leq \Oc_K \left( \frac{p^2 \log (e + p^2 \log (2 + p))}{e + p^2 \log (2+p)} 
  \right) = \Oc_K(1), $$
  which gives 
   \begin{align*} \mathbb{E} \left[\sup_{z \in K} e^{ p \left| \sum_{ |y^{(n)}_k| > A } \frac{1}{z-y^{(n)}_k}\right|} \right]
  & =  1 +  \Oc_{K} \left( \frac{p \log A}{\sqrt{A}} \right). 
  \end{align*} 
\end{proof}

As a  consequence of the above bound we have the following estimates on the $L^p$ norms
of $\xi'_{\infty}/\xi_{\infty}$. 

\begin{proposition} \label{proposition:convergenceLp}
For any compact set $K$ of $\C \backslash \R$, and for all $p\geq1$, there 
exists an absolute constant $C_{p,K}$ such that:
$$ \forall A \geq 0, \sup_{z \in K} \E\left( \left| \sum_{ |y_k| > A } \frac{1}{z-y_k} \right|^{p}
\right)^{\frac{1}{p}} \leq C_{p,K} \frac{\log (2+A)}{ \sqrt{1+A}}$$
and in particular, 
$$  \sup_{z \in K} \E\left( \left| \sum_{ |y_k| > A } \frac{1}{z-y_k} \right|^{p}
\right) \underset{A \rightarrow \infty}{\longrightarrow} 0,$$
Moreover, $\xi'_{\infty}(z)/\xi_{\infty}(z)$ is in $L^p$ for all $z \notin \R$ and $p \geq 1$. 
\end{proposition}
\begin{proof}
For $A \geq 2$, let us define $q = \sqrt{A}/ \log A$. For $A$ large enough, $2 + q \leq A$  and then 
$$C_K ( 1 + q^2 \log(2+q)) \leq C_K \left( 1 + \frac{A}{\log A} \right),$$
which is smaller than $A$ if $A$ is large enough depending on $K$. By,
Proposition \ref{proposition:uniform_tail_control}, we deduce that there exists $D_K  \geq 2$ such that for $A
\geq D_K$, 
$$\E\left[ e^{ q \left| \sum_{ |y_k| > A } \frac{1}{z-y_k}\right| } \right]
\leq 1 +  \frac{C_K q \log A}{\sqrt{A}} = 1 + C_K.$$
Now, we have $x^p \ll_p e^x$, and then for $A \geq D_K \geq 2$, 
$$ \mathbb{E} \left[  \left| q \sum_{ |y_k| > A } \frac{1}{z-y_k}\right|^p  \right] \ll_{p, K} 1,$$
i.e. 
$$ \left(\mathbb{E} \left[  \left|  \sum_{ |y_k| > A } \frac{1}{z-y_k}\right|^p  \right] \right)^{1/p} \ll_{p,K} 
1/q = \frac{\log A}{\sqrt{A}} \ll \frac{\log (2+A)}{\sqrt{1+A}}$$
In order to remove the condition $A \geq D_K$, it sufficies, by using the Minkowski inequality, to check that 
for $A < D_K$, 
$$ \E\left( \left| \sum_{ A < |y_k| \leq D_K } \frac{1}{z-y_k} \right|^{p}
\right)^{\frac{1}{p}}  \ll_{p,K}  1.$$
Now, each term $1/(z-y_k)$ is bounded by $1/(\inf_{z \in K} |\Im z|)$, and then 
$$ \E\left( \left| \sum_{ A < |y_k| \leq D_K } \frac{1}{z-y_k} \right|^{p}
\right)^{\frac{1}{p}} \ll_{K} ||X_{[-D_K,D_K]}||_{L_p}.$$
Now, this last bound is finite since $X_{[-D_K,D_K]}$ admits exponential moments by Lemma \ref{lemma:preparatory}, 
and since it depends only on $K$ and $p$, we get the desired bound. 

The fact that $\xi'_{\infty}/\xi_{\infty}$ is in $L^p$ is immediately obtained by taking $A = 0$ and by observing 
that the restriction $|y_k| > 0$ in the sum is irrelevant, since $0$ is a.s. not a point in $\{y_k, k \in \mathbb{Z}\}$. 
\end{proof}

The preceding result allows to compute 
the moments of $\xi'_{\infty}/\xi_{\infty}$ 
by first restricting the infinite sums to the $y_k$'s between $-A$ and $A$, and then 
by letting $A \rightarrow \infty$. 
More precisely, for all fixed $z_1, z_2, \dots, z_p \notin \R$, 
$$ \forall p \geq 1, \frac{\xi_\infty'}{\xi_\infty}(z_1) \dots \frac{\xi_\infty'}{\xi_\infty}(z_p) \in L^p $$
and
$$ \E\left( \frac{\xi_\infty'}{\xi_\infty}(z_1) \dots \frac{\xi_\infty'}{\xi_\infty}(z_p) \right)
 = \lim_{A \rightarrow \infty} \E\left( \prod_{j=1}^p\left( i \pi + \sum_{|y_k| < A } 
 \frac{1}{z_j - y_k} \right) \right).$$
The last quantity can be computed thanks to the sine kernel correlation functions of order less or equal than $p$, on the segment $[-A, A]$.
We will now perform the computation of the two first moments. 
\begin{rmk}
Before proceeding we should mention that since we have been able to prove the convergence
of the rescaled logarithmic derivative of the characteristic polynomial to $\frac{\xi_\infty'}{\xi_\infty}$, we should also be able to obtain an alternative expression for the moments using the formulas in \cite{bib:CoSn08} for the moments of ratios of the logarithmic derivative of the characteristic polynomial. Although the combinatorial expressions there provide closed formulas, we do not find them easier to handle than the method we have described above. As we shall see it below, the formulas for the second moments are already very involved.

\end{rmk}

\paragraph{First moment $M_1(z), z \notin \R$:} 
\begin{align*}
M_1(z) & := \E\left( \frac{ \xi_{\infty}'}{\xi_{\infty}}(z)  \right)\\
& = i \pi + \lim_{A \rightarrow \infty} \E\left( \sum_{ |y_k| \leq A } \frac{1}{z-y_k} \right)\\
& = i \pi + \lim_{A \rightarrow \infty} \int_{[-A, A]} dy \frac{\rho_1(y)}{z-y}\\
& = i \pi\left( 1 - \sgn\left( \Im(z) \right) \right) \\
& = i 2 \pi \mathds{1}_{\left\{  \Im(z) < 0\right\}}
\end{align*}

\paragraph{Second moment $M_2(z, z'); z,z' \notin \R$:} 
Let us first assume that $z$ and $z'$ have not the same real part, in particular $z_1 \neq z_2$. One  has: 
\begin{align*}
M_2(z, z') & := \E\left( \frac{ \xi_{\infty}'}{\xi_{\infty}}(z) \frac{ \xi_{\infty}'}{\xi_{\infty}}(z') \right)\\
& = -\pi^2 + \pi^2 \left( \sgn\left( \Im(z) \right) + \sgn\left( \Im(z') \right) \right) + \E\left( \sum_{ k,l } \frac{1}{z-y_k} \frac{1}{z'-y_l}\right)\\
& = -\pi^2 +\pi^2 \left( \sgn\left( \Im(z) \right) + \sgn\left( \Im(z') \right) \right) + \lim_{A \rightarrow \infty} \E\left( \sum_{ |y_k|, |y_l| \leq A }  \frac{1}{z-y_k} \frac{1}{z'-y_l}\right)
\end{align*}
Moreover:
\begin{align*}
\E\left( \sum_{ |y_k|, |y_l| \leq A }  \frac{1}{z-y_k} \frac{1}{z'-y_l}\right)
& = \int_{[-A, A]} \frac{dy}{\left( z-y \right) \left( z'-y \right)} + \int_{[-A, A]^2} \frac{dy_1 dy_2 \left( 1- S(y_1-y_2)^2 \right)}{\left( z-y_1\right) \left( z'-y_2 \right)},
\end{align*}
where
$$ S(x) = \frac{\sin\left( \pi x \right)}{\pi x}$$
The first integral corresponds to the indices $k=l$ while the second integral corresponds to $k \neq l$. The former is handled by a partial fraction
decomposition (recall that $z \neq z'$):
$$ \lim_{A \rightarrow \infty}\int_{[-A, A]} \frac{dy}{\left( z-y \right) \left( z'-y \right)} = i\pi \frac{\sgn\left( \Im(z) \right) - \sgn\left( \Im(z') \right) }{z-z'}$$

The second integral can be written as $I_1 - I_2$, where 
$$ I_1 =  \int_{[-A, A]^2} \frac{dy_1 dy_2}{\left( z-y_1\right) \left( z'-y_2 \right)},$$
and 
$$ I_2 =  \int_{[-A, A]^2} \frac{S(y_1-y_2)^2}{\left( z-y_1\right) \left( z'-y_2 \right)} dy_1 dy_2.$$
One has immediately 
$$ \lim_{A \rightarrow \infty}  I_1 = \lim_{A \rightarrow \infty} \left( \int_{[-A, A]}  \frac{dy}{z-y} 
\right)  \left( \int_{[-A, A]} \frac{dy}{z'-y} 
\right) = - \pi^2 \sgn\left( \Im(z) \right) \sgn\left( \Im(z') \right).$$
For fixed $z$ and $z'$, the integral $I_2$ is dominated by
  \begin{align*}
  & \int_{\R^2} \frac{1}{(1 + |y_1|) (1 + |y_2|) [1+ (y_1 - y_2)^2]} dy_1 dy_2
  \\  & \leq \frac{1}{2}   \int_{\R^2} \frac{1}{1+ (y_1 - y_2)^2} \left(\frac{1}{(1 + |y_1|)^2} + \frac{1}{(1 + |y_2|)^2} \right)
    dy_1 dy_2 \\ & =\int_{\R} \frac{dy}{1+y^2} \int_{\R} \frac{du}{(1+ |u|)^2}  < \infty.
    \end{align*}
    Hence, 
    $$ \lim_{A \rightarrow \infty}  I_2 =  \int_{\R^2} \frac{S(y_1-y_2)^2}{\left( z-y_1\right) \left( z'-y_2 \right)} dy_1 dy_2,$$
    where the last integral is absolutely convergent. 
The change of variable $u = y_2$, $v = y_1 - y_2$ gives
$$  \lim_{A \rightarrow \infty}  I_2 = 
 \int_{\R} dv S(v)^2  \int_{\R} \frac{du} {(z - u - v)(z'-u)}.$$
The integral in $u$ can again be computed by a partial fraction
decomposition, and one gets
$$\int_{\R} \frac{du} {(z - u - v)(z'-u)} = i\pi \frac{\sgn\left( \Im(z) \right) - \sgn\left( \Im(z') \right) }{z-z'-v}.$$
Note that since $z$ and $z'$ are assumed to have different imaginary parts, the denominator does not vanish. 
One then has
 $$ \lim_{A \rightarrow \infty}  I_2  = i\pi \left[ \sgn\left( \Im(z) \right) - \sgn\left( \Im(z') \right) \right]
 \int_{\R} \frac{ S(v)^2 }{ z-z'-v} dv,
 $$
 where 
\begin{align*}  \int_{\R} \frac{ S(v)^2 }{ z-z'-v} dv
& = \frac{1}{4 \pi^2} \int_{ \R}   \frac{2 - e^{2i \pi v} -  e^{-2i \pi v} }{v^2 (z-z'-v)} dv.
 \\ & = \frac{1}{4 \pi^2} \int_{ \R}   \frac{1 - e^{2i \pi v} + 2i \pi v }{v^2 (z-z'-v)} dv + 
 \frac{1}{4 \pi^2} \int_{ \R}   \frac{1 - e^{- 2i \pi v} - 2 i \pi v}{v^2 (z-z'-v)} dv,
 \end{align*}
In the two last integrals, the integrands are bounded near zero and dominated by $1/v^2$ at infinity, and then the integrals are absolutely convergent. Moreover, the integrands can be extended to meromorphic functions of $v$, with the unique pole $v = z - z'$. Note that because of the addition of the terms $\pm 2i \pi v$, there is no pole at $ v =0$. 
In the first integral, if we replace $\R$ by the contour given by the union of $(-\infty, -R]$, $[-R, -R+iR]$, $[-R + iR, R + iR]$, $[R + iR, R]$ and $(R, \infty)$, the modified integral tends to zero when $R$ goes to infinity.
One deduces that the initial integral is equal to $ 2 i \pi$ times the sum of the residues of the integrand at the poles in the upper half plane: 
 $$\frac{1}{4 \pi^2} \int_{ \R}   \frac{1 - e^{2i \pi v} + 2i \pi v }{v^2 (z-z'-v)} dv 
 =  \frac{1 - e^{2i \pi (z-z')} + 2i \pi (z-z') }{ 2 i \pi (z-z')^2} \mathds{1}_{\Im (z-z') > 0}$$
Changing $v$ in $-v$ and exchanging $z$ and $z'$, we deduce 
$$  \frac{1}{4 \pi^2} \int_{ \R}   \frac{1 - e^{- 2i \pi v} - 2 i \pi v}{v^2 (z-z'-v)} dv
=- \frac{1 - e^{-2i \pi (z-z')} - 2i \pi (z-z') }{ 2 i \pi (z-z')^2} \mathds{1}_{\Im (z-z') < 0},$$
and  by adding the equalities: 
$$  \int_{\R} \frac{ S(v)^2 }{ z-z'-v} dv = 
\frac{\sgn \left(  \Im(z- z') \right) \left( 1 - e^{2 i \pi (z-z') \sgn  \left(  \Im(z- z') \right)} \right)}{2 i \pi (z-z')^2} 
+ \frac{1}{z-z'}.$$
By noting that 
$$ i \pi \left[ \sgn\left( \Im(z) \right) - \sgn\left( \Im(z') \right) \right]  \, \sgn \left(  \Im(z- z') \right) 
= 2 i \pi  \mathds{1}_{\Im(z) \Im (z') < 0},$$
we deduce 
$$ \lim_{A \rightarrow \infty}  I_2  = \frac{1 - e^{2 i \pi (z-z') \sgn (\Im(z-z')) }}{(z-z')^2} \mathds{1}_{\Im(z) \Im (z') < 0}
+ i \pi \frac{\sgn\left( \Im(z) \right) - \sgn\left( \Im(z') \right) }{z-z'}.
$$
Hence, 
\begin{align*}  \lim_{A \rightarrow \infty}  (I_1 - I_2) 
& =  - \pi^2 \sgn\left( \Im(z) \right) \sgn\left( \Im(z') \right) 
-  \frac{1 - e^{2 i \pi (z-z') \sgn (\Im(z-z')) }}{(z-z')^2} \mathds{1}_{\Im(z) \Im (z') < 0}
\\ & - i \pi \frac{\sgn\left( \Im(z) \right) - \sgn\left( \Im(z') \right) }{z-z'},
\end{align*}
 and 
 \begin{align*}
\lim_{A \rightarrow \infty} \E\left( \sum_{ |y_k|, |y_l| \leq A }  \frac{1}{z-y_k} \frac{1}{z'-y_l}\right)
&  =  - \pi^2 \sgn\left( \Im(z) \right) \sgn\left( \Im(z') \right) 
\\ & -  \frac{1 - e^{2 i \pi (z-z') \sgn (\Im(z-z')) }}{(z-z')^2} \mathds{1}_{\Im(z) \Im (z') < 0}.
\end{align*} 
Hence 
$$ M_2(z, z') = - 4 \pi^2 \mathds{1}_{\Im(z)< 0 , \Im (z') < 0} -  \frac{1 - e^{2 i \pi (z-z') \sgn (\Im(z-z')) }}{(z-z')^2} \mathds{1}_{\Im(z) \Im (z') < 0}.$$

This formula has been proven for $\Im (z) \neq \Im (z')$. It remains true without this assumption. Indeed, the $L^2$ convergence of  $i \pi +  \sum_{ |y_k| \leq A } \frac{1}{z-y_k} $ towards $\xi'(z)/\xi(z)$ for $A \rightarrow \infty$ has been proven uniformly in compact sets away from the real line. Since the joint moments of the former quantity are easily proven to be continuous, one deduces that $M_2$ is continuous with respect to $z, z' \notin \R$. 
\paragraph{Second moment with a conjugate $\tilde{M}_2(z, z'); z,z' \notin \R$:} 
Let us now define
$$ \tilde{M}_2(z, z') := 
 \E\left( \frac{ \xi_{\infty}'}{\xi_{\infty}}(z) \overline{\frac{ \xi_{\infty}'}{\xi_{\infty}}(z') }\right)$$
 Since 
 $$ \overline{\frac{ \xi_{\infty}'}{\xi_{\infty}}(z') } = - 2 i \pi + \frac{ \xi_{\infty}'}{\xi_{\infty}}(\overline{z'}) ,$$
 one gets 
 $$  \tilde{M}_2(z, z')  = M_2(z, \overline{z'} ) - 2 i \pi M_1(z),$$
 and then 
 $$  \tilde{M}_2(z, z') = 
 4 \pi^2 \mathds{1}_{\Im(z)< 0 , \Im (z') < 0} -  \frac{1 - e^{2 i \pi (z-\overline{z'}) \sgn (\Im(z-\overline{z'})) }}{(z-\overline{z'}
 )^2} \mathds{1}_{\Im(z) \Im (z') >  0}.$$
In particular, we get the $L^2$ norm: 
$$ \E\left( \left|\frac{ \xi_{\infty}'}{\xi_{\infty}}(z)\right|^2 \right) 
=4 \pi^2 \mathds{1}_{\Im(z)< 0}+ \frac{1 - e^{-4 \pi |\Im (z)|}}{4\Im^2(z)}. $$

As a consequence of the previous computation, if our conjecture is true and moments are also controlled then:
\begin{conjecture}
\begin{align*}
  & \lim_{T \rightarrow \infty } \frac{1}{\log^2 T}
    \E\left( \frac{\zeta'}{\zeta}\left( \half + i \omega T + \frac{a }{\log T} \right)
             \frac{\zeta'}{\zeta}\left( \half + i \omega T + \frac{a'}{\log T} \right) \right) \\
= & \mathds{1}_{\Re(a)<0, \Re(a')<0} - \frac{1 - e^{-\left( a'-a \right) \sgn \Re(a'-a)} }{\left( a - a' \right)^2} \mathds{1}_{\Re(a)\Re(a')<0} \\
  & \lim_{T \rightarrow \infty } \frac{1}{\log^2 T}
    \E\left( \frac{\zeta'}{\zeta}\left( \half + i \omega T + \frac{a }{\log T} \right)
  \overline{ \frac{\zeta'}{\zeta}\left( \half + i \omega T + \frac{a'}{\log T} \right) } \right) \\
= & \mathds{1}_{\Re(a)<0, \Re(a')<0} + \frac{1 - e^{-\left( a+\overline{a'} \right) \sgn \Re( a+\overline{a'} )} }{\left( a + \overline{a'} \right)^2} \mathds{1}_{\Re(a)\Re(a')>0}
\end{align*}
\end{conjecture}
\begin{rmk}
In Lemma \ref{lemma33}, we see that there is 
a correspondance between $a$ and $-2 i \pi z$
in this conjecture and the computations 
just above. 
This explains the signs of the terms involved in the conjecture, and the fact  the imaginary
parts of $z$ and $z'$ are replaced by the 
real parts of $a$ and 
$a'$.  
\end{rmk}
For $a=a'$, one recovers the first statement of theorem 3 in \cite{bib:GGM01}, which is equivalent to the pair correlation conjecture under Riemann hypothesis. Higher
moments formulas are also expected to be equivalent to the convergence of higher correlation functions of $\zeta$ zeros towards the corresponding correlations for the sine-kernel process.

\section{The moments of ratios related to \texorpdfstring{$\xi_\infty$}{our limiting function}}\label{lenotre}

\subsection{Expectation of ratios}
For $z \in \C$, the random variable $\xi_{\infty}(z)$ has no moment of order $1$. However, if we consider
the ratio of products of values of $\xi_{\infty}$ at points outside the real axis, and if there are the same 
number of factors in the numerator and in the denominator, then the ratio is integrable. 
This result is a consequence of the following theorem: 

\begin{thm}
\label{thm:multi_sup_is_finite}
For any $p > 0$ and any compact set $K \subset \C \backslash \R$, we have:
$$ \sup_{n \in \N \sqcup \{ \infty \}} \E\left( \sup_{ \left(z,z'\right) \in K^2}
\left| \frac{\xi_n(z')}{\xi_n(z)} \right|^p \right) < \infty$$
\end{thm}

\begin{proof}
Let $\left( z,z' \right) \in K^2$. Without loss of generality, one can enlarge the compact set $K$ to 
a compact that is symmetric with respect to the real line, and whose part above the real line is convex. 
Using the functional equation \eqref{eq:functional_eq_xi_n} if
necessary, we can then assume that $z$ and $z'$ are both in the upper half-plane. 

 Since this part of $K$ is supposed to be convex, 
$[z, z'] \subset K$. Therefore, the segment $[z, z']$ does not cut the real line, where zeros lie. Hence for $n
\in \N \sqcup \{ \infty \}$:
$$ \left| \frac{\xi_n(z')}{\xi_n(z)} \right|^p \leq \left|
\exp\left( p \Re\left( \int_{[z,z']} \frac{\xi_n'}{\xi_n} \right) \right) \right| \leq e^{ p|z-z'| \sup_{u \in K} \left| \frac{\xi_n'}{\xi_n}(u) \right| }$$
By absorbing the quantity $|z-z'| = \Oc_K\left( 1 \right)$ in the exponent $p$, we only have to 
prove that for all $p>0$:
$$             \sup_{n \in \N \sqcup \{ \infty \}}
\E\left( \sup_{z \in K} e^{p \left| \sum_{ k \in \Z } \frac{1}{z-y^{(n)}_k} \right|} \right)
 = \sup_{n \in \N \sqcup \{ \infty \}} \E\left( \sup_{z \in K} e^{p \left| \frac{\xi_n'}{\xi_n}(z) 
 \right|} \right) < \infty.$$
 By Proposition \ref{proposition:uniform_tail_control}, we know that for $A :=  C_K (1 + 4p^2 \log (2 + 2p))$, 
 $$ \sup_{n \in \N \sqcup \{ \infty \}}
\E\left( \sup_{z \in K} e^{2p \left| \sum_{ |y^{(n)}_k| > A} \frac{1}{z-y^{(n)}_k} \right|} \right) < \infty.$$
By the Cauchy-Schwarz inequality, it is then sufficient to check that 
$$ \sup_{n \in \N \sqcup \{ \infty \}}
\E\left( \sup_{z \in K} e^{2p \left| \sum_{ |y^{(n)}_k| \leq A  } \frac{1}{z-y^{(n)}_k} \right|} \right) < \infty.$$
Now, 
$$  \left| \sum_{ |y^{(n)}_k| \leq A } \frac{1}{z-y^{(n)}_k}\right|
 \leq X_{\left[-A, A\right]} \sup_{z \in K, t \in \R} \frac{1}{\left| z-t \right|}
 \leq \frac{1}{\inf_{z \in K} |\Im z|} X_{\left[-A, A\right]}$$
and all the exponential moments of this last variable are finite, thanks to Lemma \ref{lemma:preparatory}.
\end{proof}
From this bound, we are able to deduce the following  convergence result: 
\begin{proposition} \label{proposition:convergencemomentsratios}
For $z_1, \dots, z_k, z_1', \dots, z_k' \in \C \backslash \R$, and for all $n \in \mathbb{N} \sqcup \{\infty\}$,
$$\E\left( \prod_{j=1}^k \left |\frac{\xi_n(z_j')}{\xi_n(z_j)}\right |\right) < \infty$$
Moreover, for every compact set $K$ in $\C \backslash \R$,
we have the following convergence, uniformly in $z_1, z_2, \dots, z_k, z'_1, \dots, z'_k \in K$: 
$$  \E\left( \prod_{j=1}^k \frac{\xi_n     (z_j')}{\xi_n     (z_j)}\right)
\underset{n \rightarrow \infty}{\longrightarrow} 
\E\left( \prod_{j=1}^k \frac{\xi_\infty(z_j')}{\xi_\infty(z_j)}\right).$$                                                           
\end{proposition}
\begin{proof}
The finiteness of the expectation is a direct consequence of Theorem \ref{thm:multi_sup_is_finite} and 
the H\"older inequality. 
The convergence we want to prove can be written as follows: 
$$ 
\sup_{z_1, z_2, \dots, z_k, z'_1, \dots z'_k \in K} \left| \E\left[  \prod_{j=1}^k \frac{\xi_n   
(z_j')}{\xi_n     (z_j)} - 
  \prod_{j=1}^k \frac{\xi_{\infty}   (z_j')}{\xi_{\infty}    (z_j)} \right] \right| 
 \underset{n \rightarrow \infty}{\longrightarrow} 0,$$
 which is implied by 
 $$ \E\left[ \sup_{z_1, z_2, \dots, z_k, z'_1, \dots z'_k \in K} 
 \left|  \prod_{j=1}^k \frac{\xi_n     (z_j')}{\xi_n     (z_j)} - 
  \prod_{j=1}^k \frac{\xi_{\infty}   (z_j')}{\xi_{\infty}    (z_j)} \right| \right] 
  \underset{n \rightarrow \infty}{\longrightarrow} 0.$$
  Now, we have 
  \begin{align*} 
\left|  \prod_{j=1}^k \frac{\xi_n     (z_j')}{\xi_n     (z_j)} - 
  \prod_{j=1}^k \frac{\xi_{\infty}   (z_j')}{\xi_{\infty}    (z_j)} \right|
  & \leq \sum_{m=1}^k \prod_{1 \leq j < m} \left|\frac{\xi_n     (z_j')}{\xi_n     (z_j)}
  \right| \prod_{m < j \leq  k} \left|\frac{\xi_{\infty}    (z_j')}{\xi_{\infty}     (z_j)} \right| 
  \, \left| \frac{\xi_n     (z_m')}{\xi_n     (z_m)} - \frac{\xi_{\infty}     (z_m')}{\xi_{\infty}    (z_m)}
  \right|
  \end{align*} 
  It is then sufficient to show, for $1 \leq m \leq k$, 
  $$ \E \left[ \left( \sup_{z, z' \in K} \left|\frac{\xi_n     (z')}{\xi_n     (z)} \right|
  \right)^{m-1}  \left( \sup_{z, z' \in K} \left|\frac{\xi_{\infty}     (z')}{\xi_{\infty}     (z)} \right|
  \right)^{k-m} \sup_{z, z' \in K} 
  \left| \frac{\xi_n     (z')}{\xi_n     (z)} - \frac{\xi_{\infty}     (z')}{\xi_{\infty}    (z)}
  \right| \right]  \underset{n \rightarrow \infty}{\longrightarrow} 0,$$
  which is implied (after two applications of the Cauchy-Schwarz inequality) by 
  $$\sup_{r \in \mathbb{N} \sqcup \infty} 
  \E \left[ \left( \sup_{z, z' \in K} \left|\frac{\xi_r     (z')}{\xi_r    (z)} \right|
  \right)^{4m-4} \right]^{1/4}  
  \sup_{r \in \mathbb{N} \sqcup \infty} 
  \E \left[ \left( \sup_{z, z' \in K} \left|\frac{\xi_r     (z')}{\xi_r    (z)} \right|
  \right)^{4k-4m} \right]^{1/4} $$ $$ \times
  \E \left[ \sup_{z, z' \in K}  \left| \frac{\xi_n     (z')}{\xi_n     (z)} - \frac{\xi_{\infty}     (z')}{\xi_{\infty}    (z)}
  \right|^2 \right]^{1/2}  \underset{n \rightarrow \infty}{\longrightarrow} 0.$$
  From Theorem \ref{thm:multi_sup_is_finite}, it is then sufficient to show that
  $$ \E \left[ \sup_{z, z' \in K}  \left| \frac{\xi_n     (z')}{\xi_n     (z)} - \frac{\xi_{\infty}     (z')}{\xi_{\infty}    (z)}
  \right|^2 \right] \underset{n \rightarrow \infty}{\longrightarrow} 0.$$
  Now, 
  \begin{align*} 
\sup_{z, z' \in K} \left| \frac{\xi_n     (z')}{\xi_n     (z)} - \frac{\xi_{\infty}     (z')}{\xi_{\infty}    (z)} \right| 
& = \sup_{z, z' \in K} \left| \frac{\xi_n(z') [ \xi_{\infty} (z) - \xi_n (z) ] + \xi_n(z) [ \xi_n(z') - \xi_{\infty}(z')]}
{\xi_n(z) \xi_{\infty}(z)}  \right|
\\ & \leq
\frac{ 2 \left(\sup_{r \in \mathbb{N} \sqcup \{\infty\} }\sup_{z \in K}  |\xi_r(z)| \right)
 \left(\sup_{z \in K} |\xi_n (z) - \xi_{\infty}(z)| \right)}{ \inf_{r \in \mathbb{N} \sqcup \{\infty\} }
 \inf_{z \in K} |\xi_r (z)|^2}.
\end{align*}
Almost surely, $\xi_{n}$ converges uniformly to $\xi_{\infty}$ on $K$. Hence, the numerator of the last fraction 
converges to 
zero when $n$ goes to infinity. On the other hand, since $\xi_{\infty}$ does not vanish on $K$ (all its zeros 
as real), its infimum $a$ on $K$ is strictly positive. By the uniform convergence
of $(\xi_r)_{r \geq 1}$ towards $\xi_{\infty}$, 
there exists $r_0 \geq 1$ such that $r \geq r_0$ implies $\inf_K |\xi_r|  \geq a/2$. 
Moreover, since $\xi_r$ also has only real zeros, $\inf_K |\xi_r| > 0 $ for all $r < r_0$. We deduce 
that the denominator of the fraction above is strictly positive. Since it does not depend on $n$, whereas 
the numerator goes to zero, we get almost surely: 
$$ \sup_{z, z' \in K} \left| \frac{\xi_n     (z')}{\xi_n     (z)}
- \frac{\xi_{\infty}     (z')}{\xi_{\infty}    (z)} \right|^2 
 \underset{n \rightarrow \infty}{\longrightarrow} 0.$$
By dominated convergence, for all $B > 0$, 
$$ \mathbb{E} \left[ B \wedge \sup_{z, z' \in K} \left| \frac{\xi_n     (z')}{\xi_n     (z)}
- \frac{\xi_{\infty}     (z')}{\xi_{\infty}    (z)} \right|^2 \right]\underset{n \rightarrow \infty}{\longrightarrow} 0.$$
and then 
\begin{align*}
&\underset{n \rightarrow \infty}{\lim\sup} \,
 \mathbb{E} \left[  \sup_{z, z' \in K} \left| \frac{\xi_n     (z')}{\xi_n     (z)}
- \frac{\xi_{\infty}     (z')}{\xi_{\infty}    (z)} \right|^2 \right]
\\ & \leq \underset{n \rightarrow \infty}{\lim\sup} \,
 \mathbb{E} \left[ \mathds{1}_{\sup_{z,z' \in K} \left| \frac{\xi_n     (z')}{\xi_n     (z)}
- \frac{\xi_{\infty}     (z')}{\xi_{\infty}    (z)} \right|^2 \geq B} \sup_{z, z' \in K} \left| \frac{\xi_n     (z')}{\xi_n     (z)}
- \frac{\xi_{\infty}     (z')}{\xi_{\infty}    (z)} \right|^2 \right]
\\ & \leq \frac{1}{B} \underset{n \rightarrow \infty}{\lim\sup} \,
 \mathbb{E} \left[  \sup_{z, z' \in K} \left| \frac{\xi_n     (z')}{\xi_n     (z)}
- \frac{\xi_{\infty}     (z')}{\xi_{\infty}    (z)} \right|^4 \right]
\\  & \leq \frac{1}{B}\sup_{n \in \mathbb{N}} \mathbb{E} \left[  \left(\sup_{z, z' \in K} \left| \frac{\xi_n     (z')}{\xi_n     (z)}
\right| + \sup_{z, z' \in K} \left| \frac{\xi_{\infty}     (z')}{\xi_{\infty}    (z)} \right| \right)^4 \right] 
\\ & \leq \frac{16}{B} \sup_{n \in \mathbb{N} \sqcup \{\infty\}}
 \mathbb{E} \left[  \sup_{z, z' \in K} \left| \frac{\xi_n     (z')}{\xi_n     (z)}\right|^4 \right]
\end{align*}
By Theorem \ref{thm:multi_sup_is_finite}, the last quantity is $\Oc_K(1/B)$. 
Since $B$ can be chosen arbitrarily large, we get
$$ \mathbb{E} \left[  \sup_{z, z' \in K} \left| \frac{\xi_n     (z')}{\xi_n     (z)}
- \frac{\xi_{\infty}     (z')}{\xi_{\infty}    (z)} \right|^2 \right] 
\underset{n \rightarrow \infty}{\longrightarrow} 0.$$
\end{proof}
Now, the joint moments of ratios of $\xi_n$ can be explicitly computed, by using tools
given by Borodin, Olshanski and Strahov. In \cite{bib:BoStr} and \cite{bib:BOS}, they established that certain 
determinantal formulas for ratios of characteristic polynomials are equivalent to 
a certain property regarding the underlying point process of zeros\footnote{The authors are grateful to
Brad Rodgers for many insightful discussions on the subject}. This property was 
named \textit{Giambelli compatibility} (equation 0.2 in \cite{bib:BOS}). We are now concerned
with a particular case of that general framework. Consider a point process
$\Lambda =  \Lambda_n = \left( \lambda_1, \lambda_2, \dots, \lambda_n \right)$ of $n$-point configurations 
in $\C$. We assume that the underlying probability distribution is of the form:
\begin{eqnarray}
\label{eq:point_process}
\P\left( \Lambda \in dx \right) = \frac{1}{C_n} \left| \Delta(x) \right|^2 \prod_{i=1}^n \alpha(dx_i)
\end{eqnarray}
where $\Delta(x) = \prod_{1 \leq i < j \leq n} \left( x_i - x_j \right)$ is the Vandermonde determinant, $\alpha$ is a
reference measure on $\C$ whose moments are all finite, and $C_n$ is a normalisation constant.
We then have the following result: 
\begin{thm}
\label{thm:BOS_ratio_formula}
If for $u \in \C$ we note
$$  D(u) = \prod_{i=1}^n\left( u - \lambda_i \right),$$
then the following formal identity holds for all $k \geq 1$, 
\begin{eqnarray}
\label{eq:BOS_formula}
\det\left( \frac{1}{u_i-v_j} \right)
\E\left( \prod_{j=1}^k \frac{D(v_j)}{D(u_j)} \right) = 
\det\left( \frac{1}{u_i-v_j} \E\left( \frac{D(v_j)}{D(u_i)} \right) \right)_{i,j=1}^k. 
\end{eqnarray}
This identity has to be understood as follows. Writing 
$$\frac{1}{u_j - \lambda_i} = \sum_{m = 1}^{\infty} \frac{\lambda_i^m}{u_j^{m+1}},$$
we deduce an expression of $\prod_{j=1}^k \frac{D(v_j)}{D(u_j)}$ and $\frac{D(v_j)}{D(u_i)}$ as multivariate power series 
in the variables $u_1, \dots, u_k, v_1, \dots, v_k$ for which all the nonnegative exponents are bounded by $n$, and whose 
coefficients are polynomial functions of $\lambda_1, \dots, \lambda_n$. 
The fact that the moments of $\alpha$ are all finite implies that one can take, term by term, the 
expectation of these power series. The two sides of \eqref{eq:BOS_formula} can then both be written as 
power series in $u_1, \dots, u_k, v_1, \dots, v_k$ with exponents bounded from above, divided by 
$\prod_{1 \leq i, j \leq k} (u_i - v_j)$. The formula  \eqref{eq:BOS_formula} says that these two power series 
coincide. 
\end{thm}
\begin{proof}[Pointers to the proof]
This result is proven in \cite{bib:BOS}, up to small changes. Comparing our notation with \cite{bib:BOS}, 
we have
$$D(v_j) = v_j^{n} E(-v_j), \; [D(u_j)]^{-1} = u_j^{-n} H(u_j),$$
and then we deduce immediately our formula from Proposition 2.2. of \cite{bib:BOS}, by changing the sign of 
the $v_j$'s and by multiplying both sides by 
$\prod_{j=1}^{k} (v_j/u_j)^{n}$. Note that this multiplication is the reason why we allow here nonnegative 
exponents up to $n$ in the formal power series. 
In \cite{bib:BOS}, the results are proved  for $\alpha$ carried by $\R$, however, they
can  immediately be extended to $\C$: the only change occurs in the proof of Theorem 3.1 in \cite{bib:BOS}. We
have to replace $A_{\lambda_i + N - i + N -j}$ by $A_{\lambda_i + N - i, N -j}$, where 
$$A_{p,q} := \int_{\C} x^p \bar{x}^q \alpha(dx),$$
the conjugate coming from the fact that the joint density of $\Lambda$ involves  $|  \Delta(x) |^2$ 
instead of $(\Delta(x))^2$. 
\end{proof}

We specialise $\alpha$ to be the Lebesgue measure on the circle $S^1 = \left\{ |z| = 1 \right\}$ and, thanks to the Weyl integration formula, equation \eqref{eq:point_process} becomes the density of eigenvalues for the CUE. The random vector $\Lambda^{(n)}$ can therefore be seen as the zeros of the characteristic polynomial $Z_n$ (equation \eqref{eq:def_Z_n} ). The following corollary is intuitive, although the proof requires some care in passing from a statement on formal power series to a statement on actual analytic functions:
\begin{thm}
\label{thm:ratio_formula_xi_n}
For $\left( z_1, \dots, z_k \right) \in \left(\C \backslash \R\right)^k$ and
$\left( z_1', \dots, z_k' \right) \in \C^k $, 
such that for $1 \leq i, j \leq k$, $z_i - z'_j$ is not an integer multiple of $n$, 
\begin{eqnarray}
\label{eq:ratio_formula_xi_n}
  \det\left( \frac{1}{e^{\frac{i 2 \pi z_i}{n}}-e^{\frac{i 2 \pi z_j'}{n}}} \right)_{i,j= 1}^k
  \E\left( \prod_{j=1}^k \frac{\xi_n(z_j')}{\xi_n(z_j)} \right)
=
  \det\left( \frac{1}{e^{\frac{i 2 \pi z_i}{n}}-e^{\frac{i 2 \pi z_j'}{n}}} \E\left( \frac{\xi_n(z_j')}{\xi_n(z_i)} \right) \right)_{i,j=1}^k 
\end{eqnarray}
and moreover:
\begin{eqnarray}
\label{eq:ratio_formula_xi_n_bis}
\E\left( \frac{\xi_n(z')}{\xi_n(z)} \right) = \left\{\begin{array}{cc}
                                                    1              & \textrm{if } \Im(z)>0 \\
                                                    e^{i2\pi(z'-z)} & \textrm{if } \Im(z)<0 \\
                                                    \end{array}\right.
\end{eqnarray}
\end{thm}
\begin{proof}
Recall that $$\xi_n\left( z \right) = \frac{Z_n(e^{2 i \pi z/n})}{Z_n(1)} 
= \frac{D(e^{2 i \pi z/n})}{D(1)}.$$
When forming a ratio, simplifications occur and give:
$$ \forall 1 \leq j \leq k, \; \frac{\xi_n(z_j')}{\xi_n(z_j)}  = 
\frac{D(e^{\frac{i 2 \pi z_j'}{n}})}{D(e^{\frac{i 2 \pi z_j}{n}})}, $$
Now we set $\left( u, v \right) \in \C^k \times \C^k$ such that $v_j = e^{\frac{i 2 \pi z_j'}{n}}$
and $u_j = e^{\frac{i 2 \pi z_j}{n}}$. 
The result we have to prove is equivalent to the following: 
the equation \eqref{eq:BOS_formula} holds as an equality of complex numbers for all 
$\left( u_1, \dots, u_k, v_1, \dots v_k \right) \in \left( \C \backslash S^1 \right)^k \times \C^k$ such that
$u_i \neq v_j$ for all $i, j \in \{1, \dots, k\}$. 

Now, all the computations in Theorem \ref{thm:BOS_ratio_formula}, implicitly needed in order to write an equality of formal 
series divided by $\prod_{1 \leq i,j \leq n} (u_i - v_j)$, can be translated to get an equality 
of complex numbers, provided that the formal series converge absolutely and that the denominator does not vanish. 
This last condition is satisfied since we assume $u_i \neq v_j$ for all $i, j \in \{1, \dots, k\}$. 
Now, if $|u_i| > 1$ for all $i$, each term of the power series corresponding to 
$D(v_j)/D(u_i)$ is dominated by the corresponding term of the power series
$$ \prod_{m = 1}^n (|v_j| + |\lambda_m|) \prod_{m = 1}^n \left( \sum_{p = 0}^{\infty} \frac{|\lambda_m|^p}{ |u_i|^{p+1}} 
\right) = \left( \frac{|v_j| + 1}{|u_i| - 1} \right)^n.$$
We deduce that the power series involved in the left-hand side of \eqref{eq:BOS_formula} after removing the denominator 
 $\prod_{1 \leq i,j \leq n} (u_i - v_j)$ is term by term majorized by the series corresponding to 
 $$\prod_{j = 1}^k \left( \frac{|v_j| + 1}{|u_j| - 1} \right)^n
 \sum_{ \sigma \in \mathfrak{S}_k} \prod_{1 \leq i, j \leq k, j \neq \sigma(i)} (|u_i| + |v_j|),$$
 which is convergent since this quantity is finite. 
  Similarly, the series in the right-hand side is bounded by 
  $$ \sum_{ \sigma \in \mathfrak{S}_k} \prod_{1 \leq i, j \leq k, j \neq \sigma(i)} (|u_i| + |v_j|)
  \prod_{i=1}^k \left( \frac{|v_{\sigma(i)}| + 1}{|u_i| - 1} \right)^n.$$
  Hence, we have proven that \eqref{eq:BOS_formula} holds under the assumption that $u_i$ is outside the 
  unit disc for all $i$. 
Now to extend the result to inside the circle, we shall remove an arc from the circle 
so that our working domain becomes connected.\footnote{This idea was suggested to us by Brad Rodgers} Going back to the definition of the 
distribution of $\Lambda^{(n)}$, assume that $\alpha$ is not the uniform measure on the circle, but the
uniform mesure on $\Dc_\varepsilon := \left\{ z \in \C \ | \ |z| = 1, |z-1| \geq \varepsilon \right\}$ for 
a certain $\varepsilon \in (0,1)$. In that setting, equation \eqref{eq:BOS_formula} still occurs for $|u_i| > 1$, since 
our proof is available as soon as the measure $\alpha$ is supported by the unit circle. 
Now, the expectations involved in \eqref{eq:BOS_formula} are integrals, with respect to the distribution of 
$\Lambda^{(n)}$, of rational functions of $(u_1, \dots, u_k, v_1, \dots, v_k, \lambda_1, \dots, \lambda_n)$.
If $u_1, \dots, u_k$ are in a compact set $K_1$ of $\C \backslash \Dc_\varepsilon$, and $v_1, \dots, v_k$ are 
in a compact set $K_2$ of $\C$, then these rational functions are bounded by a quantity depending only on 
$K_1$ and $K_2$, since almost surely on the law of $\Lambda^{(n)}$, 
$$\left| \frac{D(v_j)}{D(u_i)} \right| = \prod_{m = 1}^n \left| \frac{v_j - \lambda_m}{u_i - \lambda_m}
\right| \leq \left( \frac{1 + \sup_{v \in K_2} |v|}{\operatorname{dist} (K_1, \Dc_\varepsilon )} \right)^n < \infty.$$
Hence, using dominated convergence, one deduces that  the expectations in \eqref{eq:BOS_formula} are holomorphic
functions of 
$(u_1, \dots, u_k, v_1, \dots, v_k)$ on $ \left( \C \backslash \Dc_\varepsilon \right)^k \times \C^k $.
Hence, the two sides of \eqref{eq:BOS_formula} can be written as quotients by $\prod_{1 \leq i,j \leq k}
(u_i - v_j)$ of holomorphic functions. 
Since these holomorphic functions coincide on $ \left( \{z \in \C, |z| > 1 \}\right)^k \times \C^k $, and 
$ \left( \C \backslash \Dc_\varepsilon \right)^k \times \C^k $ is connected, they coincide 
on $ \left( \C \backslash \Dc_\varepsilon \right)^k \times \C^k $, and in particular, 
\eqref{eq:BOS_formula} holds for all $u_1, \dots, u_k \in \C \backslash S^1$ and $v_1, \dots, v_k \in 
\C$. 
Now, if $u_1, \dots, u_k \in \C \backslash S^1$, $v_1, \dots, v_k \in 
\C$ are fixed, the left-hand side of \eqref{eq:BOS_formula} is the integral, with respect 
to the law of $\Lambda^{(n)}$, of a continuous, bounded function of 
$(\lambda_1, \dots, \lambda_n) \in (S^1)^k $, and the right-hand side is a linear combination of 
products of such integrals. Hence, the two sides of \eqref{eq:BOS_formula} are continuous with respect to 
the law of $\Lambda^{(n)}$. Now, it is easy to check that the law of $\Lambda^{(n)}$ for 
$\alpha$ uniform on  $\Dc_\varepsilon$ tends to the law for $\alpha$ uniform on $S^1$ when 
$\epsilon$ goes to zero. Hence, since  \eqref{eq:BOS_formula} holds for $\alpha$ uniform on 
$\Dc_\varepsilon$, it also occurs for $\alpha$ uniform on $S^1$.

It remains to prove \eqref{eq:ratio_formula_xi_n_bis}. Using the change of variables $u = e^{2i \pi z/n}$ and 
$v = e^{2i \pi z'/n}$, we have to check 
\begin{equation} \mathbb{E} \left( \frac{D(v)}{D(u)} \right) =  \left\{\begin{array}{cc}
                                                    1              & \textrm{if } |u|<1 \\
                                                    (v/u)^n & \textrm{if } |u| > 1 \\
                                                    \end{array}\right. \label{eq:ratio_formula_xi_n_bis_2}
\end{equation} 
 If $|u| < 1$, we can write 
 $$\frac{D(v)}{D(u)} = \prod_{m = 1}^n \frac{1 - v \lambda_m^{-1}}{1 - u \lambda_m^{-1}} 
 = \prod_{m = 1}^n (1 - v \lambda_m^{-1}) \left(\sum_{p = 0}^{\infty}  u^p \lambda_m^{-p} \right).$$
 If we expand this expression as a power series in $u$ and $v$ with polynomial coefficients 
 in $\lambda_1^{-1}, \dots, \lambda_n^{-1}$, this series is term by term dominated by 
 $$\left[ (1 + |v|)  \left(\sum_{p = 0}^{\infty}  |u|^p \right)  \right]^n < \infty.$$
 Hence, the expectation of $D(v)/D(u)$ can be obtained by adding the expectations of each term of the 
 corresponding power series. For all nonnegative integers $p, q \geq 0$, the term in 
 $u^p v^q$ is a polynomial in $\lambda_1^{-1}, \dots, \lambda_n^{-1}$ with total degree $p+q$. 
 Now, the law of $\Lambda^{(n)}$ remains invariant if we multiply $(\lambda_1^{-1}, \dots, \lambda_n^{-1})$ 
 by any $z \in S^1$, and then the expectation of the term in $u^p v^q$ is invariant by 
 multiplication by $z^{-p-q}$, which implies that it is zero for all $(p,q) \neq (0,0)$. 
 Hence the expectation of $D(v)/D(u)$ is equal to the constant term of the corresponding series, which is 
 equal to $1$, and then we get  \eqref{eq:ratio_formula_xi_n_bis_2} for $|u| < 1$. 
 The case $|u| > 1$, $v \neq 0$ is the deduced as follows: we have
 $$ D(u) = \prod_{m=1}^n (u - \lambda_m) = (-u)^{n} \left( \prod_{m=1}^n \lambda_m \right) 
 \prod_{m=1}^n  (u^{-1} - \overline{\lambda_m}) = (-u)^{n} \left( \prod_{m=1}^n \lambda_m \right) 
 \overline{D(\bar{u}^{-1})},$$
 $$\frac{D(v)}{D(u)} = (v/u)^n \frac{\overline{D(\bar{v}^{-1})}}{\overline{D(\bar{u}^{-1})}},$$
 and then, since $|\bar{u}^{-1}| <1$, 
 $$ \mathbb{E} \left( \frac{D(v)}{D(u)} \right) = (v/u)^n.$$
Using dominated convergence, it is easy to check that $\mathbb{E} [D(v)/D(u)]$ is continuous with respect to 
$(u,v) \in ( \C \backslash S^1 ) \times \C$, which allows to extend \eqref{eq:ratio_formula_xi_n_bis_2}
to the case $|u| > 1$, $v = 0$.

\end{proof}

As Borodin, Olshanski and Strahov have in fact noticed, taking the limit for $n$ going to infinity is meaningful. 
Here, we can go further since we have now constructed the limiting object $\xi_{\infty}$. 
More precisely, using the convergence proven in Proposition \ref{proposition:convergencemomentsratios}, we 
easily get the following: 
\begin{thm}[Ratio formula]\label{formuleratios}
\label{thm:ratio_formula_xi_infty}
For all $z_1, \dots, z_k, z'_1, \dots, z'_k \in \C \backslash \R$ such that 
$z_i \neq z'_j$ for $1 \leq i, j \leq n$, we have
$$
\det\left( \frac{1}{z_i-z_j'} \right)_{i,j = 1}^{k}\E\left( \prod_{j=1}^k \frac{\xi_\infty(z_j')}{\xi_\infty(z_j)} \right) =
 \det\left( \frac{1}{z_i-z_j'} 
\E\left( \frac{\xi_\infty(z_j')}{\xi_\infty(z_i)} \right) \right)_{i,j=1}^k
$$
and moreover:
$$
\E\left( \frac{\xi_\infty(z')}{\xi_\infty(z)} \right)
= \left\{\begin{array}{cc}
	  1              & \textrm{if } \Im(z)>0 \\
	  e^{i2\pi(z'-z)} & \textrm{if } \Im(z)<0 \\
	  \end{array}\right.
$$
\end{thm}
 The condition $z_i \neq z'_j$ is not really restrictive, since for $z_i = z'_j$, the ratio inside the expectation 
 can be immediatly simplified by removing the factor $\xi_{\infty}(z_i) = \xi_{\infty}(z'_j)$ in the
 numerator and the denominator. If the $z_i$ and the $z'_j$ are all pairwise distinct, we can 
 divide by the Cauchy determinant in the left-hand side, in order to get the joint moment 
 of ratios $\xi_{\infty}(z'_j)/ \xi_{\infty}(z_j)$. If some of the $z_i$ or some of the $z'_j$ are equal, 
 the Cauchy determinant is zero, so the ratio formula does not give the moment directly: however, 
 the moment can be recovered from the fact that it is continuous with respect to 
 $z_1, \dots, z_k, z'_1, \dots, z'_k \notin \R$, this property of continuity coming from the uniformity 
 of the convergence in Proposition \ref{proposition:convergencemomentsratios}.
The joint moments of ratios of the form $\frac{\xi_\infty(z')}{\xi_\infty(z)}$ and conjugates of such ratios
can then be easily deduced from the following: 
$$\overline{\xi_{\infty} (z)} = e^{-i \pi \bar{z}} \prod_{k \in \mathbb{Z}} \left( 1 - \frac{\bar{z}}{y_k} \right)
= e^{-2 i \pi \bar{z}} \left[ e^{i \pi \bar{z}} \prod_{k \in \mathbb{Z}} \left( 1 - \frac{\bar{z}}{y_k} \right) 
\right] = e^{- 2i \pi \bar{z}} \xi_{\infty} (\bar{z}).$$
In this way, we get for all $z, z' \notin \R$, 
$$\mathbb{E} \left[ \left| \frac{\xi_{\infty}(z')}{\xi_{\infty}(z)} \right|^2 \right] 
= e^{- 4 \pi \Im (z'-z) \mathds{1}_{\Im(z) < 0}} \left( 1 + (1 - e^{-4 \pi \Im(z') \sgn (\Im(z))} ) \frac{|z-z'|^2}
{4 \Im(z) \Im(z')} \right).$$
Given Conjecture \ref{conjecture}, it is natural to expect the following: 

\begin{conjecture}
Let $\omega$ be a uniform random variable on $[0, 1]$ and $T>0$ a real parameter going to infinity. 
Then, for all $z_1, \dots, z_k, z'_1, \dots, z'_k \in \C \backslash \R$, 
such that $z_i \neq z'_j$ for all $i, j$, 
\begin{align*}& \mathbb{E}\left( \prod_{j=1}^k
\frac{ \zeta\left( \half + i T \omega - \frac{i 2\pi z'_j}{\log T} \right) }
{ \zeta\left( \half + i T \omega  - \frac{i 2\pi z_j}{\log T} \right)  } \right)
  \\ & \stackrel{T \rightarrow \infty}{\longrightarrow} 
   \det\left( \frac{1}{z_i-z_j'} \right)^{-1} \det\left( \frac{\mathds{1}_{\Im (z_i) > 0} + e^{2i \pi (z'_j - z_i)}
  \mathds{1}_{\Im (z_i) < 0} } {z_i-z_j'} \right)_{i,j=1}^k,
 \end{align*}  
 where the last expression is well-defined where the $z_i$ and the $z'_j$ are all distinct, and is 
 extended by continuity to the case where some of the $z_i$ or some of the $z'_j$ are equal. 
 \end{conjecture}
\begin{rmk}
	In a recent work \cite{Rod15}, Rodgers has shown that the GUE conjectures and the Riemann hypothesis imply the above conjecture. 
\end{rmk}

\subsection{Moments of the logarithmic derivative}

We have seen in Section \ref{subsection:log_der} how to compute the expectation of products of the logarithmic derivative of the characteristic polynomial evaluated at different points at the microscopic scale. In particular, it appeared that this method is hard to exploit when one considers a product with three or more factors. On the other hand, one might try to compute such expectations using the ratios formula \eqref{formuleratios}. Indeed, it follows from Theorem \ref*{thm:multi_sup_is_finite} that we can differentiate $\E\left( \prod_{j=1}^k \frac{\xi_\infty(z_j')}{\xi_\infty(z_j)} \right) $ with respect to $z_j'$ and then take $z_j'=z_j$, with the $z_j'$'s all distinct. As an application one can see that the formulas given  Section \ref{subsection:log_der} can be  obtained with this method in a much quicker way. We shall use this approach to establish a general formula for the moments of the logarithmic derivative. 

Before proceeding, let us mention again that the moments of the logarithmic derivative of the characteristic polynomial as well as their asymptotic behavior have already been studied in the random matrix literature in relation with the Riemann zeta function (see e.g.  \cite{bib:CFZ08}, \cite{bib:CoSn} or \cite{bib:CoSn08}). Since the formula for the ratios that is usually used in this literature is different from our formula \eqref{formuleratios}, the formula we shall establish will look different as well.

We will state our main formula at the end of the section after discussing several computational steps.

We assume for the moment that the $z_j$'s and the $z'_j$'s are all distinct and not on the real line. 
If $A$ denote the set of indexes $j$ such that $z_j$ has negative real part, we get
$$\mathbb{E} \left[ \prod_{j=1}^k \frac{ \xi_{\infty}(z'_j)}{ \xi_{\infty}(z_j)} \right] 
= \frac{ \det \left( \frac{e^{2 i\pi (z'_j - z_i) \mathds{1}_{i \in A}}}{z_i - z'_j} \right)_{i,j = 1}^k}
{ \det  \left( \frac{1}{z_i - z'_j} \right)_{i,j = 1}^k}.$$
The denominator is the Cauchy determinant: 
$$ \det  \left( \frac{1}{z_i - z'_j} \right)_{i,j = 1}^k 
= \frac{ \prod_{i < j} (z_j - z_i)  \prod_{i < j} (z'_i - z'_j) }{ \prod_{i,j} (z_i - z'_j)}.$$
Expanding the numerator then gives, after dividing by the Cauchy determinant: 
$$ \mathbb{E} \left[ \prod_{j=1}^k \frac{ \xi_{\infty}(z'_j)}{ \xi_{\infty}(z_j)} \right] 
= \sum_{\sigma \in \mathfrak{S}_k} \epsilon (\sigma) 
\prod_{i \in A} e^{2i \pi (z'_{\sigma(i)}- z_i)} \prod_{i, j \neq \sigma(i)} (z_i - z'_j)
\prod_{i < j} (z_j - z_i)^{-1}  \prod_{i < j} (z'_i - z'_j)^{-1}.$$
This expression is proven for $z_j$, $z'_j$ all distinct: by continuity, it  also holds
for the $z_j$'s distinct on one hand, and the $z'_j$'s distinct on the other hand. 
Now, it is possible, in the last expression, to differentiate inside the expectation, with respect to 
any set of variables. Indeed, since the product of ratios of $\xi_{\infty}$ 
is holomorphic with respect to all the variables
on $\mathbb{C} \backslash \mathbb{R}$, differentiating is equivalent to taking suitable integrals on small circles, by 
using the formula 
$$f'(z) = \frac{1}{2 \pi \epsilon} \int_{0}^{2 \pi} e^{-i \theta} f(z + \epsilon e^{i \theta}) d\theta,$$
and the integrals can be exchanged with the expectation, because all the moments of ratios of 
$\xi_{\infty}$ are uniformly bounded on compact sets of $\mathbb{C} \backslash \mathbb{R}$, by Theorem 3.11. 
We deduce that for $z_1, \dots, z_k$ pairwise distinct and not real: 
$$\mathbb{E} \left[ \prod_{j=1}^k \frac{ \xi'_{\infty}(z_j)}{ \xi_{\infty}(z_j)} \right] 
=
\sum_{\sigma \in \mathfrak{S}_k} \epsilon (\sigma) \frac{\partial^k}{\partial z'_1 \dots \partial z'_k} \left(
\prod_{i \in A} e^{2i \pi (z'_{\sigma(i)}- z_i)} \right.$$ $$ \left. \prod_{i, j \neq \sigma(i)} (z_i - z'_j)
\prod_{i < j} (z_j - z_i)^{-1}  \prod_{i < j} (z'_i - z'_j)^{-1}
\right)_{z'_1 = z_1, \dots, z'_k = z_k}.$$
For each permutation $\sigma$, we have a multiple derivative of a product, so we have to add all the possible terms 
obtained by distributing the derivations on the different factors. 
If $j$ is not in the set $F(\sigma)$ of fixed points of $\sigma$, then the product contains the factor $z_j - z'_j$. 
If this factor is not differentiated with respect to $z'_j$, the corresponding term vanishes by taking 
$z'_j = z_j$. If this factor is differentiated with respect to $z'_j$, it becomes equal to $-1$. Hence, we 
get
$$ \mathbb{E} \left[ \prod_{j=1}^k \frac{ \xi'_{\infty}(z_j)}{ \xi_{\infty}(z_j)} \right] 
=
\prod_{i < j} (z_j - z_i)^{-1} \sum_{\sigma \in \mathfrak{S}_k} \epsilon (\sigma) (-1)^{k - |F(\sigma)|} ... $$ $$
... \times
\frac{\partial^{|F(\sigma)|}}{\prod_{j \in F(\sigma)} \partial z'_j}  \left(
\prod_{i \in A} e^{2i \pi (z'_{\sigma(i)}- z_i)} \prod_{i, j \neq i, \sigma(i)} (z_i - z'_j)
\prod_{i < j} (z'_i - z'_j)^{-1}
\right)_{z'_1 = z_1, \dots, z'_k = z_k}.$$
We have to differentiate a product of three factors (which are products themselves). 
Hence, we can write the result as a sum of terms indexed by partitions of $F(\sigma)$ into three subsets $U$,
$V$, $W$. For a given partition, the corresponding term can be nonzero only if the first product depends on all the 
variables indexed by $U$, which means that for all $j \in U$, there exists $i \in A$ such that
$\sigma(i)= j$, i.e. $i = \sigma^{-1}(j)$. Since $U$ is included in $F(\sigma)$, we have $j \in F(\sigma)$ and 
then $\sigma^{-1}(j) = j$, i.e. $i = j$ and $j \in A$. Hence, we only need to consider partitions for which 
$U \subset A$. Moreover, if this condition is satisfied, each derivation of the first term simply multiplies 
it by $2 i \pi$, and then we get: 
$$ \mathbb{E} \left[ \prod_{j=1}^k \frac{ \xi'_{\infty}(z_j)}{ \xi_{\infty}(z_j)} \right] 
=
\prod_{i < j} (z_j - z_i)^{-1} \sum_{\sigma \in \mathfrak{S}_k} \epsilon (\sigma) (-1)^{k - |F(\sigma)|} 
\prod_{i \in A} e^{2i \pi (z_{\sigma(i)}- z_i)}... $$ $$
... \times 
\sum_{U \sqcup V \sqcup W = F(\sigma), U \subset A } (2 i \pi)^{|U|} 
\frac{\partial^{|V|}}{\prod_{j \in V} \partial z'_j}  \left(
\prod_{i, j \neq i, \sigma(i)} (z_i - z'_j) \right)_{z'_1 = z_1, \dots, z'_k = z_k}
... $$ $$ ... \times
\frac{\partial^{|W|}}{\prod_{j \in W} \partial z'_j}  \left( \prod_{i < j} (z'_i - z'_j)^{-1}
\right)_{z'_1 = z_1, \dots, z'_k = z_k}.$$
In order to compute the differential with respect to $z'_j$, $j \in V$ in this formula, we need, for each 
$j \in V$, to remove one of the factors  $z_i - z'_j$ and to replace it by $-1$. The index $i$ in the removed factor 
is free as soon as it is different from $j$ and $\sigma^{-1}(j)$, now $j =  \sigma^{-1}(j)$ since 
$j \in V \subset F(\sigma)$. Hence, we get: 
$$ \frac{\partial^{|V|}}{\prod_{j \in V} \partial z'_j}  \left(
\prod_{i, j \neq i, \sigma(i)} (z_i - z'_j) \right)_{z'_1 = z_1, \dots, z'_k = z_k}
$$ $$= (-1)^{|V|} \sum_{\forall j \in V, w_j \in \{1, \dots, k\} \backslash \{j\}}  \, \prod_{j \notin V, i \neq j, \sigma^{-1}(j)} 
\,(z_i - z_j) \prod_{j \in V, i \neq j, w_j} (z_i - z_j).$$
The computation of the derivation with respect to the indices in $W$ is done by using the following lemma: 
\begin{lemma}
	For $x_1 \neq x_2 \neq \dots \neq x_k$, let 
	$$\Delta^{-1}(x_1, \dots, x_k) := \prod_{1 \leq i < j \leq k} (x_j - x_i)^{-1}.$$
	Then, for $1 \leq m \leq k$
	$$\frac{\partial^m}{ \prod_{j=1}^m \partial x_j}  \Delta^{-1}(x_1, \dots, x_k)
	=  \Delta^{-1}(x_1, \dots, x_k) \sum_{i_1 \neq 1, \dots, i_m \neq m} 
	2^{N(i_1, \dots, i_m)} \prod_{p=1}^m (x_{i_p} - x_p)^{-1},$$
	where $N(i_1, \dots, i_m)$ denotes the number of indices $p \in \{1, \dots, m\}$ such 
	that $i_p \in \{1, \dots, p\}$ and 
	$i_{i_p} = p$. 
\end{lemma}
\begin{proof}
	For $m = 1$, we obtain $k-1$ terms, obtained by differentiating each of the factors $(x_j - x_1)^{-1}$ with 
	respect to $x_1$. This multiplies the factor by $(x_j - x_1)^{-1}$, since the derivative is  $(x_j - x_1)^{-2}$.
	Hence, 
	$$ \frac{\partial^m}{\partial x_1}  \Delta^{-1}(x_1, \dots, x_k)
	=  \Delta^{-1}(x_1, \dots, x_k) \sum_{i \neq 1} (x_i - x_1)^{-1},$$
	which proves the formula for $m = 1$, since $N(i) = 0$ for all $i \neq 1$. 
	Let us now deduce the formula for $m \in \{2, \dots, k\}$ from the formula for $m-1$. 
	If the claimed forumula is true for $m-1$, 
	$$ \frac{\partial^m}{ \prod_{j=1}^m \partial x_j}  \Delta^{-1}(x_1, \dots, x_k)
	= \sum_{i_1 \neq 1, \dots, i_{m-1} \neq m-1} 
	2^{N(i_1, \dots, i_{m-1})} \prod_{p=1}^{m-1} (x_{i_p} - x_p)^{-1}
	\frac{\partial}{\partial x_m} \Delta^{-1}(x_1, \dots, x_k) 
	$$ $$+ \Delta^{-1}(x_1, \dots, x_k) \sum_{i_1 \neq 1, \dots, i_{m-1} \neq m-1} 2^{N(i_1, \dots, i_{m-1})}
	\frac{\partial}{\partial x_m} \left(\prod_{p=1}^{m-1} (x_{i_p} - x_p)^{-1} \right).$$
	The derivative in the first term gives terms with an extra factor $(x_{i_m} - x_m)^{-1}$, for all $i_m \neq m$. 
	The derivative in the second term gives terms with an extra factor $(x_p - x_m)^{-1}$, for all $p 
	\in \{1, \dots, m-1\}$ such that $i_p = m$. Hence 
	$$ \frac{\partial^m}{ \prod_{j=1}^m \partial x_j}  \Delta^{-1}(x_1, \dots, x_k) = 
	\sum_{i_1 \neq 1, \dots, i_{m} \neq m} 2^{N(i_1, \dots, i_{m})} \prod_{p=1}^{m-1} (x_{i_p} - x_p)^{-1}
	\Delta^{-1}(x_1, \dots, x_k)$$ $$
	+ \Delta^{-1}(x_1, \dots, x_k)  \sum_{i_1 \neq 1, \dots, i_{m-1} \neq m-1}
	2^{N(i_1, \dots, i_{m-1})} \sum_{p \in \{1, \dots, m-1\}, 
		i_p = m}  (x_p - x_m)^{-1} \prod_{q=1}^{m-1} (x_{i_q} - x_q)^{-1}.$$
	Now, if the index $p$ in the last sum is denoted $i_m$, the constraint on $i_m$ is that $i_m \in \{1, \dots, m-1\}$
	and $i_{i_m} = m$, or equivalently, $i_m \neq m$, $i_m \in \{1, \dots, m\}$, $i_{i_m} = m$. 
	Moreover, the factor $(x_p - x_m)^{-1}$ is equal to  $(x_{i_m} - x_m)^{-1}$. Hence, 
	$$ \frac{\partial^m}{ \prod_{j=1}^m \partial x_j}  \Delta^{-1}(x_1, \dots, x_k) = 
	\sum_{i_1 \neq 1, \dots, i_{m} \neq m} 2^{N(i_1, \dots, i_{m-1})} \prod_{p=1}^{m} (x_{i_p} - x_p)^{-1}
	\Delta^{-1}(x_1, \dots, x_k)$$ $$
	+ \Delta^{-1}(x_1, \dots, x_k)  \sum_{i_1 \neq 1, \dots, i_{m} \neq m}
	2^{N(i_1, \dots, i_{m-1})} \mathds{1}_{i_m \in \{1, \dots,m \}, i_{i_m} = m}
	\prod_{p=1}^{m} (x_{i_p} - x_p)^{-1}.$$
	and then
	\begin{align*}
	& \frac{\partial^m}{ \prod_{j=1}^m \partial x_j}  \Delta^{-1}(x_1, \dots, x_k)
	\\ & =  \Delta^{-1}(x_1, \dots, x_k) \sum_{i_1 \neq 1, \dots, i_m \neq m} 
	2^{N(i_1, \dots, i_{m-1})} (1 + \mathds{1}_{i_m \in \{1, \dots,m \}, i_{i_m} = m}) \prod_{p=1}^m (x_{i_p} - x_p)^{-1}
	\\ & = \Delta^{-1}(x_1, \dots, x_k) \sum_{i_1 \neq 1, \dots, i_m \neq m} 
	2^{N(i_1, \dots, i_{m-1}) +\mathds{1}_{i_m \in \{1, \dots,m \}, i_{i_m} = m} } \prod_{p=1}^m (x_{i_p} - x_p)^{-1}
	\\ & = \Delta^{-1}(x_1, \dots, x_k) \sum_{i_1 \neq 1, \dots, i_m \neq m} 
	2^{N(i_1, \dots, i_{m})} \prod_{p=1}^m (x_{i_p} - x_p)^{-1}.
	\end{align*}
\end{proof}
	From this lemma, we immediately get
	$$ \frac{\partial^{|W|}}{\prod_{j \in W} \partial z'_j}  \left( \prod_{i < j} (z'_i - z'_j)^{-1}
	\right)_{z'_1 = z_1, \dots, z'_k = z_k}
	$$ $$ =  \left( \prod_{i < j} (z_i - z_j)^{-1} \right)  \sum_{\forall j \in W, w_j \in \{1, \dots, k\}
		\backslash \{j\}} 2^{N(w_j, j \in W)} \prod_{j \in W} (z_{w_j} - z_j)^{-1},$$
	where $N(w_j, j \in W)$ denotes the number of pairs $\{j, k\} \subset W$,
	such that $w_j = k$ and $w_k = j$, in other words $1/2$ of the number of 
	$j \in W$ such that $w_j \in W$ and $w_{w_j} = j$. 
	Hence, we deduce 
	$$ \mathbb{E} \left[ \prod_{j=1}^k \frac{ \xi'_{\infty}(z_j)}{ \xi_{\infty}(z_j)} \right] 
	=
	\prod_{i \neq j} (z_j - z_i)^{-1} \sum_{\sigma \in \mathfrak{S}_k} \epsilon (\sigma) (-1)^{k - |F(\sigma)|} 
	\prod_{i \in A} e^{2i \pi (z_{\sigma(i)}- z_i)}... $$ $$
	... \times 
	\sum_{U \sqcup V \sqcup W = F(\sigma), U \subset A } (2 i \pi)^{|U|} 
	(-1)^{|V|} \sum_{\forall j \in V \cup W, w_j \in \{1, \dots, k\} \backslash \{j\}} 
	2^{N(w_j, j \in W)} \, ... $$ $$... \times \prod_{j \notin V, i \neq j, \sigma^{-1}(j)} (z_i - z_j)
	\, \prod_{j \in V, i \neq j, w_j} (z_i - z_j) \prod_{j \in W} (z_{w_j} - z_j)^{-1}$$
	The sum indexed by $U, V, W$ can be simplified by considering $X = V \cup W$, and by using the fact that 
	$U$ is the complement of $X$ in $F(\sigma)$. We have $|U| = |F(\sigma)| - |X|$, and 
	by splitting the following product in two factors corresponding to $j \in W \subset F(\sigma)$ and $j \notin V \cup
	W = X$,
	$$\prod_{j \notin V, i \neq j, \sigma^{-1}(j)} (z_i - z_j)
	= \prod_{j \in W, i \neq j} (z_i - z_j) \prod_{j \notin X, i \neq j, \sigma^{-1}(j)} (z_i - z_j),$$
	which implies 
	$$  \prod_{j \notin V, i \neq j, \sigma^{-1}(j)} (z_i - z_j)
	\, \prod_{j \in V, i \neq j, w_j} (z_i - z_j) \prod_{j \in W} (z_{w_j} - z_j)^{-1}
	$$ $$ = \prod_{j \in X, i \neq j, w_j} (z_i - z_j) \prod_{j \notin X, i \neq j, \sigma^{-1}(j)} (z_i - z_j).$$
	We deduce 
	$$ \mathbb{E} \left[ \prod_{j=1}^k \frac{ \xi'_{\infty}(z_j)}{ \xi_{\infty}(z_j)} \right] 
	=
	\prod_{i \neq j} (z_j - z_i)^{-1} \sum_{\sigma \in \mathfrak{S}_k} \epsilon (\sigma) (-1)^{k - |F(\sigma)|} 
	\prod_{i \in A} e^{2i \pi (z_{\sigma(i)}- z_i)}... $$ $$
	... \times  \sum_{F(\sigma) \backslash A \subset X \subset F(\sigma) } (2 i \pi)^{|F(\sigma)| - |X|} 
	\sum_{\forall j \in X, w_j \in \{1, \dots, k\} \backslash \{j\}} 
	\prod_{j \in X, i \neq j, w_j} (z_i - z_j)... $$ $$
	... \times  \prod_{j \notin X, i \neq j, \sigma^{-1}(j)} (z_i - z_j)
	\sum_{W \subset X} (-1)^{|X| - |W|} 2^{N(w_j, j \in W)}.$$
	Let us now compute the last sum indexed by $W$. If there exists $j \in X$ such that $w_j \notin X$ or 
	$w_{w_j} \neq j$, then for all $W \subset X$ such that $j \notin W$, we have
	$$2 N(w_i, i \in W \cup \{j\}) = \sum_{i \in W} \mathds{1}_{w_i \in W \cup \{j\}, w_{w_i} = i} + 
	\mathds{1}_{w_j \in W \cup \{j\}, w_{w_j} = j}.$$
	The last indicator function is equal to zero since $w_j \notin X$ or $w_{w_j} \neq j$ by assumption. Hence,
	$$ 2 N(w_i, i \in W \cup \{j\}) = \sum_{i \in W} \mathds{1}_{w_i \in W , w_{w_i} = i}
	+ \sum_{i \in W} \mathds{1}_{w_i = j , w_{w_i} = i}.$$
	In the second sum, the indicator functions are also zero: otherwise we would have $w_j = w_{w_i} = i \in W$, 
	and $w_{w_j} = w_i = j$. Hence 
	$$  N(w_i, i \in W \cup \{j\}) = N(w_i, i \in W )$$ and 
	\begin{align*}&  \sum_{W \subset X} (-1)^{|X| - |W|} 2^{N(w_i, i \in W)}
	\\ &  = \sum_{W \subset X, j \notin W} \left( (-1)^{|X| - |W|} 2^{N(w_i, i \in W)} + 
	(-1)^{|X| - |W| - 1} 2^{N(w_i, i \in W \cup \{j\} )} \right)
	\\ & = \sum_{W \subset X, j \notin W} \left( (-1)^{|X| - |W|} 2^{N(w_i, i \in W)} + 
	(-1)^{|X| - |W| - 1} 2^{N(w_i, i \in W)} \right) = 0.
	\end{align*}
	We can then restrict our computations to the case where $w_j \in X$ and $w_{w_j} = j$ for all $j \in X$: 
	in other words $w_j = \tau(j)$, where $\tau$ is an involution of $X$ without fixed point.
	If $\mathcal{I}_X$ denotes the set of all involutions of $X$ without fixed point (in particular $\mathcal{I}_X$ is 
	empty if $|X|$ is odd), we get 
	$$ \mathbb{E} \left[ \prod_{j=1}^k \frac{ \xi'_{\infty}(z_j)}{ \xi_{\infty}(z_j)} \right] 
	=
	\prod_{i \neq j} (z_j - z_i)^{-1} \sum_{\sigma \in \mathfrak{S}_k} \epsilon (\sigma) (-1)^{k - |F(\sigma)|} 
	\prod_{i \in A} e^{2i \pi (z_{\sigma(i)}- z_i)}... $$ $$
	... \times  \sum_{F(\sigma) \backslash A \subset X \subset F(\sigma) } (2 i \pi)^{|F(\sigma)| - |X|} 
	\sum_{\tau \in \mathcal{I}_X} 
	\prod_{j \in X, i \neq j, \tau(j)} (z_i - z_j)... $$ $$
	... \times  \prod_{j \notin X, i \neq j, \sigma^{-1}(j)} (z_i - z_j)
	\sum_{W \subset X} (-1)^{|W|} 2^{N(\tau(j), j \in W)}.$$
	In the sum in $W$, we have replaced $(-1)^{|X| - |W|}$ by $(-1)^{|W|}$, since $\tau$ can exist only for 
	$|X|$ even. 
	Now, if $X_1, \dots, X_h$ are the supports of the cycles of $\tau$ ($h = |X|/2$), we have for 
	all $W \subset X$, 
	$$(-1)^{|W|} = \prod_{p= 1}^h (-1)^{|W \cap X_p|}, \, 2^{N(\tau(j), j \in W)} 
	= \prod_{p=1}^h 2^{\mathds{1}_{W \cap X_p = X_p}}.$$
	Hence, 
	\begin{align*} \sum_{W \subset X} (-1)^{|W|} 2^{N(\tau(j), j \in W)}
	& = \sum_{W_1 \subset X_1, \dots, W_h \subset X_h} 
	\prod_{p=1}^h (-1)^{|W_p|} 2^{\mathds{1}_{W_p = X_p}} 
	\\ & = \prod_{p = 1}^h \sum_{W_p \subset X_p} (-1)^{|W_p|} 2^{\mathds{1}_{W_p = X_p}} 
	= \prod_{p = 1}^h (1 - 1 - 1 + 2) = 1.
	\end{align*}
	We know have 
	$$ \mathbb{E} \left[ \prod_{j=1}^k \frac{ \xi'_{\infty}(z_j)}{ \xi_{\infty}(z_j)} \right] 
	=
	\prod_{i \neq j} (z_i - z_j)^{-1} \sum_{\sigma \in \mathfrak{S}_k} \epsilon (\sigma) (-1)^{k - |F(\sigma)|} 
	\prod_{i \in A} e^{2i \pi (z_{\sigma(i)}- z_i)}... $$ $$
	... \times  \sum_{F(\sigma) \backslash A \subset X \subset F(\sigma), |X| \in 2 
		\mathbb{Z} } (2 i \pi)^{|F(\sigma)| - |X|} 
	\sum_{\tau \in \mathcal{I}_X} 
	\prod_{j \in X, i \neq j, \tau(j)} (z_i - z_j) \prod_{j \notin X, i \neq j, \sigma^{-1}(j)} (z_i - z_j),$$
	which can be simplified in 
	$$ \mathbb{E} \left[ \prod_{j=1}^k \frac{ \xi'_{\infty}(z_j)}{ \xi_{\infty}(z_j)} \right] 
	= \sum_{\sigma \in \mathfrak{S}_k} \epsilon (\sigma) (-1)^{k - |F(\sigma)|} 
	\prod_{j \notin F(\sigma)} (z_{\sigma^{-1}(j)} - z_j)^{-1} ... $$ $$
	... \times e^{2i \pi \sum_{j \in A} (z_{\sigma(j)}- z_j)} \sum_{F(\sigma) \backslash A \subset X \subset F(\sigma), |X| \in 2 
		\mathbb{Z} } (2 i \pi)^{|F(\sigma)| - |X|} 
	\sum_{\tau \in \mathcal{I}_X} 
	\prod_{j \in X} (z_{\tau(j)} - z_j)^{-1},$$

Let us now reorder this sum in function of the permutation $\rho = \sigma \circ \tau$. The condition
$ F(\sigma) \backslash A \subset X \subset F(\sigma)$ means that only the points in $A$ can be fixed by $\rho$. 
Moreover, for a given $\rho$, $\sigma$ can be any permutation obtained by removing some of the 
$2$-cycles of $\rho$. If $S(\rho)$ denotes the set of such permutations $\sigma$, and if for $\sigma \in S(\rho)$, 
we denote by $N(\sigma, \rho)$ the number of $2$-cycles which are removed, we get 
$$  \mathbb{E} \left[ \prod_{j=1}^k \frac{ \xi'_{\infty}(z_j)}{ \xi_{\infty}(z_j)} \right] 
=(-1)^{k}\sum_{\rho \in \mathfrak{S}_k, F(\rho) \subset A} 
\epsilon (\rho) (-2 i \pi)^{ |F(\rho)|} 
...$$ $$... \times \prod_{j \notin F(\rho)} (z_{\rho^{-1}(j)} - z_j)^{-1} 
\sum_{\sigma \in S(\rho)} (-1)^{N(\sigma, \rho)}  e^{2i \pi \sum_{j \in A} (z_{\sigma(j)}- z_j)}.$$
Let us suppose that $\rho$ has a $2$-cycle completely outside or inside $A$. 
In this case, if $\sigma, \sigma' \in S(\rho)$ differ only by this cycle, 
$$(-1)^{N(\sigma, \rho)}  e^{2i \pi \sum_{j \in A} (z_{\sigma(j)}- z_j)}
= - (-1)^{N(\sigma', \rho)}  e^{2i \pi \sum_{j \in A} (z_{\sigma'(j)}- z_j)}$$
Hence, the sum for $\sigma \in S(\rho)$ vanishes. Otherwise, the $2$-cycles of $\rho$
are of the form $(a,b)$ where $a \in A$ and $b \notin A$, and removing such a cycle
removes a factor $- e^{2i \pi (z_b-z_a)}$ in the term
$$(-1)^{N(\sigma, \rho)}  e^{2i \pi \sum_{j \in A} (z_{\sigma(j)}- z_j)}.$$
Hence, if $\mathfrak{S}_{k, A}$ denotes the set of permutations of order $k$ such that 
all the fixed points are in $A$ and all the $2$-cycles have one element in $A$ and one outside $A$,
and if for $\rho \in \mathfrak{S}_{k, A}$, $G(\rho,A)$ is the set of elements in $A$ in the $2$-cycles of $\rho$, then 
we get 
$$ \mathbb{E} \left[ \prod_{j=1}^k \frac{ \xi'_{\infty}(z_j)}{ \xi_{\infty}(z_j)} \right] 
=(-1)^{k}\sum_{\rho \in \mathfrak{S}_{k,A}} 
\epsilon (\rho) (-2 i \pi)^{ |F(\rho)|} 
...$$ $$... \times \prod_{j \notin F(\rho)} (z_{\rho^{-1}(j)} - z_j)^{-1} 
e^{2i \pi \sum_{j \in A} (z_{\rho(j)}- z_j)} \prod_{j  \in G(\rho,A)} \left(1 - e^{2i \pi (z_j - z_{\rho(j)})}
\right).$$
We now summarize the above in the following Proposition:

\begin{proposition}
	Let $z_1,\cdots,z_n$ be distinct complex numbers in $\mathbb  C\setminus\mathbb R$, and let $A$ denote the set of indexes $j$ such that $z_j$ has negative real part.
Let us also note $\mathfrak{S}_{k, A}$  the set of permutations of order $k$ such that 
all the fixed points are in $A$ and all the $2$-cycles have one element in $A$ and one outside $A$, and let $F(\sigma)$ be the set of fixed points of $\sigma$. For $\rho \in \mathfrak{S}_{k, A}$, we note $G(\rho,A)$  the set of elements in $A$ in the $2$-cycles of $\rho$. Then the following formula for the moments of the logarithmic derivative holds: 
$$ \mathbb{E} \left[ \prod_{j=1}^k \frac{ \xi'_{\infty}(z_j)}{ \xi_{\infty}(z_j)} \right] 
=(-1)^{k}\sum_{\rho \in \mathfrak{S}_{k,A}} 
\epsilon (\rho) (-2 i \pi)^{ |F(\rho)|} 
...$$ $$... \times \prod_{j \notin F(\rho)} (z_{\rho^{-1}(j)} - z_j)^{-1} 
e^{2i \pi \sum_{j \in A} (z_{\rho(j)}- z_j)} \prod_{j  \in G(\rho,A)} \left(1 - e^{2i \pi (z_j - z_{\rho(j)})}
\right).$$
\end{proposition}

\begin{rmk}
For $k=1$ and $k=2$, we recover what we obtained before.
\end{rmk}

\section{Mesoscopic fluctuations and blue noise}
\label{section:mesoscopic}

The function $\frac{ \xi_{\infty}' }{ \xi_{\infty} }(z)-i\pi $ studied in the pervious section was recently considered by Aizenman and Warzel in \cite{bib:AiWa}. They prove that for any $z \in \R$, the value of this function follows the Cauchy distribution: in fact, their result applies
to  more general point processes than the  sine kernel process. In the present 
 paper, we  deal with the same function but away from the real line. In this section we shall view this function in the framework of linear statistics and will study its fluctuations on a mesoscopic level. It is may be worth noting here that  $\frac{ \xi_{\infty}' }{ \xi_{\infty} }(z)-i\pi $ also has a spectral interpretation: informally, it is the trace of the resolvent of the (unbounded) random Hermitian operator whose spectrum consists exactly of the points $(y_k)_{k\in\Z}$ that we constructed in \cite{bib:MNN}. This interpretation is informal since the series corresponding to the resolvant is not absolutely convergent.

For $s \geq 0$, we consider the Sobolev space: 
$$ H^{s} := \left\{ f \in L^2\left( \R, \C \right) \ \big | \ \int_{\R} \left|\hat{f}(k)\right|^2 \left(1+|k|^2\right)^s dk \right\},$$

We then call {\it blue noise} a Gaussian family of centered variables indexed by $H^{1/2}$, denoted 
$(\mathcal{B}(f))_{f \in H^{1/2}}$, such that $f \mapsto \mathcal{B}(f)$ is linear, $\mathcal{B}(f)$ is 
a.s. real if $f$ is a real-valued function, and 
$$\mathbb{E} [|\mathcal{B}(f)|^2 ] = \frac{1}{2 \pi} \int_{\R} |k| |\hat{f}(k)|^2 dk.$$
The covariance structure of $\mathcal{B}$ is then: 
$$\mathbb{E} [\mathcal{B}(f) \mathcal{B}(g) ] = \frac{1}{2 \pi} \int_{\R} |k| \hat{f}(k) \hat{g}(-k) dk,$$
$$\mathbb{E} [\mathcal{B}(f) \overline{\mathcal{B}(g)} ] = \frac{1}{2 \pi} \int_{\R} |k| \hat{f}(k)
 \overline{\hat{g}(k)} dk.$$
Similarly as for the Brownian motion, we can take the notation: 
$$\int_{\R} f(t) d \mathcal{B}_t := \mathcal{B}(f).$$
Now, for any function $f \in L^1(\R, \C) \cap L^2(\R, \C)$, we have 
\begin{align*} \mathbb{E} \left[ \left( \sum_{k \in \Z} |f\left( y_k \right)| \right)^2 \right]
& = \int_{\R} |f|^2 + \int_{\R^2} (1 - S^2(x-y)) |f(x)||f(y)| dx dy
\\ & \leq \int_{\R} |f|^2 + \left(\int_{\R} |f| \right)^2 < \infty,
\end{align*}
and then 
$$ X_f := \sum_{k \in \Z} f\left( y_k \right) - \int f$$
is well-defined as a square-integrable random variable. 
As we will see in Corollary \ref{corollary:X_f_is_continuous}, $X_f$ can also be defined as
a square-integrable random variable as soon as $f \in H^{1/2}$, even if $f$ is not integrable. 

In this section, we examine the behavior of $\sum_{k \in \Z} f\left( \frac{y_k}{L} \right) 
- L \int f$ as $L \rightarrow \infty$ for suitable functions $f$:
\begin{thm}
\label{thm:blue_cv}
If $\left( y_k ; k \in \Z \right)$ is a sine kernel 
point process, there is a blue noise $\mathscr{B}$ such that 
$$ \left( X_{ f\left( \frac{\cdot}{L} \right) } \right)_{f \in H^{1/2}}
\stackrel{L \rightarrow \infty}{\longrightarrow} \left(\int
f(t) d\mathscr{B}_t \right)_{f \in H^{1/2}} , $$
the convergence holding in law for finite-dimensional marginals. 
\end{thm}

In Subsection \ref{subsection:blue_banana}, we analyse the asymptotic 
behavior of the Stieltjes transform of the sine kernel process. To that endeavor, we apply
the result to the complex-valued functions $f_z\left( t \right) = \frac{1}{z-t}$. 

\subsection{The sine kernel from afar}

We will need an intermediate proposition:
\begin{proposition}[Adapted from Soshnikov \cite{bib:soshnikov00}]
\label{proposition:sosh_adapted}
If $f$ is a smooth, real-valued function with compact support and if 
the $p$-th cumulant of $X_f$ is denoted $C_{p}(f)$, then we 
have:
$$ C_{1}(f) = 0$$
$$ \left| C_{2}(f) - \frac{1}{2 \pi}  \int \left| \hat{f}(k) \right|^2 |k| dk \right| \ll \int |k| \left|\hat{f}(k)\right|^2 \mathds{1}_{ |k| \geq \pi } dk$$
$$ \forall p \geq 3, \left| C_{p}(f) \right| \ll_p \int_{ k_1  + \dots +  k_p  = 0 }
                                                                              \mathds{1}_{ |k_1| + \dots + |k_p| > 2\pi} 
                                                                              |k_1| \left|
                                                                              \hat{f}\left( k_1 \right) \dots
                                                                              \hat{f}\left( k_p \right)
                                                                              \right| dk$$
where in the previous equation, $dk$ stands for the Lebesgue measure on the hyperplane $\left\{ k_1 + \dots + k_p = 0 \right\}$.
\end{proposition}

\begin{proof}
The first equality is immediate. Now, since $y_k^{(n)} = \frac{n}{2 \pi} \theta_{k}^{(n)}$ converges almost
surely to $y_k$, $X_f$ is the almost sure limit of:
\begin{align*}
 X_{n,f} := & \sum_{k \in \Z} f\left( \frac{n}{2 \pi} \theta_{k}^{(n)} \right) - \int f \\
          = & \sum_{k=1}^n\left( -\frac{1}{n}\int f + \sum_{l \in \Z} f\left( \frac{n}{2 \pi} \theta_{k}^{(n)} + n l\right) \right) \\
          = & \sum_{k=1}^n \psi_n\left( \theta_{k}^{(n)} \right),
\end{align*}
where $\psi_n$ is the sequence of $2\pi$-periodic functions with zero mean:
$$ \psi_n\left( \theta \right) = -\frac{1}{n}\int f + \sum_{l \in \Z} f\left( \frac{n}{2 \pi} \theta + n l\right) $$
If $\hat{f}$ is the Fourier transform of $f$, the Fourier coefficients 
$$\left(c_k (\psi_n) := \frac{1}{2 \pi} \int_0^{2 \pi} \psi_n(\theta) e^{-ik \theta} d\theta  ; k \in \Z \right)$$
of $\psi_n$ are given by:
$$ c_0(\psi_n) = 0,$$
$$ \forall k \in \Z^*, c_k(\psi_n) = \frac{\sqrt{2 \pi}}{n} \hat{f}\left( \frac{2 \pi k}{n} \right).$$
If $C_{p,n}\left( f \right)$ is the $p$-th cumulant of $X_{n,f}$, thanks to the main combinatorial lemma and Lemma 1 in \cite{bib:soshnikov00}, we have:
$$ C_{1,n}\left( f \right) = 0$$
$$ \left| C_{2, n}\left( f \right) - \frac{2 \pi}{n} \sum_{k \in \Z} \frac{|k|}{n} \hat{f}\left( \frac{2 \pi k}{n} \right)
                                                                               \hat{f}\left(-\frac{2 \pi k}{n} \right)  \right| 
                               \ll  \frac{1}{n} \sum_{ |k| > \half n} \frac{|k|}{n} \hat{f}\left( \frac{2 \pi k}{n} \right)
                                                                                    \hat{f}\left(-\frac{2 \pi k}{n} \right)$$
$$ \forall p \geq 3, \left| C_{p,n}\left( f \right) \right| \ll_p 
               \frac{1}{n^{p-1}} \sum_{ \substack{ k_1  + \dots +  k_p  = 0 \\
                                                  |k_1| + \dots + |k_p| > n}} \frac{|k_1|}{n} \left|
                                                                              \hat{f}\left( \frac{2 \pi k_1}{n} \right) \dots
                                                                              \hat{f}\left( \frac{2 \pi k_p}{n} \right)
                                                                              \right| $$
As $\hat{f}$ decays at infinity faster than any power, we recognize three converging Riemann sums. The first one is:
$$ \frac{2 \pi}{n} \sum_{k \in \Z} \frac{|k|}{n} \hat{f}\left( \frac{2 \pi k}{n} \right)
\hat{f}\left(-\frac{2 \pi k}{n} \right)
   \stackrel{n \rightarrow \infty}{\longrightarrow}
   2 \pi \int |k| \left| \hat{f}\left( 2 \pi k \right) \right|^2 dk
   =  \frac{1}{2 \pi} \int |k| \left| \hat{f}\left(  k \right) \right|^2 dk.
   $$
The others appear as error terms and are Riemann sums converging to integrals on the hyperplane $\left\{ k_1 + \dots + k_p = 0 \right\} \subset \R^{p}$.
\begin{align*}
\forall p \geq 2, \   & \frac{1}{n^{p-1}} \sum_{ \substack{ k_1  + \dots +  k_p  = 0 \\
                                                           |k_1| + \dots + |k_p| > n}} \frac{|k_1|}{n} \left|
                                                                              \hat{f}\left( \frac{2 \pi k_1}{n} \right) \dots
                                                                              \hat{f}\left( \frac{2 \pi k_p}{n} \right)
                                                                              \right|\\
\stackrel{n \rightarrow \infty}{\longrightarrow} & \int_{ k_1  + \dots +  k_p  = 0 }
                                                                              \mathds{1}_{ |k_1| + \dots + |k_p| > 1} 
                                                                              |k_1| \left|
                                                                              \hat{f}\left( 2 \pi k_1 \right) \dots
                                                                              \hat{f}\left( 2 \pi k_p \right)                                                                              \right| dk.
\end{align*}

Therefore, for every $p \geq 1$, the $p$-th cumulant of $X_{n,f}$ is bounded independently of $n$ and the 
sequence $|X_{n,f}|^p$ is uniformly integrable. Thus, the convergence of $X_{n,f}$ to $X_f$ is not only almost 
sure but also in every $L^p\left( \Omega \right)$, $\Omega$ being the underlying probability space.

Now since
$$ \forall p \geq 1, C_{p}(f) \stackrel{n \rightarrow \infty}{\longrightarrow} C_{p}(f),$$
we have
$$ \left| C_{2}(f) - \frac{1}{2 \pi}
\int \left| \hat{f}( k) \right|^2 |k| dk \right| \ll \int |k| \left|f(2\pi k)\right|^2 \mathds{1}_{ |k| \geq \half } dk$$
$$ \forall p \geq 3, \left| C_{p}(f) \right| \ll_p \int_{ k_1  + \dots +  k_p  = 0 }
                                                                              \mathds{1}_{ |k_1| + \dots + |k_p| > 1} 
                                                                              |k_1| \left|
                                                                              \hat{f}\left( 2 \pi k_1 \right) \dots
                                                                              \hat{f}\left( 2 \pi k_p \right)                                                                              \right| dk$$
After an obvious change of variables, we recover the claimed estimates.
\end{proof}

\begin{corollary}
\label{corollary:X_f_is_continuous}
The map 
$$ f \mapsto X_f$$
 from $L^1 (\R, \C) \cap H^{1/2}$ to $L^2\left( \Omega \right)$ 
admits a linear extension to $H^{1/2}$, which satisfies the following property of continuity: 
$$ \E\left( \left| X_f \right|^2 \right)^{\frac{1}{2}} \ll  \| f \| _{H^{\half}},$$
uniformly, for all $f \in H^{1/2}$. This extension is unique up to almost sure equality. 
\end{corollary}
\begin{proof}
The estimate on the second cumulant, given by Proposition \ref{proposition:sosh_adapted}, implies
$$
   \E\left( \left| X_f \right|^2 \right)^{\frac{1}{2}}
   \ll \| f \| _{H^{\half}}
$$
for every smooth, real-valued function $f$ with compact support. By linearity, this esimate remains true without 
the assumption that $f$ is real-valued. We deduce the existence of a family $(Y_f)_{f \in H^{1/2}}$ of 
random variables such that $Y_f = X_f$ a.s. if $f$ is smooth with compact support, and 
$$  \E\left( \left| Y_f \right|^2 \right)^{\frac{1}{2}}
   \ll \| f \| _{H^{\half}} $$
This family is unique up to almost sure equality. Then, we are done if we show that $X_f = Y_f$ almost surely 
as soon as $f \in L^1 \cap H^{1/2}$. Now, 
the map $f \mapsto X_f - Y_f$ from $f \in L^1 \cap H^{1/2}$ to $L^2 (\Omega)$ is a.s. equal to 
zero on $\mathcal{C}^{\infty}_c (\R, \C)$. 
Moreover, we have seen above, by using the two first correlation functions of the sine kernel process, that 
$$\E\left( \left| X_f \right|^2 \right)^{\frac{1}{2}} \leq ||f||_{L^1} + ||f||_{L^2}$$
which implies: 
$$ \E\left( \left| X_f - Y_f \right|^2 \right)^{\frac{1}{2}} \ll ||f||_{L^1} + ||f||_{L^2} + 
||f||_{H^{1/2}} \ll  ||f||_{L^1}  + ||f||_{H^{1/2}}.$$
Hence, the map $f \mapsto X_f - Y_f$ from $f \in L^1 \cap H^{1/2}$ is continuous, and since it vanishes 
on  $\mathcal{C}^{\infty}_c$, which is dense in $L^1 \cap H^{1/2}$, it vanishes everywhere.

%For L^p, the following proof would hold if we had stronger control over cumulants.
%Thanks to Jensen's inequality, it suffices the prove the inequality for $p$ an even integer. Then, expand $\E\left( \left| X_f \right|^p \right)$ as a polynomial of the cumulants $\left( C_k(f); 2 \leq k \leq p \right)$. Each monomial is of the form $C_{k_1}(f) \dots C_{k_l}(f)$ with $\sum k_j = p$. Because of the previous proposition:
%$$ \forall k \geq 2, \left| C_{k}(f) \right| \ll_k \| f \| _{H^{\half}}^k$$
%Hence, every monomial term in the expansion can be dominated by $\| f \| _{H^{\half}}^p$
\end{proof}

\begin{proof}[Proof of Theorem \ref{thm:blue_cv}]
It is sufficient to prove convergence in law of the one-dimensional marginals, for real-valued functions $f$. 
Indeed, if we have this convergence, if $f_1, \dots, f_m$ are real-valued functions in $H^{1/2}$, and 
if $\lambda_1, \dots, \lambda_m \in \R$, then we have the 
convergence in law
$$  X_{ f\left( \frac{\cdot}{L} \right) } 
= \sum_{j=1}^m   \lambda_j X_{ f_j\left( \frac{\cdot}{L} \right) } 
\stackrel{L \rightarrow \infty}{\longrightarrow} \mathcal{B}(f)  = \sum_{j=1}^m \lambda_j \mathcal{B}(f_j),$$
for
$$ f = \lambda_1 f_1 + \lambda_2 f_2 + \dots + \lambda_m f_m.$$
Applying the bounded, continuous function $x \mapsto e^{ix}$ gives the convergence of the Fourier transform 
of 
$(X_{ f_j\left( \frac{\cdot}{L} \right) })_{1 \leq j \leq m}$ towards 
the Fourier transform of $(\mathcal{B}(f_j))_{1 \leq j \leq m}$, and then the convergence of the finite-dimensional 
marginals claimed in Theorem \ref{thm:blue_cv}, for real-valued functions. 
The case of complex-valued functions is then deduced by linearity. 

If remains to prove that for all $f \in H^{1/2}$, real-valued, 
$$X_{ f\left( \frac{\cdot}{L} \right) } \stackrel{L \rightarrow \infty}{\longrightarrow} \mathcal{B}(f).$$
Let us first assume that $f$ is smooth function with compact support. If $C_{p}^{(L)}(f)$ is the
$p$-th cumulant of $X_{ f\left( \frac{\cdot}{L} \right) }$, then by rescaling the space variable:
$$ \forall k \in \R, \widehat{f\left( \frac{\cdot}{L} \right)}(k) = L \hat{f}(Lk)$$
and
$$ C_{1}^{(L)}(f) = 0$$
$$ \left| C_{2}^{(L)}(f) - \frac{1}{2\pi} \int \left| \hat{f}(k) \right|^2 |k| dk \right| \ll \int |k| \left|\hat{f}(k)\right|^2 \mathds{1}_{ \left\{ |k| \geq L \pi \right\} } dk$$
$$ \forall p \geq 3, \left| C_{p}^{(L)}(f) \right| \ll_p \int_{ k_1  + \dots +  k_p  = 0 }
                                                                              \mathds{1}_{\left\{ |k_1| + \dots + |k_p| > 2\pi L \right\} } 
                                                                              |k_1| \left|
                                                                              \hat{f}\left( k_1 \right) \dots
                                                                              \hat{f}\left( k_p \right)
                                                                              \right| dk$$
Therefore, as $L \rightarrow \infty$, $X_{ f\left( \frac{\cdot}{L} \right) }$ converges in law to a centered 
Gaussian random variable with variance $\frac{1}{2 \pi} \int |k| \left|f(k)\right|^2 dk$, i.e. to $\mathcal{B}(f)$. 

Now, if $f$ is only supposed to be in $H^{\half}$, let us 
consider a sequence of smooth compactly supported functions $\left( f_n \right)_{n \in \N}$ such that:
$$ \| f - f_n\| _{H^{\half}} \stackrel{n \rightarrow \infty}{\longrightarrow} 0 $$
We will be done after proving that for any $t$ in a compact set:
$$ \E\left( e^{i t X_{ f\left( \frac{\cdot}{L} \right) }} \right) \stackrel{n \rightarrow \infty}{\longrightarrow} \exp\left( - \frac{t^2}{4 \pi} \int |k| |\hat{f}(k)|^2 dk \right)$$
We have because of the triangular inequality, for fixed $n$:
\begin{align*}
     & \left| \E\left( e^{i t X_{ f\left( \frac{\cdot}{L} \right) }} \right) - \exp\left( - \frac{t^2}{4 \pi } \int |k| |\hat{f}(k)|^2 dk \right) \right|\\
\leq & \left| \E\left( e^{i t X_{ f\left( \frac{\cdot}{L} \right) }} \right) - \E\left( e^{i t X_{ f_n\left( \frac{\cdot}{L} \right) }} \right) \right| \\
     & + \left| \E\left( e^{i t X_{ f_n\left( \frac{\cdot}{L} \right) }} \right) - \exp\left( - \frac{t^2}{4 \pi } \int |k| |\hat{f_n}(k)|^2 dk \right) \right| \\
     & + \left| \exp\left( - \frac{t^2}{4 \pi } \int |k| |\hat{f}(k)|^2 dk \right) - \exp\left( - \frac{t^2}{4 \pi } \int |k| |\hat{f_n}(k)|^2 dk \right) \right|
\end{align*}
The third term is a $\Oc\left( t^2 \| f - f_n\| _{H^{\half}}^2 \right)$. The second disappears when we take the $\limsup_{L \rightarrow \infty}$. As for the first term, we have for any $\varepsilon>0$:
\begin{align*}
     & \left| \E\left( e^{i t X_{ f\left( \frac{\cdot}{L} \right) }} \right) - \E\left( e^{i t X_{ f_n\left( \frac{\cdot}{L} \right) }} \right) \right| \\
\leq & \E\left( \left| e^{i t X_{ f\left( \frac{\cdot}{L} \right) } - i t X_{ f_n\left( \frac{\cdot}{L} \right) }} - 1 \right| \right) \\
\leq & 2 \P\left( \left| X_{ f\left( \frac{\cdot}{L} \right) } - X_{ f_n\left( \frac{\cdot}{L} \right) } \right| \geq \varepsilon \right) + \varepsilon |t|\\
\leq & 2 \frac{ \E\left( \left| X_{ f\left( \frac{\cdot}{L} \right) } - X_{ f_n\left( \frac{\cdot}{L} \right) } \right|^2 \right) }{\varepsilon^2} + \varepsilon |t|
\end{align*}
By linearity and the second cumulant estimate:
$$ \E\left( \left| X_{ f\left( \frac{\cdot}{L} \right) } - X_{ f_n\left( \frac{\cdot}{L} \right) } \right|^2 \right)^\half
=  \E\left( \left| X_{ (f-f_n)\left( \frac{\cdot}{L} \right) } \right|^2 \right)^\half \ll \| f - f_n\| _{H^{\half}}$$
Hence for any fixed $n$ and $\varepsilon>0$:
\begin{align*}
     & \limsup_{L \rightarrow \infty} \left| \E\left( e^{i t X_{ f\left( \frac{\cdot}{L} \right) }} \right) - \exp\left( - \frac{t^2}{2} \int |k| |\hat{f}(k)|^2 dk \right) \right|\\
\ll  & \| f - f_n\| _{H^{\half}}^2 \left( \frac{1}{\varepsilon^2} + t^2 \right) + \varepsilon |t|
\end{align*}
Taking $n \rightarrow \infty$, then $\varepsilon \rightarrow 0$ concludes the proof.
\end{proof}

\subsection{Application to the Stieltjes transform of the sine kernel}
\label{subsection:blue_banana}

  For $z \in \C \backslash \R$, $f_{z} : t \mapsto 
  1/(z-t)$ is in $H^{1/2}$. Indeed, one can check (by using the inverse Fourier transform for example) that 
  $$\hat{f_z} (k) = -i \sqrt{2 \pi} \sgn{\Im (z)}
  e^{-izk} \mathds{1}_{k \Im (z) < 0},$$
  and then $\hat{f_z}$ decays exponentially at infinity.  
  Moreover, $X_{f_z}$ can be related to the logarithmic derivative of $\xi_{\infty}$: 
  \begin{proposition}
   For all $z \notin \R$, we have almost surely, 
   $$X_{f_z} = \frac{\xi_{\infty}'(z)}{\xi_{\infty}(z)} - 2 i \pi \mathds{1}_{\Im z < 0} 
    = i \pi \sgn \Im z + \frac{1}{z - y_0} + \sum_{k = 1}^{\infty}  \left(\frac{1}{z-y_k}
    + \frac{1}{z - y_{-k}} \right).$$
  \end{proposition}
  \begin{proof}
   Let $\varphi$ be a smooth even function from $\R$ to $[0,1]$, nonincreasing on $\R_+$,
   equal to $1$ on $[-1,1]$ and to 
   $0$ on $\R \backslash [-2,2]$.
   If for $A > 0$, $f^{(A)}_z(t) = f_z(t) \varphi(t/A)$, we have 
   $$|f^{(A)}_z(t) - f_z(t)| \leq |f_z(t)| \mathds{1}_{|t| \geq A}$$
   and 
   $$|(f^{(A)}_z)'(t) - f'_z(t)| = \left| f'_z(t) \varphi(t/A) + \frac{1}{A} \varphi'(t/A) f_z(t)  - f'_z(t) \right| 
   \ll |f'_z(t)| \mathds{1}_{|t| \geq A} + \frac{|f_z(t)|}{A}.$$
   For $z$ fixed, $|f_z(t)|$ is dominated by $1/(1+|t|)$, 
   $|f'_z(t)|$ is dominated by $1/(1+|t|)^2$, and then 
   $$ |f^{(A)}_z(t) - f_z(t)|^2 + |(f^{(A)}_z)'(t) - f'_z(t)|^2 \ll 
   \frac{\mathds{1}_{|t| \geq A}}{(1  + |t|)^2} + \frac{1}{A^2 (1+ |t|)^2}.$$
   We deduce that $f^{(A)}_z$ converges to $f_z$ in $H^1$, and a fortiori in $H^{1/2}$. 
   Hence, in $L^2 (\Omega)$, 
   \begin{align*}
X_{f_z} = \underset{A \rightarrow \infty}{\lim} X_{f_{z}^{(A)}} 
   & = \underset{A \rightarrow \infty}{\lim} \left(
   \sum_{k \in \mathbb{Z}} \frac{\varphi(y_k/A)}{z - y_k} - \int_{\R} 
   \frac{\varphi(y/A)}{z-y} dy \right)
  \\ &  =  \underset{A \rightarrow \infty}{\lim} \int_1^2 (-\varphi'(u))
    \left( \sum_{k \in \mathds{Z}}  \frac{\mathds{1}_{|y_k| \leq Au}}{z-y_k} - 
    \int_{-Au}^{Au} \frac{dy}{z-y} \right) du 
  \end{align*}
   From Proposition \ref{proposition:convergenceLp}, one easily deduces that 
   $$ \sum_{k \in \mathbb{Z}} \frac{\mathds{1}_{|y_k| \leq B}}{z-y_k} - 
    \int_{-B}^{B} \frac{dy}{z-y}
    \underset{B \rightarrow \infty}{\longrightarrow}
    \frac{\xi_{\infty}'(z)}{\xi_{\infty}(z)} - 2 i \pi \mathds{1}_{\Im z < 0}$$
    in $L^p(\Omega)$ for all $p \geq 1$, and in particular in $L^2(\Omega)$. 
    Now, since $-\varphi'$ is nonnegative in $[1,2]$ and has integral $1$, one has 
    $$ \left| \left| \int_1^2 (-\varphi'(u))
    \left( \sum_{k \in \mathds{Z}}  \frac{\mathds{1}_{|y_k| \leq Au}}{z-y_k} - 
    \int_{-Au}^{Au} \frac{dy}{z-y} \right) du -  \frac{\xi_{\infty}'(z)}{\xi_{\infty}(z)}
    + 2 i \pi \mathds{1}_{\Im z < 0} \right| \right|_{L^2(\Omega)} $$
    $$ \leq \int_1^2 (-\varphi'(u)) du 
    \left| \left|
    \sum_{k \in \mathds{Z}}  \frac{\mathds{1}_{|y_k| \leq Au}}{z-y_k} - 
    \int_{-Au}^{Au} \frac{dy}{z-y}  -  \frac{\xi_{\infty}'(z)}{\xi_{\infty}(z)}
    + 2 i \pi \mathds{1}_{\Im z < 0} \right| \right|_{L^2(\Omega)}$$
    $$  \leq \sup_{B \in [A, 2A]}  \left| \left|
     \sum_{k \in \mathds{Z}}  \frac{\mathds{1}_{|y_k| \leq B}}{z-y_k} - 
    \int_{-B}^{B} \frac{dy}{z-y}  -  \frac{\xi_{\infty}'(z)}{\xi_{\infty}(z)}
    + 2 i \pi \mathds{1}_{\Im z < 0} \right| \right|_{L^2(\Omega)},$$
    which tends to zero when $A$ goes to infinity. 
    Hence, in $L^2(\Omega)$, 
    $$ X_{f_z} =   \underset{A \rightarrow \infty}{\lim} \int_1^2 (-\varphi'(u))
    \left( \sum_{k \in \mathds{Z}}  \frac{\mathds{1}_{|y_k| \leq Au}}{z-y_k} - 
    \int_{-Au}^{Au} \frac{dy}{z-y} \right) du
    =  \frac{\xi_{\infty}'(z)}{\xi_{\infty}(z)} - 2 i \pi \mathds{1}_{\Im z < 0}.$$    
  \end{proof}
A consequence of the previous proposition is the following result: 
\begin{proposition}
 For $z \in \C \backslash \R$, 
 let $$F(z) := X_{f_z} = \frac{\xi_{\infty}'(z)}{\xi_{\infty}(z)} - 2 i \pi \mathds{1}_{\Im z < 0}.$$
 Then, one has the convergence in law: 
 $$(L F(Lz))_{z \in \C \backslash \R} 
 \underset{L \rightarrow \infty}{\longrightarrow} (G(z))_{z \in \C \backslash \R},$$
 where $G(z) = \mathcal{B}(f_z)$ for all $z \in  \C \backslash \R$. 
 The centered gaussian process $(G(z))_{z \in \C \backslash \R}$ has the covariance structure given, 
 for all $z_1, z_2 \notin \R$, by
 $$\E [G(z_1) G(z_2)]  = - \frac{\mathds{1}_{\Im(z_1) 
 \Im(z_2) < 0}}{(z_2 - z_1)^2},$$
 $$\E [G(z_1) \overline{G(z_2)}]  = - \frac{\mathds{1}_{\Im(z_1) 
 \Im(z_2) > 0}}{(\overline{z_2} - z_1)^2},$$
 and in particular
 $$\E [|G(z_1)|^2] 
  = \frac{1}{4 \Im^2 (z_1)}.$$ 
\end{proposition}
\begin{proof}
We have, for $L > 0$, 
$$f_z(t/L) = \frac{1}{z - (t/L)} = \frac{L}{Lz - t} = L f_{Lz}(t),$$
and then 
$$X_{f_z(\cdot/L)} = L X_{f_{Lz}} = L F(Lz).$$
The convergence in law given in this proposition is then a consequence of Theorem \ref{thm:blue_cv}.
It remains to compute the covariance structure. 
  For $z_1, z_2 \in \C \backslash \R$,
  \begin{align*} & \E [\mathscr{B}(f_{z_1})\mathscr{B}(f_{z_2})] 
  \\  & = \frac{1}{2 \pi} \int_{\R}
   |k| (-i \sqrt{2 \pi} \sgn \Im (z_1) 
   e^{-iz_1k} \mathds{1}_{k \Im(z_1) < 0} )
   (-i \sqrt{2 \pi} \sgn \Im (z_2) 
   e^{iz_2k} \mathds{1}_{-k \Im(z_2) < 0} ) dk.
  \end{align*}
If $\Im (z_1)$ and $\Im(z_2)$ have the same 
sign, the product of the indicator functions 
vanishes for all $k \in \R$, 
so 
$$\E [\mathscr{B}(f_{z_1})\mathscr{B}(f_{z_2})] = 0.$$
If $\Im(z_1)$ and $\Im(z_2)$ have not the same 
sign, we get 
$$\E [\mathscr{B}(f_{z_1})\mathscr{B}(f_{z_2})] = \int_{\R}
|k| e^{i k (z_2 - z_1)} \mathds{1}_{k \Im(z_2) > 0}.$$
By doing the change of variable 
$k' = k \sgn \Im (z_2)$, we get 
$$\E [\mathscr{B}(f_{z_1})\mathscr{B}(f_{z_2})] = \int_{0}^{\infty} 
k e^{i k (z_2 - z_1) \sgn \Im (z_2)} dk $$
Now,  for all $y  >0$, 
$$\int_{0}^{\infty} k e^{-yk} dk 
= \int_{0}^{\infty} (u/y) e^{-u} d(u/y) 
 = 1/y^2,$$
 and by analytic continuation, this formula is true for all $y$ with strictly positive real part. Applying 
 this to $y = -i (z_2 - z_1) 
 \sgn \Im (z_2)$, we have
 $$\E [\mathscr{B}(f_{z_1})\mathscr{B}(f_{z_2})] = - 1/(z_2 - z_1)^2$$
 for $\Im (z_1) \Im (z_2) < 0$, and then in any 
 case, 
 $$\E [\mathscr{B}(f_{z_1})\mathscr{B}(f_{z_2})]  = - \frac{\mathds{1}_{\Im(z_1) 
 \Im(z_2) < 0}}{(z_2 - z_1)^2}.$$
Since the blue noise here is real-valued for real functions, 
$\mathscr{B} (f_{\overline{z_2}})
 = \overline{\mathscr{B} (f_{z_2})}$, and then 
 $$\E [\mathscr{B}(f_{z_1})\overline{\mathscr{B}(f_{z_2})}]  = - \frac{\mathds{1}_{\Im(z_1) 
 \Im(z_2) > 0}}{(\overline{z_2} - z_1)^2}.$$
 \end{proof}
\begin{rmk}
The covariance structure of $F$ has been computed above in this paper. We have 
  $$\E [F(z_1) F(z_2)]
  = - \frac{1 - e^{2 i \pi (z_1 - z_2) 
  \sgn \Im (z_1 - z_2) }}{ (z_1 -z_2)^2}
  \mathds{1}_{\Im (z_1) \Im (z_2) < 0},$$

  and then 

$$\E [(L F(Lz_1))(L F(Lz_2))] =
  \underset{L \rightarrow \infty}
  {\longrightarrow} - \frac{\mathds{1}_{\Im(z_1) 
 \Im(z_2) < 0}}{(z_2 - z_1)^2} 
 = \E [G(z_1) G(z_2)].$$
 Similarly, 
 $$ \E [(L F(Lz_1))(\overline{L F(Lz_2)})]  \underset{L \rightarrow \infty}
  {\longrightarrow} \E [G(z_1) \overline{G(z_2)}].$$
  This convergence is naturally expected once the previous proposition is proven. 
\end{rmk} 
The stochastic process $z \mapsto X_{f_z}$ admits the version 
$$z \mapsto F(z) = \frac{\xi_{\infty}'(z)}{\xi_{\infty}(z)} - 2 i \pi \mathds{1}_{\Im z < 0},$$
which is holomorphic on $\C \backslash \R$. One can ask if the situation is similar for $G$. The answer is positive: 
\begin{proposition}
The random function $G$ admits a version which is holomorphic on $\C \backslash \R$. Moreover, 
$z \mapsto L F (Lz)$ converges in law to an holomorphic version of $G$ when $L$ goes to infinity, in the sense of the uniform convergence on compact sets of $\C \backslash \R$. 
\end{proposition}
\begin{proof}
We first compute the $L^2$ norm
of $G(z_1) - G(z_2)$ when $z_1, z_2 \notin 
\R$: 
$$\E[|G(z_1) - G(z_2)|^2]  = \E[|G(z_1)|^2] + \E[|G(z_2)|^2]
- \E[ G(z_1) \overline{G(z_2)}] 
- \E[ G(z_2) \overline{G(z_1)}] 
$$ $$ = - \frac{1}{(z_1 - \overline{z_1})^2}
- \frac{1}{(z_2 - \overline{z_2})^2}
+ \mathds{1}_{\Im(z_1) \Im(z_2) > 0}
\left( \frac{1}{(z_1 - \overline{z_2})^2}
+ \frac{1}{(z_2 - \overline{z_1})^2} \right). $$
Let us now assume that $z_1$ and $z_2$ 
are in a given compact set $K$ of 
$\C \backslash \R$. Let 
us denote: 
$$c_K := \inf \{ |\Im(z)|, z \in K\} > 0.$$
If $z_1, z_2 \in K$ have imaginary parts 
of different signs,
necessarily $|z_1 - z_2| \geq 2 c_K$
and from the computations  above, 
$$ \E[|G(z_1) - G(z_2)|^2] 
= \frac{1}{4 \Im ^2(z_1)} 
+ \frac{1}{4 \Im ^2(z_2)} 
\leq \frac{1}{2 c_K^2}.$$
One deduces 
$$\E[|G(z_1) - G(z_2)|^2] \leq \frac{1}{8 c_K^4}
|z_1 - z_2|^2.$$
If $z_1, z_2 \in K$ have imaginary parts 
with the same sign, 
$$
\E[|G(z_1) - G(z_2)|^2] = 
A(z_1, \overline{z_1}) + A(z_2, \overline{z_2}) - A(z_1, \overline{z_2})
- A(z_2, \overline{z_1})),$$
where 
$$A(u,v) := - \frac{1}{(u-v)^2}.$$
The function $A$ of two variables is holomorphic in the open set of $(a,b)
\in \C^2$ such that $\Im(a) \Im(z_1) > 0$ and 
$\Im(b) \Im (z_1) < 0$. 
Since the set $[z_1, z_2] \times [\overline{z_1}, \overline{z_2}]$ is 
included in this domain (recall that 
$\Im(z_1)$ and $\Im(z_2)$ have the same sign), we have
$$ A(z_1, \overline{z_1}) + A(z_2, \overline{z_2}) - A(z_1, \overline{z_2})
- A(z_2, \overline{z_1}))
= \int_{z_1}^{z_2} \int_{\overline{z_1}}^{\overline{z_2}} A''_{1,2} (u,v) du 
dv,$$
where $A''_{1,2}$ is the second derivative of $A$ with respect to the two variables.
Hence, 
$$\E[|G(z_1) - G(z_2)|^2]
= 6 \int_{z_1}^{z_2}
 \int_{\overline{z_1}}^{\overline{z_2}}
 \frac{du dv}{(u-v)^4}.$$
 Now, for $u \in [z_1,z_2]$, $v \in [
 \overline{z_1}, \overline{z_2}]$, we
 have $|\Im (u) - \Im (v)| \geq 2 
 c_K$, since $z_1, z_2 \in K$. Hence, 
 $|u-v|^4 \geq 16 c_K^4$, and 
 $$\E[|G(z_1) - G(z_2)|^2]
 \leq \frac{3}{8 c_K^4} 
 \int_{z_1}^{z_2}
 \int_{\overline{z_1}}^{\overline{z_2}}
 |du| |dv|.$$
 Hence, similarly as in the case 
 $\Im(z_1) \Im(z_2) < 0$, we have
 $$\E[|G(z_1) - G(z_2)|^2] 
 \leq \frac{3}{8 c_K^4} |z_1 - z_2|^2.$$
 
By Kolmogorov's criterion, $G$ admits a continuous version on $\C \backslash \R$. 
We now assume that $G$ itself is continuous.

Let $\Gamma : [0,1] \mapsto \C$ be a closed, piecewise smooth 
contour in $\C \backslash \R$. Since $G$ is continuous, the integral of $G$ along $\Gamma$ is well-defined, and one has
$$ \left| \int_{\Gamma} G(z) dz \right|^2
= \int_{0}^1 \int_0^1 G(\Gamma(t))
\overline{G (\Gamma(u))}  \Gamma'(t) 
\overline{\Gamma'(u)} dt du.$$
If we denote $\overline{\Gamma}$ the contour given by $\overline{\Gamma}(t)
= \overline{\Gamma(t)}$, we can write 
$$ \left| \int_{\Gamma} G(z) dz \right|^2
= \int_{0}^1 \int_0^1
G(\Gamma(t)) \widetilde{G}(\overline{\Gamma}(u)) \Gamma'(t) \overline{\Gamma}'(u) dt du,$$
where $\widetilde{G}$ is the function from 
$\C \backslash \R$, given by 
$$\widetilde{G} (z) = \overline{G(\overline{z})}.$$
Hence, 
$$\left| \int_{\Gamma} G(z) dz \right|^2
= \int_{\Gamma} \int_{\overline{\Gamma}}
G(z_1) \widetilde{G}(z_2) dz_1 dz_2.$$
Now, for $z_1 \in \Gamma$, $z_2 \in 
\overline{\Gamma}$
$$\mathbb{E} [ |G(z_1)| |\widetilde{G}(z_2)| ]
 \leq \left(\mathbb{E} [|G(z_1)|^2]\right)^{1/2}\left(\mathbb{E} [|\overline{G(z_2)}|^2]\right)^{1/2}
= \frac{1}{4 |\Im(z_1)| |\Im(z_2)|},$$
which implies 
$$\int_{\Gamma} \int_{\overline{\Gamma}}
\mathbb{E} [|G(z_1) \widetilde{G}(z_2)|] |dz_1| \, | dz_2| \leq 
\frac{(\ell(\Gamma))^2}{4 c_{\Gamma}^2} <
\infty,$$
 where $\ell(\Gamma)$ is the length of 
 $\Gamma$ and $c_{\Gamma}$ the infimum of $|\Im(z)|$ for $z \in \Gamma$.
 This bound allows to write
 $$\mathbb{E} \left[\left| \int_{\Gamma} G(z) dz \right|^2 \right]
= \int_{\Gamma} \int_{\overline{\Gamma}}
\mathbb{E} [G(z_1) \widetilde{G}(z_2)] dz_1 dz_2.$$
Now, for $z_1 \in \Gamma$ and 
$z_2 \in \overline{\Gamma}$, $\Im (z_1)$ and $\Im (\overline{z_2})$ have the same sign, 
which implies 
$$  \mathbb{E} [G(z_1) \widetilde{G}(z_2)]
=  \mathbb{E} [G(z_1) \overline{G(\overline{z_2})}] 
= - \frac{1}{(z_2 - z_1)^2},$$
and then 
$$\mathbb{E} \left[\left| \int_{\Gamma} G(z) dz \right|^2 \right] =  - 
\int_{\Gamma} \int_{\overline{\Gamma}}
\frac{dz_1 dz_2}{(z_2  - z_1)^2},$$
which is equal to zero, since the function
$(z_1, z_2) \mapsto 1/(z_2 - z_1)^2$
is holomorphic and the contours $\Gamma$ 
and $\overline{\Gamma}$ are closed. 
Hence, for all closed, piecewise smooth contours $\Gamma$ on $\C \backslash
\R$, one has almost surely 
$$\int_{\Gamma} G(z) dz = 0.$$
One deduces that almost surely, this equality holds simultaneously for all polygonal closed contours whose vertices have rational real and imaginary parts. 
Then, by continity of $G$, one can remove the condition of rationality, and deduces that almost surely, $G$ is
holomorphic on 
$\C \backslash \R$.

We know $z \mapsto L F (Lz)$ converges in law to $G$ in the sense of the finite-dimensional marginals: it remains to prove 
that this convergence occurs in the space of continuous functions, i.e. that the family 
of laws of $(L F(Lz))_{z \in \C}$ is tight in this space. 
For a compact set $K$ of $\C \backslash \R$, and for $z_1, z_2 \in K$, one has 
$$\E [| L F(Lz_1) -  L F(Lz_2)|^2] 
= \frac{1 - e^{-4 L \pi |\Im(z_1)|}}{4 \Im^2(z_1)} + 
 \frac{1 - e^{-4 L \pi |\Im(z_2)|}}{4 \Im^2(z_2)}
$$
if $\Im (z_1) \Im(z_2) < 0$, and 
\begin{align*}
\E [| L F(Lz_1) -  L F(Lz_2)|^2] 
& = A_L(z_1, \overline{z_1}) +
A_L(z_2, \overline{z_2})  - 
A_L(z_1, \overline{z_2})
- A_L(z_2, \overline{z_1})
\\ & = \int_{z_1}^{z_2} \int_{\overline{z_1}}^
{\overline{z_2}} (A''_L)_{1,2} (u,v)
du dv
\end{align*}
if $\Im (z_1) \Im(z_2) > 0$, for 
$$A_L(u,v) = - \frac{
1 - e^{2 i \pi L(u-v) \operatorname{sgn} 
\Im (u-v)}}{ (u-v)^2}.$$
In the first case, we get
$$\E [| L F(Lz_1) -  L F(Lz_2)|^2]
\leq \frac{|z_1 - z_2|^2}{8 c_K^4}$$ and 
in the second case, 
$$\E [| L F(Lz_1) -  L F(Lz_2)|^2]
\leq |z_2-z_1|^2 \sup_{|\Im(u)|, |\Im(v)|
> c_K} |(A''_L)_{1,2}(u,v)|.$$
Note that $A_L$ is holomorphic in
$\{(u,v) \in \C^2, \Im(u) \Im(v) < 0\}$, since $\sgn \Im( u-v)$ is locally constant on this set. 
Now, 
$$(A_{L}')_{1} (u,v) 
= \frac{2 (1 - e^{2 i \pi L(u-v) 
\sgn \Im(u-v)})}{(u-v)^3}
+ \frac{2 i \pi L \sgn \Im(u-v)
e^{2 i \pi L(u-v) 
\sgn \Im(u-v)}}{(u-v)^2},$$
\begin{align*}(A_{L}'')_{1,2} (u,v) 
&= \frac{6 (1 - e^{2 i \pi L(u-v) 
\sgn \Im(u-v)})}{(u-v)^4}
+  \frac{8 i \pi L \sgn \Im(u-v) e^{2 i \pi L(u-v) 
\sgn \Im(u-v)}}{(u-v)^3}
\\ & + \frac{4 \pi^2 L^2 e^{2 i \pi L(u-v) 
\sgn \Im(u-v)}}{(u-v)^2},
\end{align*}
\begin{align*}|(A_{L}'')_{1,2} (u,v)|
& \leq \frac{6 (1 + e^{- 2 \pi L |\Im (u-v)|)}}{|u-v|^4}
+  \frac{8 \pi L  e^{- 2 \pi L |\Im (u-v)|}}
{|u-v|^3} +  \frac{4 \pi^2 L^2 
e^{- 2 \pi L |\Im (u-v)|}}{|u-v|^2}.
\\ & \leq \frac{12}{|\Im(u-v)|^4}
+ \frac{8 \pi L e^{- 2 \pi L |\Im (u-v)|}
}{|\Im(u-v)|^3}
+ \frac{4 \pi^2 L^2 
e^{- 2 \pi L |\Im (u-v)|}}{|\Im(u-v)|^2}
\\ & 
\leq \frac{4 \pi^2 }{|\Im(u-v)|^4}
\left(1 + (L |\Im(u-v)| + L^2 (\Im(u-v))^2)
e^{- 2 \pi L |\Im (u-v)|} \right)
\\ & 
\leq \frac{\pi^2}{4 c^4_K}
\left( 1 + \sup_{x \geq 0} (x+ x^2)e^{-2
\pi x} \right). 
\end{align*}
Hence, 
$$\sup_{L > 0} \E [| L F(Lz_1) -  L F(Lz_2)|^2]
\leq \tilde{c}_K |z_2-z_1|^2,$$
where $\tilde{c}_K > 0$ depends only on $K$.  By Kolmogorov's criterion, the 
laws of $(LF(Lz))_{z \in \C \backslash \R}$ form a tight family for the uniform convergence on compact sets 
of $\C \backslash \R$. 

\end{proof}

\pagebreak

% --------------------------------------------------------------------
\bibliographystyle{halpha}
\bibliography{Bib_CueCharPolyCV}

\begin{thebibliography}{BHNY08}

\bibitem[AW13]{bib:AiWa}
M.~Aizenman and S.~Warzel.
\newblock {On the ubiquity of the Cauchy distribution in spectral problems}.
\newblock Arxiv, http://arxiv.org/pdf/1312.7769.pdf, 2013.

\bibitem[BG06]{bib:BG}
D.~Bump and A.~Gamburd.
\newblock {On the averages of characteristic polynomials from classical
  groups}.
\newblock {\em Comm. Math. Phys.}, 265(1):227--274, 2006.

\bibitem[BHNN13]{bib:BHNN}
Y.~Barhoumi, C.-P. Hughes, J.~Najnudel, and A.~Nikeghbali.
\newblock {On the number of zeros of linear combinations of indepepndent
  characteristic polynomials of random unitary matrices}.
\newblock Arxiv, http://arxiv.org/pdf/1301.5144.pdf, 2013.

\bibitem[BHNY08]{bib:BHNY}
P.~Bourgade, C.-P. Hughes, A.~Nikeghbali, and M.~Yor.
\newblock {The characteristic polynomial of a random unitary matrix: a
  probabilistic approach}.
\newblock {\em Duke Math. J.}, 145(1):45--69, 2008.

\bibitem[BNN12]{bib:BNN}
P.~Bourgade, J.~Najnudel, and A.~Nikeghbali.
\newblock { A unitary extension of virtual permutations}.
\newblock {\em IMRN}, 2013(18):4101--4134, 2012.

\bibitem[BOS06]{bib:BOS}
Alexei Borodin, Grigori Olshanski, and Eugene Strahov.
\newblock Giambelli compatible point processes.
\newblock {\em Adv. in Appl. Math.}, 37(2):209--248, 2006.

\bibitem[Bou10]{bib:Bou10}
P.~Bourgade.
\newblock Mesoscopic fluctuations of the zeta zeros.
\newblock {\em Probab. Theory Related Fields}, 148(3-4):479--500, 2010.

\bibitem[BS06]{bib:BoStr}
A.~Borodin and E.~Strahov.
\newblock {Averages of charactersitic polynomials in random matrix theory}.
\newblock {\em Comm. Pure. Appl. Math.}, 59(2):161--253, 2006.

\bibitem[CFZ08]{bib:CFZ08}
B.~Conrey, D.-W. Farmer, and M.-R. Zirnbauer.
\newblock {Autocorrelation of ratios of {$L$}-functions}.
\newblock {\em Commun. Number Theory Phys.}, 2(3):593--636, 2008.

\bibitem[CL95]{bib:CL}
O.~Costin and J.~Lebowitz.
\newblock {Gaussian fluctuations in random matrices}.
\newblock {\em Phys. Rev. Lett.}, 75(1):69--72, 1995.

\bibitem[CS07]{bib:CoSn}
B.~Conrey and N.~Snaith.
\newblock {Applications of the L-functions ratios conjectures}.
\newblock {\em Proc. Lond. Math. Soc.}, 94(3):594--646, 2007.

\bibitem[CS08]{bib:CoSn08}
B.~Conrey and N.~Snaith.
\newblock Correlation of eigenvalues and riemann zeros.
\newblock {\em Commun. Number Theory Phys.}, 2(3):477--536, 2008.

\bibitem[CS14]{bib:CoSn14}
B.~Conrey and N.~Snaith.
\newblock In support of n-correlation.
\newblock {\em Comm. Math. Phys.}, 330(2):639--653, 2014.

\bibitem[FGLL13]{bib:FGLL}
D.-W. Farmer, S.~M. Gonek, Y.~Lee, and S.~J. Lester.
\newblock Mean values of $\zeta^{'}/\zeta(s)$, correlations of zeros and the
  distribution of almost primes.
\newblock {\em Q. J. Math.}, 64(4):1057--1089, 2013.

\bibitem[FS03]{bib:FyStr}
Y.-V. Fyodorov and E.~Strahov.
\newblock An exact formula for general spectral correlation function of random
  {H}ermitian matrices.
\newblock {\em J. Phys. A: Math. Gen.}, 36:3203--3214, 2003.

\bibitem[GGM01]{bib:GGM01}
D.~A. Goldston, S.~M. Gonek, and H.~L. Montgomery.
\newblock Mean values of the logarithmic derivative of the {R}iemann
  zeta-function with applications to primes in short intervals.
\newblock {\em J. Reine Angew. Math.}, 537:105--126, 2001.

\bibitem[HKO01]{bib:HKO}
C.-P. Hughes, J.-P. Keating, and N.~O'Connell.
\newblock {On the Characteristic Polynomial of a Random Unitary Matrix}.
\newblock {\em Commun. Math. Physics}, 220:429--451, 2001.

\bibitem[Hug01]{bib:HughesPhD}
C.-P. Hughes.
\newblock {On the Characteristic Polynomial of a Random Unitary Matrix and the
  Riemann Zeta Function}.
\newblock {\em PhD Thesis}, 2001.

\bibitem[KS00]{bib:KS}
J.-P. Keating and N.~Snaith.
\newblock {Random Matrix Theory and $\zeta(1/2 + it)$ }.
\newblock {\em Commun. Math. Physics}, 214:57--89, 2000.

\bibitem[MM13]{bib:meckes}
E.-S. Meckes and M.-W. Meckes.
\newblock {Spectral measures of powers of random matrices}.
\newblock {\em Electron. Commun. Probab.}, 18:no. 78, 13, 2013.

\bibitem[MNN13]{bib:MNN}
K.~Maples, J.~Najnudel, and A.~Nikeghbali.
\newblock {Limit operators for circular ensembles}.
\newblock {\em arxiv}, 2013, arXiv:1304.3757.

\bibitem[Mon73]{bib:Montgomery}
H.~L. Montgomery.
\newblock The pair correlation of zeros of the zeta function.
\newblock In {\em Analytic number theory ({P}roc. {S}ympos. {P}ure {M}ath.,
  {V}ol. {XXIV}, {S}t. {L}ouis {U}niv., {S}t. {L}ouis, {M}o., 1972)}, pages
  181--193. Amer. Math. Soc., Providence, R.I., 1973.

\bibitem[Rod15]{Rod15}
B.~Rodgers.
\newblock Tail bounds for counts of zeros and eigenvalues, and an application
  to ratios.
\newblock Arxiv, http://arxiv.org/pdf/1502.05658.pdf, 2015.

\bibitem[RS96]{bib:RudSar}
Ze{\'e}v Rudnick and Peter Sarnak.
\newblock Zeros of principal {$L$}-functions and random matrix theory.
\newblock {\em Duke Math. J.}, 81(2):269--322, 1996.
\newblock A celebration of John F. Nash, Jr.

\bibitem[SF03]{bib:StrFy}
E.~Strahov. and Y.-V. Fyodorov.
\newblock On universality of correlation functions of characteristic
  polynomials: {R}iemann-{H}ilbert {A}pproach.
\newblock {\em Commun. Math. Phys.}, 241:343--382, 2003.

\bibitem[Sos00]{bib:soshnikov00}
Alexander Soshnikov.
\newblock The central limit theorem for local linear statistics in classical
  compact groups and related combinatorial identities.
\newblock {\em Ann. Probab.}, 28(3):1353--1370, 2000.

\bibitem[Sos02]{bib:soshnikov02}
A.~Soshnikov.
\newblock {Gaussian limit for determinantal random point fields}.
\newblock {\em Ann. Probab.}, 30(1):171--187, 2002.

\bibitem[Tit86]{bib:Titchmarsh}
E.-C. Titchmarsh.
\newblock {\em {The theory of the {R}iemann zeta-function}}.
\newblock The Clarendon Press, Oxford University Press, New York, second
  edition, 1986.
\newblock Edited and with a preface by D. R. Heath-Brown.

\end{thebibliography}

\end{document}